\documentclass[12pt,reqno]{amsart}
\usepackage{amsmath,amsthm,amsfonts,amssymb,amscd,amstext}
\usepackage[ansinew]{inputenc}
\usepackage{graphicx}
\usepackage{euscript,color}
\usepackage{a4wide}

\usepackage[backref]{hyperref}

\usepackage{pxfonts}   %Nice fonts! The pxfonts package must be loaded AFTER amssymb, amsmath etc
\numberwithin{equation}{section}

\newcommand{\qand}{\quad\text{and}\quad}

\theoremstyle{plain}
\newtheorem{maintheorem}{Theorem}

\newtheorem{maincorollary}[maintheorem]{Corollary}

\newtheorem{theorem}{Theorem}[section]
\newtheorem{proposition}[theorem]{Proposition}
\newtheorem{corollary}[theorem]{Corollary}
\newtheorem{lemma}[theorem]{Lemma}
\newtheorem{claim}[theorem]{Claim}
\newtheorem{remark}[theorem]{Remark}
\newtheorem{question}[theorem]{Question}
\newtheorem{example}[theorem]{Example}
\theoremstyle{definition}
\newtheorem{definition}{Definition}

\newcommand{\RR}{{\mathbb R}}

\newcommand{\NN}{{\mathbb N}}

\newcommand{\sS}{{\mathbb S}}

\newcommand{\SB}{{\mathcal B}}

\newcommand{\SD}{{\mathcal D}}
\newcommand{\SE}{{\mathcal E}}

\newcommand{\SK}{{\mathcal K}}

\newcommand{\SM}{{\mathcal M}}
\newcommand{\SN}{{\mathcal N}}

\newcommand{\SP}{{\mathcal P}}
\newcommand{\SQ}{{\mathcal Q}}

\newcommand{\SU}{{\mathcal U}}
\newcommand{\SV}{{\mathcal V}}
\newcommand{\SW}{{\mathcal W}}

\newcommand{\diam}{\operatorname{diam}}
\renewcommand{\epsilon}{\varepsilon}
\newcommand{\dist}{\operatorname{dist}}

\newcommand{\interior}{\operatorname{int}}

\newcommand{\esup}{\operatorname{ess\,sup}}

\newcommand{\Leb}{\operatorname{Leb}}

\newcommand{\supp}{\operatorname{supp}}

\title[Pesin's formula for $C^1$ maps]
{Pesin's Entropy Formula for $C^1$ non-uniformly expanding maps}

\date{January 10, 2019}

\author{Vítor Araujo}

 \address[V.A.]{ Instituto de Matem\'atica,
  Universidade Federal da Bahia Av. Ademar de Barros s/n,
  40170-110 Salvador, Brazil.}

\email{vitor.d.araujo@ufba.br}
\urladdr{www.sd.mat.ufba.br/$\sim$vitor.d.araujo}

\author{Felipe Santos}

\address[F.S.]{ Centro de Formação de Professores,
  Universidade Federal do Reconcavo da Bahia, Avenida Nestor
  de Mello Pita 535, 45300-000, Amargosa, BA-Brazil}

\email{felipefonsantos@gmail.com}

\thanks{V.A. was partially supported by CNPq and
  FAPESB-Brazil; and F.S. was partially supported by
  CAPES-Brazil.}

\begin{document}

\keywords{non-uniform expansion, SRB/physical-like measures,
  equilibrium states, Pesin's entropy formula, $C^1$ smooth
  maps, uniform expansion}

\subjclass[2010]
{Primary: 37D25; Secondary: 37D35, 37D20, 37C40.}

\begin{abstract}
  We prove existence of equilibrium states with special
  properties for a class of distance expanding local
  homeomorphisms on compact metric spaces and continuous
  potentials. Moreover, we formulate a C$^1$ generalization
  of Pesin's Entropy Formula: all ergodic weak-SRB-like
  measures satisfy Pesin's Entropy Formula for $C^1$
  non-uniformly expanding maps. We show that for
  weak-expanding maps such that $\Leb$-a.e $x$ has positive
  frequency of hyperbolic times, then all the necessarily
  existing ergodic weak-SRB-like measures satisfy Pesin's
  Entropy Formula and are equilibrium states for the
  potential $\psi=-\log|\det Df|$. In particular, this holds
  for any $C^1$-expanding map and, in this case, the set of
  invariant probability measures that satisfy Pesin's
  Entropy Formula is the weak$^*$-closed convex hull of the
  ergodic weak-SRB-like
  measures. \href{https://rdcu.be/bhnWQ}{Online paper.}
\end{abstract}

\maketitle
\tableofcontents

\section{Introduction}
\label{sec:introd}

%In the 1960s and 1970s a more or less complete description of measure theoretic  entropy, topological pressure and equilibrium states was developed.

Equilibrium states, a concept originating from statistical
mechanics, are special classes of probability measures on
compact metric spaces $X$ that are characterized by a
Variational Principle. In the classical setting, given a
continuous map $T:X\rightarrow X$ on a compact metric space
$X$ and a continuous function $\phi: X\rightarrow \RR$, then
a $T$-invariant probability measure $\mu$ is called
$\phi$-equilibrium state (or equilibrium state for the
potential $\phi$) if
\begin{eqnarray*}
P_{\textrm{top}}(T,\phi)=h_{\mu}(T)+\int\phi\,d\mu, \ \ \textrm{where}\ \ P_{\textrm{top}}(T,\phi)=\sup\limits_{\lambda\in \SM_T}\left\{h_{\lambda}(T)+\int\phi\,d\lambda\right\},
\end{eqnarray*}
and $\SM_T$ is the set of all $T$-invariant probability
measures. The quantity $P_{\textrm{top}}(T,\phi)$ is called
the \emph{topological pressure} and the identity on the
right hand side above is a consequence of the variational
principle, see e.g. \cite{walters2000introduction,PrzyUrb10}
for definitions of entropy $h_{\mu}(T)$ and topological
pressure $P_{\textrm{top}}(T,\phi)$.

Depending on the dynamical system, these measures can have
additional properties. Sinai, Ruelle, Bowen
\cite{Si72,BR75,Bo75,Ru76} were the forerunners of the
theory of equilibrium states of (uniformly hyperbolic)
smooth dynamical systems. They established an important
connection between equilibrium states and the soon to be
called physical/SRB measures. This kind of measures provides
asymptotic information on a set of trajectories that one
expects is large enough to be observable from the point of
view of the Lebesgue measure (or some other relevant
reference measure), that is, a measure $\mu$ is
\emph{physical} or $SRB$ if its \emph{ergodic basin}
\[
B(\mu)=\left\{x\in X:
  \frac{1}{n}\sum_{j=0}^{n-1}\phi(T^jx)\xrightarrow[n\rightarrow+\infty]{}\int\phi
  d\mu, \ \forall \phi\in C^0(X,\RR)\right\}
\]
has positive volume (Lebesgue measure) or other relevant
distinguished measure.

Several difficulties arise when trying to extend
this theory. Despite substantial progress by several
authors, a global picture is still far from complete. For
example, the basic strategy used by Sinai-Ruelle-Bowen was
to (semi)conjugate the dynamics to a subshift of finite
type, via a Markov partition.  However, existence of
generating Markov partitions is known only in few cases and,
often, such partitions can not be finite. Moreover, there
exist transformations and functions admitting no equilibrium
states or SRB measure.

In many situations with some kind of expansion it is
possible to ensure that equilibrium states always
exist. However, equilibrium states are generally not unique
and the way they are obtained does not provide additional
information for the study of the dynamics. In the setting of
uniformly expanding maps, equilibrium states always exist
and they are unique SRB measures if the potential is
H\"older continuous and the dynamics is transitive.

Some natural questions arise: when does a system have an
equilibrium state?  Second, what statistical properties do
those probabilities exhibit with respect to the
(non-invariant) reference measure? Third, when is the
equilibrium state unique?  The existence of equilibrium
states is a relatively soft property that can often be
established via compacteness arguments. Statistical
properties and uniqueness of equilibrium state are usually
more subtle and require a better understanding of the
dynamics. In our setting, we do not expect uniqueness of
equilibrium states because we consider dynamical systems
with low regularity (continuous or of class $C^1$ only) and
only continuous potentials.

From the Thermodynamical Formalism, the answer for the first
question is known to be affirmative for  distance
expanding maps in a compact metric space $X$ and all
continuous potentials (see
\cite{ruelle2004thermodynamic,PrzyUrb10}). Moreover,
inspired by the definition of SRB-like measure given in
\cite{CatsEnrich2011}, which always exists for all continuous
transformations on a compact metric space, we have an answer
to the second question. That is, given a reference measure
$\nu$, there always exist weak-$\nu$-SRB-like and
$\nu$-SRB-like measures (generalizations of the notion of
SRB measure, see the statement of the resuts for more
details).

Our first result shows that, in our context, the
$\nu$-SRB-like measures can be seen as measures that
naturally arise as accumulation points of $\nu_n$-SRB
measures. In addition, for topologically exact distance
expanding maps, the limit measure is a $\nu$-SRB-like
probability measure with full generalized basin.

Moreover, we show that if $T:X\rightarrow X$ is an open
distance expanding topologically transitive map in a compact
metric space $X$ and $\phi:X\rightarrow\RR$ is continuous,
then for each (necessarily existing) conformal measure
$\nu_{\phi}$ all the (necessarily existing)
$\nu_{\phi}$-SRB-like measures are equilibrium states for
the potential $\phi$.

Let now $M$ be a compact boundaryless finite dimensional
Riemannian manifold. In the 1970's, Pesin showed how two
important concepts are related: Lyapunov exponents and
measure-theoretic entropy in the smooth ergodic theory of
dynamical systems. In \cite{Pe77}, Pesin showed that if
$\mu$ is an $f$-invariant measure of a $C^2$ (or
$C^{1+\alpha}$, $\alpha>0$) diffeomorphism of a compact
manifold which is absolutely continuous with respect to the
Lebesgue (volume) reference measure of the manifold, then
$$h_{\mu}(f)=\int\Sigma^+d\mu,$$
where $\Sigma^+$ denotes the sum of the positive Lyapunov
exponents at a regular point, counting multiplicities.

Ledrappier and Young in \cite{LY85} characterized the
measures which satisfy Pesin's Entropy Formula for $C^2$
diffeomorphisms.  Liu, Qian and Zhu \cite{Li98,QZ2002}
generalized Pesin's Entropy Formula for $C^2$ endomorphisms.

There are extensive results concerning Pesin’s entropy
formula, but the vast majority of these results were
obtained under the assumption that the dynamics is at least
$C^{1+\alpha}$ regular.  In fact, there is still a gap
between $C^{1+\alpha}$ and $C^1$ dynamics, despite recent
progress in this direction
\cite{tahzibi2002c1,BesVar2010,CatsEnrich2012,catsigeras2015pesin,catsigeras2016weak,Sun20121421,qiu2011existence,CAO20163964}.

We formulate a $C^1$ generalization of Pesin's Entropy
Formula, for non-uniformly expanding local diffeomorphisms
$f$, obtaining a sufficient condition to ensure that an
$f$-invariant probability measure satisfies Pesin's Entropy
Formula.

% All the necessarily existing SRB-like measures are
% equilibrium states for the potential
% $\psi=-\log|\det Df|$. In particular, we extend the result
% obtained in \cite{CatsEnrich2012}, that is, for each
% C$^1$-expanding map $f$ on some compact Riemannian
% manifold $M$ of finite dimension any SRB-like measure
% satisfies Pesin’s entropy formula.

Moreover, we study the existence of ergodic weak-SRB-like
measures and some of their properties for dynamics with some
expansion: either non uniformly expanding $C^1$
transformations, with hyperbolic times for Lebesgue almost
every point, or maps which are expanding except at a finite
subset of the ambient space. It is known that not all
dynamics admit ergodic SRB-like measures: see \cite[Example
5.4]{CatsEnrich2011}. Moreover, in \cite{catsigeras2016weak}
Catsigeras, Cerminara, and Enrich show that ergodic SRB-like
measures do exist for $C^1$ Anosov diffeomorfism and, more
recently in \cite{CatsTroub2017}, Catsigeras and Troubetzkoy
show that for $C^0$-generic continuous interval dynamics all
ergodic measures are SRB-like.

\subsection{Statement of results}
\label{sec:statem-results}

Let $T:X\rightarrow X$ be a continuous transformation
defined on a compact metric space $(X, d)$. We present first
some preliminary definitions needed to state the main
results.

\subsubsection{Topological pressure}
\label{sec:topological-pressure}

The \emph{dynamical ball}  centered at $x\in X$, radius
$\delta >0$, and length $n\geq1$ is defined by
%\begin{eqnarray*}\label{def-bola-dinamica}
$B(x,n,\delta)=\{y\in X: d(T^jx,T^jy)\leq\delta, \ 0\leq j\leq n-1\}$.
%\end{eqnarray*}

Let $\nu$ be a Borel probability measure on $X$. We define
\begin{eqnarray}\label{def-de-entropia-via-supremo-essencial}
  h_{\nu}(T,x)=\lim_{\delta\to
  0}\limsup_{n\to+\infty}-\frac{1}{n}\log\nu(B(x,n,\delta))
  &\textrm{and}&h_{\nu}(T,\mu)=\mu\text{-}\esup h_{\nu}(T,x).
\end{eqnarray}
We note that $\nu$ is not necessarily $T$-invariant. If
$\mu$ is $T$-invariant and ergodic, then we have
$h_{\mu}(T, x) = h_{\mu}(T)$ for $\mu$-a.e. $x\in X$, where
$h_{\mu}(T)$ is the metric entropy of $T$ with respect to
$\mu$.

Let $n$ be a natural number, $\epsilon > 0$ and let $K$ be a
compact subset of $X$. A subset $F$ of $X$ is said to
$(n, \varepsilon)$-span $K$ with respect to $T$ if, for each
$x\in K$, there exists $y\in F$ with
$d(T^jx, T^jy)\leq\varepsilon$ for all $0\leq j\leq n-1$,
that is, $K\subset \bigcup_{x\in F}B(x,n,\varepsilon).$

%\begin{definition}
%Let $n$ be a natural number, $\varepsilon > 0$ and let $K$ be a compact subset of $X$. A subset $F$ of $X$ is said to $(n, \varepsilon)$ span $K$ with respect to $T$ if $\forall x\in K$ there exists $y\in F$ with $d(T^jx, T^jy)\leq\varepsilon$ for all $0\leq j\leq n-1$, that is,
%$$K\subset \bigcup\limits_{x\in F}B(x,n,\varepsilon).$$
%\end{definition}

Given a compact subset $K$ of $X$, we set
\begin{eqnarray}\label{def-entropia-do-conj-K}
  h(T; K) = \lim_{\varepsilon\rightarrow 0}\liminf_{n\rightarrow+\infty}\frac{1}{n}\log N(n, \varepsilon, K),
\end{eqnarray}
where $N(n, \varepsilon, K)$ denotes the smallest
cardinality of any $(n,\varepsilon)$-spanning set for $K$
with respect to $T$.

\begin{definition}\label{def-de-entropia-top-via-conjs-geradores}
  The \emph{topological entropy} of $T$ is
  $h_{top}(T) = \sup_{K}h(T, K)$, where the supremum
  is taken over the collection of all compact subset of $X$.
\end{definition}

Let $\phi: X\to\RR$ be a real continuous function, usually
referred to as \emph{potential}.  Given an open cover
$\alpha$ for $X$ we define the pressure $P_{T}(\phi,\alpha)$
of $\phi$ with respect to $\alpha$ by
\begin{align*}
  P_{T}(\phi,\alpha)
  :=
  \lim_{n\to+\infty}\frac{1}{n}\log\inf_{\SU\subset \alpha^n}\left\{\sum_{U\in\SU}e^{S_n\phi(U)}\right\},
\end{align*}
where the infimum is taken over all subcovers $\SU$ of
$\alpha^n =\bigvee_{j=0}^{n-1}T^{-j}(\alpha)$ and we write
$S_n\phi(x):=\sum_{j=0}^{n-1}\phi(T^jx)$ and
$S_n\phi(U):=\sup_{x\in U}S_n\phi(x)$ in what follows.

\begin{definition}\label{definicao-de-pressao-top-via-cobertura}
  The \emph{topological pressure} $P_{top}(T,\phi)$ of the
  potential $\phi$ with respect to the dynamics $T$ is
  defined by
$$P_{top}(T,\phi)=\lim_{\delta\rightarrow0}\left\{\sup_{|\alpha|\leq\delta}P_T(\phi,\alpha)\right\}$$
where $|\alpha|$ denotes the diameter of the open cover
$\alpha$ and the supremum is taken over all open covers with
diameter at most $\delta$.
\end{definition}

For given $n>0$ and $\varepsilon>0$, a subset $E\subset X$
is called $(n,\varepsilon)$-separated if, for any given pair
$x, y\in E$, $x\neq y$, then there exists $0\leq j\leq n-1$
such that $d(T^jx, T^jy)>\epsilon$.  This provides an
alternative way of defining topological pressure, as follows.

\begin{definition}\label{definicao-de-pressao-top-via-conj-separados}
The topological pressure $P_{top}(T,\phi)$ of the potential $\phi$ with respect to the dynamics $T$ is defined by
$$P_{top}(T,\phi)=\lim_{\varepsilon\rightarrow0}\limsup_{n\rightarrow+\infty}\frac{1}{n}\log\sup\left\{\sum_{x\in E}e^{S_n\phi(x)}\right\}$$
where the supremum is taken over all maximal
$(n, \varepsilon)$-separated sets $E$.
\end{definition}
We refer the reader to \cite{walters2000introduction} for
more details and properties of the topological pressure.

\subsubsection{Distance expanding maps}
\label{sec:distance-expand-maps}

A continuous mapping $T:X\rightarrow X$ is said to be
\emph{distance expanding} (with respect to the metric $d$,
also known as ``Ruelle expanding'') if there exist constants
$\lambda> 1$, $\eta > 0$ and $n\geq 1$, such that for all
$x, y \in X$
\begin{eqnarray}\label{eq3.0.1}
\textrm{if} \ \ d(x,y)<2\eta, \ \ \textrm{then} \ \ d(T^n(x),T^n(y))\geq\lambda d(x,y).
\end{eqnarray}
In the sequel we will always assume (without loss of
generality, see chapter 3 in \cite{PrzyUrb10}) that $n=1$,
that is
\begin{eqnarray}\label{condicao-de-expansao}
d(x,y)<2\eta \ \ \Longrightarrow \ \ d(T(x),T(y))\geq\lambda d(x,y).
\end{eqnarray}
We refer the reader to \cite{Coven1980,ashley1992,PrzyUrb10}
for more details and properties of distance expanding map.

\subsubsection{Transfer operator}
\label{sec:transfer-operator}

We consider the Ruelle-Perron-Fr{\"o}benius transfer
operator $\mathcal{L}_{T,\phi}=\mathcal{L}_\phi$ associated
to $T:X\rightarrow X$ and the continuous function
(potential) $\phi:X\rightarrow \RR$ as the linear operator
defined on the space $C^0(X,\RR)$ of continuous functions
$g:X\rightarrow\RR$ by
$$\mathcal{L}_\phi(g)(x)=\sum_{T(y)=x}e^{\phi(y)}g(y).$$

The dual of the Ruelle-Perron-Fr{\"o}benius transfer
operator is given by
$$
\begin{array}{rclrccl}
  \mathcal{L}^*_\phi: & \mathcal{M} & \rightarrow & \mathcal{M} & & & \\
                      & \eta & \mapsto & \mathcal{L}^*_\phi \eta : & C^0(X,\RR) & \rightarrow & \mathbb{R}\\
                      &   &         &                      & \psi & \mapsto     & \displaystyle\int \mathcal{L}_\phi\psi \,d\eta.\\
\end{array}
$$
where $\SM$ is the family of Borel probability measures in
$X$.

\subsubsection{SRB and weak-SRB-like probability measures}
\label{sec:srb-srb-like}

For any point $x\in X$ we define
\begin{eqnarray}\label{11}
\sigma_n(x)=\frac{1}{n}\sum\limits_{j=0}^{n-1}\delta_{T^j(x)},
\end{eqnarray}
where $\delta_y$ is the Dirac delta probability measure
supported on $y\in X$. The sequence (\ref{11}) gives the
\emph{empiric probabilities} of the orbit of $x$. Let
$\SM_T$ be the (non-empty) set of $T$-invariant Borel
probability measures in $X$.

\begin{definition}%(The $p\omega-limit \ set$ of $x$.)
  For each point $x\in M$, we denote by
  $p\omega(x)\subset\mathcal{M}_T$ the limit set of the
  empirical sequence with initial state $x$ in the weak$^*$
  topology of $\SM$, that is,
  \begin{align*}
    p\omega(x)
    :=
    \left\{\mu\in\mathcal{M}_T:
    \exists n_j\xrightarrow[j\rightarrow+\infty]{}+\infty
    \;\textrm{such that} \;
    \sigma_{n_j}(x)\xrightarrow[j\rightarrow+\infty]{w^*}\mu\right\}.
  \end{align*}
\end{definition}

\begin{definition}\label{nu-SRB probability measure}
  Fixing an underlying reference measure $\nu$ on $X$, we
  say that $\mu\in \mathcal{M}_T$ is $\nu$-SRB (or
  $\nu$-physical) if $\nu(B(\mu))>0$, where
$$B(\mu)=\left\{x\in X; p\omega(x)=\{\mu\}\right\} \ \ \textrm{is the "ergodic basin" of} \ \mu. $$
\end{definition}
Let $\mu\in \mathcal{M}_T$ and $\varepsilon > 0$ be given
and consider the following measurable subsets of $X$
\begin{align}\label{eq3}
  A_{\epsilon,n}(\mu)
  :=
  \{x\in X:\dist(\sigma_n(x),\mu)<\epsilon\}
  \qand
  A_{\epsilon}(\mu)
  :=
  \{x\in X: \dist(p\omega(x),\mu)<\epsilon\}.
\end{align}
We say that $A_{\epsilon,n}(\mu)$ is the
$\epsilon$-\emph{pseudo basin of $\mu$ up to time} $n$ and
$A_{\epsilon}(\mu)$ is the basin of $\epsilon$-\emph{weak
  statistical attraction} of $\mu$.

\begin{definition}\label{def2} %(SRB and SRB-like (or weakly-physical) measures)
  Fix a reference probability measure $\nu$ for the space
  $X$. We say that a $T$-invariant probability measure $\mu$
  is
 \begin{enumerate}
 \item \emph{$\nu$-SRB-like} (or
   \emph{$\nu$-physical-like}), if
   $\nu(A_{\varepsilon}(\mu)) > 0$ for all
   $\varepsilon > 0$;
 \item \emph{$\nu$-weak-SRB-like} (or
   \emph{$\nu$-weak-physical-like}), if
   $\limsup_{n\rightarrow+\infty}\frac{1}{n}\log\nu(A_{\varepsilon,n}(\mu))=0,
   \ \forall\varepsilon>0.$
 \end{enumerate}
 When $\nu=\Leb$ we say that $\mu$ is simply SRB-like (or
 weak-SRB-like).
\end{definition}

\begin{remark}\label{rmk:toda-medida-SRB-like-eh-weak-SRB-like}
  It is easy to see that every $\nu$-SRB measure is also a
  $\nu$-SRB-like measure. Moreover, the $\nu$-SRB-like
  measures are a particular case of $\nu$-weak-SRB-like (see
  \cite[Theorem 1, item B]{catsigeras2016weak}).
\end{remark}

\begin{definition} Given a continuous map $T:X\rightarrow X$
  and a continuous function $\phi:M\rightarrow\mathbb{R}$,
  we say that a probability measure $\nu$ is conformal for
  $T$ with respect to $\phi$ (or $\phi$-conformal) if there
  exists $\lambda>0$ so that
  $\mathcal{L}^*_\phi\nu=\lambda \nu$.
\end{definition}

\subsubsection{SRB-like measures as limits of SRB measures}
\label{sec:srb-like-measures}

Our first result shows that we can see the $\nu$-SRB-like
measures as accumulation points of $\nu_n$-SRB measures.

\begin{maintheorem}
\label{mthm:theorem-conv-das-medidas-tipo-SRB}
Let $T:X\rightarrow X$ be an open distance expanding
topologically transitive map of a compact metric space $X$,
$(\phi_n)_{n\geq1}$ a sequence of H\"older continuous
potentials, $(\nu_n)_{n\geq1}$ a sequence of conformal
measures associated to the pair $(T,\phi_n)$ and
$(\mu_n)_{n\geq1}$ a sequence of $\nu_n$-SRB
measures. Assume that
\begin{enumerate}
\item $\phi_{n_j}\xrightarrow[j\rightarrow+\infty]{}\phi$
  (in the topology of uniform convergence);
\item $\nu_{n_j}\xrightarrow[j\rightarrow+\infty]{w^*}\nu$
  (in the weak$^*$ topology);
\item $\mu_{n_j}\xrightarrow[j\rightarrow+\infty]{w^*}\mu$
  (in the weak$^*$ topology).
 % \item For all $\varepsilon>0$, $\nu_n(A_{\varepsilon}(\mu_n))=1$
\end{enumerate}
Then $\nu$ is a conformal measure for $(T,\phi)$ and $\mu$
is $\nu$-SRB-like. % In particular, $\mu$ is an equilibrium
% state for the potential $\phi$.
Moreover, $\mu$ is an equilibrium state for the potential
$\phi$ and, if $T$ is topologically exact, then
$\nu(X\setminus A_{\epsilon}(\mu))=0$ for all $\epsilon>0$.
%Moreover, if $T$ is topologically exact then $\nu(X\setminus A_{\varepsilon}(\mu))=0$ for all $\varepsilon>0$.
\end{maintheorem}
This is one of the motivations for the study of SRB-like
measures as the natural extension of the notion of
physical/SRB measure for $C^1$ maps. Additional
justification is given by the results of this work.

We stress that H\"older continuous potentials are only used
in this work in the assumption of
Theorem~\ref{mthm:theorem-conv-das-medidas-tipo-SRB}.  Since
all continuous potentials can be approximated by Lipschitz
potentials we have, for topologically exact maps in the
setting of Theorem A, that there always exists some
$\phi$-conformal measure admitting a $\nu$-SRB measure with
full basin of $\epsilon$-weak statistical attraction, for
all $\epsilon>0$.

Next result extends, in particular, the main result obtained
in \cite[Theorem 2.3]{CatsEnrich2012} proved only for
expanding circle maps.

\begin{maintheorem}
  \label{mthm:Toda-medida-nu-SRB-like-e-estado-de-equilibrio}
  Let $T:X\rightarrow X$ be an open expanding topologically
  transitive map of a compact metric space $X$ and
  $\phi:X\rightarrow\RR$ a continuous potential. For each
  (necessarily existing) conformal measure $\nu_{\phi}$ all
  the (necessarily existing) $\nu_{\phi}$-SRB-like measures
  are equilibrium states for the potential $\phi$.
  \end{maintheorem}

  Next we explore more properties of $\nu$-SRB-like measures.

\begin{definition}\label{definition-of-SK(phi)}
  Let $T:X\rightarrow X$ be a continuous map and
  $\phi:X\rightarrow\RR$ a continuous potential. Given $r>0$
  we define the subset
  $ \SK_r(\phi)=\{\mu\in\SM_T: h_{\mu}(T)+\int\phi d\mu\geq
  P_{\textrm{top}}(T,\phi)-r\}.  $
\end{definition}

\begin{maincorollary}\label{mthm:pressao-nao-negativa-implica-caracterizacao-dos-estados-de-equilibrio}
  Let $T:X\rightarrow X$ be an open expanding topologically
  transitive map of a compact metric space $X$,
  $\phi:X\rightarrow\RR$ a continuous potential and $\nu$ a
  $\phi$-conformal measure. If $\mu$ is a $\phi$-equilibrium
  state such that $h_{\nu}(T,\mu)<\infty$, then every ergodic
  component $\mu_x$ of $\mu$ is a $\nu$-weak-SRB-like
  measure. Moreover, if $\mu$ is the unique
  $\phi$-equilibrium state, then $\mu$ is $\nu$-SRB,
  $\nu(B(\mu))=1$ and $\mu$ satisfies the following large
  deviation bound: for every weak$^*$ neighborhood
  $\mathcal{V}$ of $\mu$ we have
  \begin{align*}
    \limsup_{n\to+\infty}
    \frac1n\log\nu(\{x\in X: \ \sigma_n(x)\in
    \SM\setminus\SV\})\leq -I(\SV)
  \end{align*}
where $I(\SV)=\sup\{r>0: \SK_r(\phi)\subset\SV\}$.
\end{maincorollary}

\subsubsection{Non-uniformly expanding maps}
\label{sec:non-uniformly-expand}

We now extend Theorem
\ref{mthm:Toda-medida-nu-SRB-like-e-estado-de-equilibrio} to
weaker forms of expansion.

We denote by $\|\cdot\|$ a Riemannian norm on the compact
$m$-dimensional boundaryless manifold $M$, $m\geq1$; by
$d(\cdot,\cdot)$ the induced distance and by $\Leb$ a
Riemannian volume form, which we call \emph{Lebesgue
  measure} or \emph{volume} and assume to be normalized:
$\Leb(M)=1$. Note that $\Leb$ is not necessarily
$f$-invariant.

Recall that a $C^1$-map $f:M\rightarrow M$ is
\emph{uniformly expanding} or just \emph{expanding} if there
is some $\lambda > 1$ such for some choice of a metric in
$M$ one has
$$\|Df(x)v\|>\lambda\|v\|,  \ \textrm{for all} \ x\in M \ \textrm{and all} \ v\in T_xM.$$

In the following statements we always assume $f:M \to M$ to
be a $C^1$ local diffeomorphism and define
\begin{align*}
  \psi:=-\log Jf\qand Jf=|\det Df|.
\end{align*}

% Recall that a Borel set in a topological space is said to
% have total probability if it has probability one for every
% $f$-invariant probability measure.

% \begin{corollary}
%   Let $f:M\rightarrow M$ be $C^1$ local diffeomorphism and
%   $Y$ be a subset with total probability. If
% $$\liminf_{n\rightarrow+\infty}\frac{1}{n}\sum_{j=0}^{n-1}\log\|Df(f^j(x))^{-1}\|<0, \quad\forall x\in Y,$$
% then all the (necessarily existing) SRB-like measures are
% equilibrium states for the potential $\psi=-\log|\det Df|$.
% \end{corollary}

% \begin{proof}
%   For a $C^1$-local diffeomorphism $f:M\to M$ we have that
%   $\Leb$ is $\psi$-conformal. Moreover, we know by
%   \cite[Theorem 1.15]{Ze03} that if
% $$\liminf_{n\rightarrow+\infty}\frac{1}{n}\sum_{j=0}^{n-1}\log\|Df(f^j(x))^{-1}\|<0$$
% for $x$ in a total probability subset, then $f$ is uniformly
% expanding. Thus, applying Theorem
% \ref{mthm:Toda-medida-nu-SRB-like-e-estado-de-equilibrio} we
% finish the proof.
% \end{proof}

The \emph{Lyapunov exponents} of a $C^1$ local
diffeomorphism $f$ of a compact manifold $M$ are defined by
Oseledets Theorem which states that, for any $f$-invariant
probability measure $\mu$, for almost all points $x\in M$
there is $\kappa(x)\geq1$, a filtration
$T_xM=F_1(x)\supset F_2(x)\supset\cdots\supset
F_{\kappa(x)}(x)\supset F_{\kappa(x)+1}(x)=\{0\},$ and
numbers $\lambda_1(x)>\lambda_2(x)>\cdots > \lambda_k(x)$
such $Df(x)\cdot F_i(x)=F_i(f(x))$ and
$$\lim_{n\to\infty}\frac{1}{n}\log\|Df^n(x)v\|=\lambda_i(x)$$
for all $v\in F_i(x)\backslash F_{i+1}(x)$ and
$0\leq i\leq \kappa(x)$. The numbers $\lambda_i(x)$ are
called \emph{Lyapunov exponents} of $f$ at the point
$x$. For more details on Lyapunov exponents and non-uniform
hyperbolicity, see \cite{BaPe13}.

\begin{definition}
  Let $f:M \rightarrow M$ be a $C^1$ local
  diffeomorphism. We say that $\mu\in\SM_f$ satisfies
  Pesin's Entropy Formula if $h_{\mu}(f)=\int\Sigma^+d\mu,$
  where $\Sigma^{+}$ denotes the sum of the positive
  Lyapunov exponents at a regular point, counting
  multiplicities.
\end{definition}

Given $0<\sigma<1$ we define
\begin{eqnarray}\label{conj:NUE}
  H(\sigma)
  =
  \left\{x\in M:
  \limsup_{n\to+\infty}\frac1n
  \sum_{j=0}^{n-1}\log\|Df(f^j(x))^{-1}\|<\log\sigma
  \right\}.
\end{eqnarray}

\begin{definition}\label{definicao-de-NUE}
  We say that $f:M\rightarrow M$ is \emph{non-uniformly
    expanding} if there exists $\sigma\in(0,1)$ such that
  $\Leb(H(\sigma))=1$.
\end{definition}

A probability measure $\nu$ (not necessarily invariant) is
\emph{expanding} if there exists $\sigma\in(0,1)$ such that
$\nu(H(\sigma))=1$.

Next result relates non-uniform expansion and expanding
measures with the Entorpy Formula.

\begin{maintheorem}\label{mthm:formula de pesin para
    difeo-local com HT}
  Let $f:M \to M$ be non-uniformly expanding. Every
  expanding weak-$SRB$-like probability measure satisfies
  Pesin’s Entropy Formula. Moreover, all ergodic
  $SRB$-like probability measures are expanding.
\end{maintheorem}

Now we add a condition ensuring that there exists ergodic
weak-SRB-like measures and that all weak-SRB-like measures
are equilibrium states.

\begin{maincorollary}\label{mthm:pressao-top-zero-implica-existencia-de-medida-ergodica}
  Let $f:M \to M$ be non-uniformly expanding. If
  $P_{\textrm{top}}(f,\psi)=0$, then all weak-SRB-like
  probability measures are $\psi$-equilibrium states and all
  the (necessarily existing) expanding ergodic weak-SRB-like
  measures satisfy Pesin’s Entropy Formula.
\end{maincorollary}

\subsubsection{Weak and non-uniformly expanding maps}
\label{sec:weak-non-uniformly}

Now we strengthen the assumptions on non-uniform expansion
to improve the properties of SRB-like measures.

\begin{definition}\label{def:weak-exp}
  We say that $f$ is weak-expanding if $\|Df(x)^{-1}\|\leq1$
  for all $x\in M$.
\end{definition}

As usual, a probability measure is \emph{atomic} if it is
supported on a finite set. We denote $\supp(\mu)$ the
support of probability measure $\mu$ and recall that
$\psi=-\log Jf=-\log|\det Df|$ in what follows.

\begin{maincorollary}
  \label{mthm:WE}
 Let $f:M \rightarrow M$ be weak-expanding and non-uniformly expanding. Then,
   \begin{enumerate}
   \item all the (necessarily existing) weak-SRB-like probability
     measures are $\psi$-equilibrium states and, in
     particular, satisfy Pesin's Entropy Formula;
     \item there exists some ergodic weak-SRB-like probability
       measure;
     \item if $\psi<0$ (that is, $f$ is volume expanding),
       then there is no atomic weak-SRB-like probability
       measure;
     \item if $\SD=\{x\in M: \|Df(x)^{-1}\|=1\}$ is finite
       and $\psi<0$ then almost all ergodic components of a
       $\psi$-equilibrium state are weak-SRB-like
       measures. Moreover, for all weak-SRB-like probability
       measures $\mu$, its ergodic components $\mu_x$ are
       expanding weak-SRB-like probability measure for
       $\mu$-a.e. $x\in M\backslash \mathcal{D}$.
   \end{enumerate}
 \end{maincorollary}

In particular, we see that an analogous result to the
existence of an atomic physical measure for a quadratic map,
as obtained by Keller in \cite{HK90}, is not
possible in the $C^1$ expanding setting, although
generically such measures must be singular with respect to
any volume from.

\begin{maincorollary}
  No atomic measure is an SRB measure for a $C^1$ uniformly
  expanding map of a compact manifold.
\end{maincorollary}

We now restate the previous results in the uniformly expanding setting.

\begin{maincorollary}
  \label{mthm:Expansora}
  Let $f:M \rightarrow M$ be a $C^1$-expanding map. Then an
  $f$-invariant probability measure $\mu$ satisfies Pesin's
  Entropy Formula if and only if its ergodic components
  $\mu_x$ are weak-SRB-like $\mu$-a.e. $x\in M$. Moreover,
  all the (necessarily existing) weak-SRB-like probability
  measures satisfy Pesin's Entropy Formula, in particular,
  are $\psi$-equilibrium states. In addition, if $\mu$ is
  the unique weak-SRB-like probability measure, then $\mu$
  is an ergodic SRB probability measure with
  $\Leb(B(\mu))=1$, $\mu$ is the unique $\psi$-equilibrium
  state and satisfies a large deviation bound similar to
  Corollary~\ref{mthm:pressao-nao-negativa-implica-caracterizacao-dos-estados-de-equilibrio}.
\end{maincorollary}

Using this we get a weak statistical stability result in the
uniformly $C^1$-expanding setting: all weak$^*$ accumulation
points of weak-SRB-like measures are generalized convex
linear combinations of ergodic weak-SRB-like measures, as
follows.

\begin{maincorollary}
\label{mthm:estabilidade-estatistica}
Let $\{f_n:M\to M\}_{n\geq1}$ be a sequence of
$C^1$-expanding maps such that $f_n\xrightarrow[]{}f$ in the
$C^1$-topology and $f:M\to M$ be a $C^1$-expanding map. Let
$(\mu_n)_{n\geq1}$ be a sequence of weak-SRB-like measures
associated $f_n$. Then each accumulation point $\mu$ of
$(\mu_n)_{n\geq1}$ is an equilibrium state for the potential
$\psi=-\log|\det Df|$ (in particular, satisfies Pesin's
Entropy Formula) and almost all ergodic components of $\mu$
are weak-SRB-like measures.
\end{maincorollary}

\subsection{Further Questions}
\label{sec:further-questions}

Here we list some questions that have arisen during the
development of this work.

\begin{question}
  Under the same assumptions of
  Corollary~\ref{mthm:pressao-nao-negativa-implica-caracterizacao-dos-estados-de-equilibrio}
  can we obtain a lower bound of large deviations with the
  same rate? (Perhaps assuming that the dynamics is
  topologically mixing?)
\end{question}

\begin{question}
  Is it possible to obtain a statistical property for the
  weak-SRB-like measures, as in the case of
  Corollary~\ref{mthm:pressao-nao-negativa-implica-caracterizacao-dos-estados-de-equilibrio}
  for $C^1$-weak expanding and non uniformly expanding maps?
\end{question}

\begin{question}
  In Corollary~\ref{mthm:estabilidade-estatistica} is it
  possible to show that the limit measure is weak-SRB-like?
\end{question}

Some examples that have naturally arisen during the
development of this work motivate further questions.

The first example is attributed to Bowen and can be found in
greater detail in \cite[Example 5.5]{CatsEnrich2011}.

\begin{example}\label{examplo:olhodeBowen}
  Consider a diffeomorphism $f$ in a ball of $\RR^2$ with
  two hyperbolic saddle points $A$ and $B$ such that a
  connected component of the unstable global manifold
  $W^u(A)\backslash\{A\}$ is an embedded arc that coincides
  with a connected component of the stable global manifold
  $W^s(B)\backslash\{B\}$, and conversely, the embedded arc
  $W^u(B)\backslash\{B\} = W^s(A)\backslash\{A\}$. Take $f$
  such that there exists a source $C\subset U$ where $U$ is
  the open ball with boundary $W^u(A)\cup W^u(B)$.

\begin{figure}[!htb]
\begin{center}
\begin{minipage}{ \linewidth}
         \centering
          \includegraphics[scale=0.5]{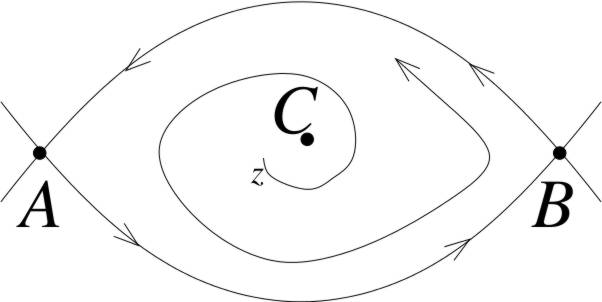}
           \caption{Bowen’s eye. }
           \label{fig:olhodeBowen }
          \end{minipage}
      \end{center}
\end{figure}

If the eigenvalues of the derivative of $f$ at $A$ and $B$
are adequately chosen as specified in \cite{Ta95,gol07},
then the empiric sequence (\ref{11}) for all
$x\in U\backslash\{C\}$ is not convergent. It has at least
two subsequences convergent to different convex combinations
of the Dirac deltas $\delta_A$ and $\delta_B$.

Thus, as observed in \cite{CatsEnrich2011}, the SRB-like
probability measures are convex linear combinations of
$\delta_A$ and $\delta_B$ and form a segment in the space
$\SM$ of probability measures. This example shows that the
SRB-like measures are not necessarily ergodic.

Moreover, the eigenvalues of $Df$ at the saddles $A$ and $B$
can be modified to obtain, instead of the result above, the
convergence of the sequence (1.1) as stated in Lemma (i) of
page 457 in \cite{TAKAHASHI84}. In fact, taking conservative
saddles (and $C^0$ perturbing $f$ outside small
neighborhoods of the saddles $A$ and $B$ so the topological
$\omega$-limit of the orbits in $U\backslash\{C\}$ still
contains $A$ and $B$), one can get for all
$x\in U\backslash\{C\}$ an empirical sequence \eqref{11}
that is convergent to a single measure
$\mu =\lambda\delta_A+(1-\lambda)\delta_B$, with a fixed
constant $0 < \lambda < 1$. So, $\mu$ is the unique SRB-like
measure. This proves that the set of SRB-like probability
measures does not depend continuously on the map.
\end{example}

This example motivates the following questions:
\begin{question}
  When can we say that a dynamic system $(f,\phi)$ has
  ergodic $\nu_{\phi}$-SRB-like measures?
\end{question}

From Example \ref{examplo:olhodeBowen}, we know that the set
of $\nu_{\phi}$-SRB-like probability measures does not in
general depend continuously on the map.

\begin{question}\label{question1}
  Fixing the potential $\phi$, is there continuous
  dependence of $\nu_{\phi}$-SRB-like probability measures
  as functions of the underlying map (in the $C^1$
  topology)?
\end{question}

Recently \cite{AlRaSiq} a result on this direction was
obtained for H{\"o}lder and hyperbolic potentials on
non-uniformly expanding maps.

\begin{question}
  It there continuous dependence of $\nu_{\phi}$-SRB-like
  probability measures as functions of the potential $\phi$?
\end{question}
A positive response in this direction is given by Theorem
\ref{mthm:theorem-conv-das-medidas-tipo-SRB}, however we
still do not know in general whether the limit measure is
ergodic.
% in Theorem \ref{mthm:topologically-exact-implies-upper-bound-for-large-deviations} we give conditions for that the limit measure is ergodic.

Recently \cite{Cat17} shows that under general random
perturbations only SRB measures can be empirically
stochastically stable.

\begin{question}
  Can we obtain stochastic stability for $SRB$-like measures
  under certain random perturbations?
\end{question}

The next example is an adaptation of the ‘intermittent’
Manneville map into a local homeomorphism of the circle.

\begin{example}\label{ex:intermittent}
  Consider $I=[-1, 1]$ and the map $\hat{f}:I\rightarrow I$
  (see Figure \ref{fig 1}) given by
$$\hat{f}(x)= \left\{ \begin{array}{ll}
2\sqrt{x}-1 & \textrm{if $x\geq0,$}\\
1-2\sqrt{|x|} & \textrm{otherwise.}
\end{array} \right.$$

This map induces a continuous local homeomorphism
$f:\mathbb{S}^1\rightarrow\mathbb{S}^1$ through the
identification $\mathbb{S}^1= I/\sim$, where $-1\sim 1$, not
differentiable at the point $0$. This is a map
differentiable everywhere except at a single point, having a
positive frequency of hyperbolic times at Lebesgue almost
every point as can be seen in \cite[section
5]{alves-araujo2004}.

\begin{figure}[!htb]
\begin{center}
\begin{minipage}{ \linewidth}
         \centering
          \includegraphics[scale=0.5]{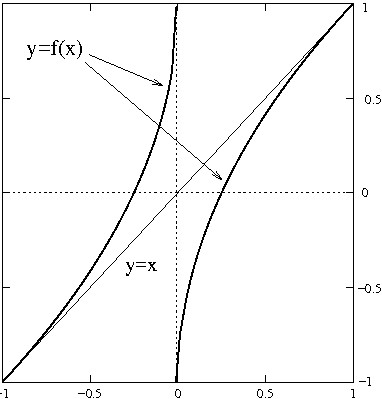}
           \caption{Graph of $\hat{f}$. }
           \label{fig 1}
          \end{minipage}
      \end{center}
\end{figure}
\end{example}

This example is weak expanding and non uniformly expanding
and suggests that we can think of a generalization of
Theorem
\ref{mthm:Toda-medida-nu-SRB-like-e-estado-de-equilibrio}
using the same ideas of Theorem \ref{mthm:formula de pesin
  para difeo-local com HT}, replacing Lebesgue measure by an
expanding $\phi$-conformal measure $\nu$ for
$\phi\in C^0(X,\RR)$ and the potential $\psi=-\log Jf$ by
$ -\log J_{\nu} f$, where $J_{\nu} f$ is the Jacobian $f$
with respect to $\nu$.

\begin{question}
  Let $T:X\rightarrow X$ be a local homeomorphism (not
  necessarily distance expanding) of the compact metric
  space $X$ and let $\phi:X\rightarrow \RR$ be a continuous
  potential. If there exists an expanding $\phi$-conformal
  measure $\nu$, then all the (necessarily existing)
  $\nu$-SRB-like measures are equilibrium states for the
  potential $\phi$?
\end{question}

It may be necessary to consider additional hypotheses about
the continuous potentials, e.g. that they are
\emph{hyperbolic potentials} (see definition and results
about H\"older hyperbolic potentials in
\cite{RamosViana2017}).

%Moreover, we can also consider a $C^1$ local diffeomorphism
%with positive frequency of hyperbolic times and non-empty
%critical/singular set (see \cite{ABV00} for definition of
%critical/singular set)
%
%\begin{question}
%  Let $f:M \to M$ be $C^1$ local diffeomorphism away from a
%  critical/singular set $\SC$ with slow recurrence to $\SC$
%  and such that for some $0<\sigma<1$, $b,\delta>0$ and
%  $\theta>0$ $\Leb-$a.e. $x$ has positive frequency
%  $\geq\theta$ of a $(\sigma, \delta,b)-$hyperbolic times of
%  $f$. Does every expanding SRB-like probability measure
%  satisfy Pesin's Entropy Formula?
%\end{question}

\subsection{Organization of the Paper}
\label{sec:organization-paper}

We present in Section \ref{sec:exampl-applic} some examples
of application for the main results.  Section
\ref{sec:Preliminary definitions and results} contains
preparatory material of the theory of $\nu$-weak-SRB-like
measures and measure-theoretic entropy that will be
necessary for the proofs.  In Sections \ref{sec:Continuous
  variation of conformal measures and SRB-like measures} and
\ref{sec:distance-expanding-maps} we prove Theorems
\ref{mthm:theorem-conv-das-medidas-tipo-SRB} and
\ref{mthm:Toda-medida-nu-SRB-like-e-estado-de-equilibrio}. In
Section \ref{sec:Pesin's Entropy Formula} we use the
properties of hyperbolic times and weak-SRB-like measures to
prove a reformulation of Pesin's Entropy Formula for $C^1$
non-uniformly expanding local diffeomorphism (Theorem
\ref{mthm:formula de pesin para difeo-local com HT}). In
Section \ref{sec:Ergodic SRB-like measure} we prove
Corollary
\ref{mthm:pressao-top-zero-implica-existencia-de-medida-ergodica}. Finally,
in Section \ref{sec:weak-expanding} we prove the results
stated in Corollaries \ref{mthm:WE},
\ref{mthm:pressao-nao-negativa-implica-caracterizacao-dos-estados-de-equilibrio},
\ref{mthm:Expansora} and
\ref{mthm:estabilidade-estatistica}.

\subsection*{Acknowledgments}
\label{sec:acknowledgments}

This is the PhD thesis of F. Santos at the Instituto de
Matematica e Estatistica-Universidade Federal da Bahia
(UFBA) under a CAPES scholarship. He thanks the Mathematics
and Statistics Institute at UFBA for the use of its
facilities and the finantial support from CAPES during his
M.Sc. and Ph.D. studies. The authors thank the anonymous
referee for the careful reading and the many useful
suggestions that greatly helped to improve the quality of
the text.

\section{Examples of application}
\label{sec:exampl-applic}

Here we present some examples of application. In the first
subsection we describe the construction of a $C^1$ uniformly
expanding map with many different SRB-like measures which,
in particular, satisfies the assumptions of Theorem
\ref{mthm:Toda-medida-nu-SRB-like-e-estado-de-equilibrio}
and Corollary \ref{mthm:Expansora}.

In the second subsection we present a class of non uniformly
expanding $C^1$ local diffeomorphisms satisfying the
assumptions of Theorem \ref{mthm:formula de pesin para
  difeo-local com HT}.

In the third subsection we exhibit examples of
weak-expanding and non uniformly expanding transformations
which are not uniformly expanding in the setting of
Corollary \ref{mthm:WE}.

\subsection{Expanding map with several absolutely continuous
  invariant pro\-ba\-bi\-li\-ty measures}
\label{sec:expanding map}

We present the construction of a $C^1$-expanding map
$f:\sS^1\rightarrow\sS^1$ which is not of class
$C^{1+\epsilon}$ for any $0<\epsilon<1$ and that has several
SRB-like measures. Moreover, such measures are absolutely
continuous with respect to Lebesgue measure (such $f$ does
not belong to the $C^1$ generic subset of $C^1(M,M)$ found
in \cite{AB06}).

\begin{example}
  Following Bowen~\cite{Bo75b}, we first adopt some notation for
  Cantor sets with positive Lebesgue measure. Let $I$ be a
  closed interval and $\alpha_n>0$ numbers with
  $\sum_{n=0}^{\infty}\alpha_n<|I|$, where $|E|$ denotes the
  one-dimensional Lebesgue measure of the subset $E$ of
  $I$. Let $\underline{a}=a_1a_2\ldots a_n$ denote a
  sequence of $0's$ and $1's$ of length
  $n=n(\underline{a})$; we denote the empty sequence
  $\underline{a}=\emptyset$ with
  $n(\underline{a})=0$. Define $I_{\emptyset}=I=[a,b]$,
  $I_{\emptyset}^*=\left[\frac{a+b}{2}-\frac{\alpha_0}{2},\frac{a+b}{2}+\frac{\alpha_0}{2}\right]$
  and $I_{\underline{a}}^*\subset I_{\underline{a}}$
  recursively as follows.

  Let $I_{\underline{a}0}$ and $I_{\underline{a}1}$ be the
  left and right intervals remaining when the interior of
  $I_{\underline{a}}^*$ is removed from $I_{\underline{a}}$;
  let $I_{\underline{a}}^*$ be the closed interval of length
  $\frac{\alpha_{n(\underline{a}k)}}{2^{n(\underline{a}k)}}$
  and having the same center as $I_{\underline{a}k}$
  (k=0,1).

  The Cantor set $K_I$ is given by
  $K_I=\bigcap\limits_{m=0}^{\infty}\bigcup\limits_{n(\underline{a})=m}I_{\underline{a}}.$

  This is the standard construction of the Cantor set except
  that we allow ourselves some flexibility in the lengths of
  the removed intervals. The measure of $K_I$ is
  $\Leb(K_I)=|I|-\sum\limits_{n=0}^{\infty}\alpha_n>0$.

  Suppose that another interval $J\supset I$ is given
  together with $\beta_n>0$ such that
  $\sum_{n=0}^{\infty}\beta_n<|J|$. One can then construct
  $J_{\underline{a}}$, $J_{\underline{a}}^*$ and $K_J$ as
  above. Let us assume now that
  $\frac{\beta_n}{\alpha_n}\xrightarrow[n\rightarrow
  \infty]{} \gamma\geq0$. Following the cons\-truc\-tion of
  Bowen \cite{Bo75b}, we get $g:I\rightarrow J$ a $C^1$
  orientation preserving homeomorphism so that
  $g'(x)=\gamma$ for all $x\in K_I$ and $g'(x)>1$ for all
  $x\in I$.

  More precisely, let us take $J=[-1,1]$ and choose
  $\beta_n>0$ with $\sum_{n=0}^{\infty}\beta_n<2$ and
  $\frac{\beta_{n+1}}{\beta_n}\rightarrow1$,
  $\left(e.g. \ \beta_n=\frac{1}{(n+100)^2}\right)$. Let
  $I=\left[\frac{\beta_0}{2},1\right]$ and
  $\alpha_n=\frac{\beta_{n+1}}{2}$. Then
  $\sum_{n=0}^{\infty}\alpha_n<1-\frac{\beta_0}{2}$ and
  $\gamma=\lim\frac{\beta_{n}}{\alpha_n}=2$. We define a
  homeomorphism $G:(-I)\cup I\rightarrow J$ by
  $G(x)= \left\{ \begin{array}{ll}
                   g(x) & \textrm{if $x\in I$}\\
                   -g(-x) & \textrm{if $x\in (-I)$}
\end{array} \right.,$ where $K_J=\bigcap_{n=0}^{+\infty}G^{-n}(J)$ and $G|_{K_J}:K_J\circlearrowleft$.

Consider now
$c_1=\left(\frac{3\beta_0}{2}-2\right)\left(\frac{4}{\beta_0}\right)^3$,
$c_2=\left(\frac{9\beta_0}{4}-3\right)\left(\frac{4}{\beta_0}\right)^2$;
$f_1:[-\frac{\beta_0}{2},-\frac{\beta_0}{4}]\rightarrow[-1,-\frac{\beta_0}{4}]$
given by
$$f_1(x)=c_1\left(x+\frac{\beta_0}{2}\right)^3-c_2\left(x+\frac{\beta_0}{2}\right)^2+2\left(x+\frac{\beta_0}{2}\right)-1$$
and
$f_2:[\frac{\beta_0}{4},\frac{\beta_0}{2}]\rightarrow[\frac{\beta_0}{4},1]$
given by
$f_2(x)=c_1\left(x-\frac{\beta_0}{2}\right)^3+c_2\left(x-\frac{\beta_0}{2}\right)^2+2\left(x-\frac{\beta_0}{2}\right)+1$. Then,
we have
\begin{enumerate}
 \item $f_1\left(-\frac{\beta_0}{2}\right)=-1$, \  $f_1\left(-\frac{\beta_0}{4}\right)=-\frac{\beta_0}{4}$, \ $f_2\left(\frac{\beta_0}{4}\right)=\frac{\beta_0}{4}$ and $f_2\left(\frac{\beta_0}{2}\right)=1$
\item $f_1^+\left(\frac{\beta_0}{2}\right)=\lim\limits_{h\rightarrow 0^+}\frac{f_1(\frac{\beta_0}{2}+h)-f_1(\frac{\beta_0}{2})}{h}=2$, \ $f_2^+\left(\frac{\beta_0}{4}\right)=2$,\ $f_1^-\left(\frac{\beta_0}{4}\right)=\lim\limits_{h\rightarrow 0^-}\frac{f_1(\frac{\beta_0}{4}+h)-f_1(\frac{\beta_0}{4})}{h}=2$ and $f_2^-\left(\frac{\beta_0}{2}\right)=2.$
    \end{enumerate}

    Consider now
    $
    J_1=\left[-\frac{\beta_0}{4},\frac{\beta_0}{4}\right]$,
    $I_1=\left[\frac{\beta_0^2}{8},\frac{\beta_0}{4}\right]$
    and choose $\beta_n'=\frac{\beta_0\beta_n}{4}>0$ with
    $\sum_{n=0}^{\infty}\beta_n'<\frac{\beta_0}{2}$ and
    $\frac{\beta_{n+1}'}{\beta_n'}=\frac{\beta_{n+1}}{\beta_n}\rightarrow1$. Let
    $\alpha_n'=\frac{\beta_{n+1}'}{2}$, then
    $\sum_{n=0}^{\infty}\alpha_n'=\sum_{n=0}^{\infty}\frac{\beta_{n+1}'}{2}<\frac{1}{2}\left(\frac{\beta_0}{2}-\beta_0'\right)=\frac{1}{2}\left(\frac{\beta_0}{2}-\frac{\beta_0^2}{4}\right)$
    and
    $\gamma=\lim\frac{\beta_{n}'}{\alpha_n'}=\lim\frac{2\beta_{n}}{\beta_{n+1}}=2$.

    Similarly to the above construction, we obtain a
    homeomorphism $G_1:(-I_1)\cup I_1\rightarrow J_1$ given
    by
$$G_1(x)= \left\{ \begin{array}{ll}
g_1(x) & \textrm{if $x\in I_1$}\\
-g_1(-x) & \textrm{if $x\in (-I_1)$}
%G_1\left(\left(-\frac{\beta_0^2}{8},\frac{\beta_0^2}{8}\right)\right)\cap J_1=\emptyset
                  \end{array} \right.,$$
                where $K_{J_1}=\bigcap_{n=0}^{+\infty}G_1^{-n}(J_1)$ is a Cantor set  with positive Lebesgue measure, $G_1|_{K_{J_1}}:K_{J_1}\circlearrowleft$ and $g_1:I_1\rightarrow J_1$ is a $C^1$ orientation preserving homeomorphism so that $g_1'(x)=2$ for all $x\in K_{J_1}$ and $g_1'(x)>1$ for all $x\in I_1$.

                Similarly we obtain
                $f_3:[0,\frac{\beta_0^2}{8}]\rightarrow[-1,-\frac{\beta_0}{4}]$
                and
                $f_4:[-\frac{\beta_0^2}{8},0]\rightarrow[\frac{\beta_0}{4},1]$
                such that $f_3(0)=-1$,
                $f_3(\frac{\beta_0^2}{8})=-\frac{\beta_0}{4}$,
                $f_3^-(\frac{\beta_0^2}{8})=2$,
                $f_3^+(0)=2$,
                $f_4(\frac{-\beta_0^2}{8})=\frac{\beta_0}{4}$,
                $f_4(0)=1$, $f_4^+(-\frac{\beta_0^2}{8})=2$
                and $f_4^-(0)=2$

Finally, define the function $f:J\rightarrow J$ by (see Figure \ref{fig2} for its graph)
$$f(x)= \left\{ \begin{array}{ll}
G(x) & \textrm{if $x\in (-I)\cup I$}\\
f_1(x) & \textrm{if $x\in \left(-\frac{\beta_0}{2},-\frac{\beta_0}{4}\right)$}\\
f_2(x) & \textrm{if $x\in \left(\frac{\beta_0}{4},\frac{\beta_0}{2}\right)$}\\
G_1(x) & \textrm{if $x\in (-I_1)\cup I_1$}\\
f_4(x) & \textrm{if $x\in \left(-\frac{\beta_0^2}{8},0\right)$}\\
f_3(x) & \textrm{if $x\in \left(0,\frac{\beta_0^2}{8}\right)$}
\end{array} \right.,$$

\begin{figure}[!htb]
\begin{center}
\begin{minipage}{ \linewidth}
         \centering
          \includegraphics[scale=0.3]{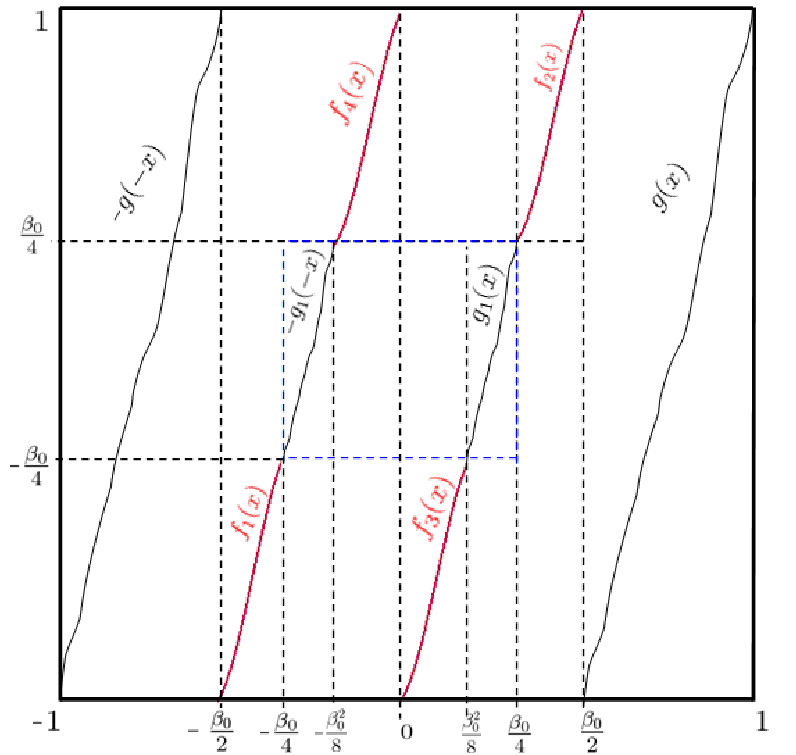}
           \caption{Graph of $f$. }
           \label{fig2}
          \end{minipage}
      \end{center}
\end{figure}

Identifying $-1$ and $1$ and making a linear change of
coordinates we obtain a $C^1$ expanding map of the circle
$f:S^1 \circlearrowleft$ which is not of class
$C^{1+\epsilon}$ for any $\epsilon>0$; see Bowen~\cite{Bo75b}.

Consider $\Leb_{K_J}$ the normalized Lebesgue measure of the
set $K_J$, this is,
$\Leb_{K_J}(A)=\Leb(A\cap K_J)/\Leb(K_J)$ for all measurable
$A\subset J$. Denote $ \Sigma^+_2=\{0,1\}^{\NN}$ and
consider the homeomorphism $h:K_J\rightarrow \Sigma^+_2$
that associates each point $x\in K_J$ the sequence
$\underline{a}\in \Sigma^+_2$ describing its location in the
set $K_J$, this is, $\underline{a}$ is such that
$I_{\underline{a}}\cap K_J=\{x\}$. Let $\mu$ be the
Bernoulli measure in $\Sigma^+_2$ giving weight 1/2 to each
digit.

\begin{claim}
$\mu=h_*\Leb_{K_J}$
\end{claim}

We show that this relation holds for an algebra of subsets
which ge\-ne\-ra\-tes the Borel $\sigma$-algebra of
$\Sigma^+_2$. For $\underline{a}=a_1a_2\ldots a_n$ we have
that $\mu([\underline{a}])=1/2^n$. Moreover,
$h^{-1}([\underline{a}])= I_{\underline{a}}\cap K_J$, so
$m_{K_J}(h^{-1}([\underline{a}]))= m_{K_J}(
I_{\underline{a}}\cap K_J)$. By construction, at each step
all remaining intervals in the construction of $K_J$ have
the same length. Thus, the n-th stage contains $2^n$
intervals $I_{\underline{a}}$ (among them) all with the same
Lebesgue measure. Since $\Leb_{K_J}$ is a probability
measure, we have
$$ 2^n \Leb_{K_J}( I_{\underline{a}}\cap K_J)=1\Longrightarrow \Leb_{K_J}( I_{\underline{a}}\cap K_J)=\frac{1}{2^n}.$$

We conclude that
$\Leb_{K_J}( I_{\underline{a}}\cap
K_J)=\Leb_{K_J}(h^{-1}([\underline{a}]))=\mu([\underline{a}])$
for all $\underline{a}=a_1a_2\ldots a_n$, $n\geq1$ proving
the Claim.

Clearly $h\circ f|_{K_J}=\sigma\circ h$ where $\sigma$ is
the standard left shift
$\sigma:\Sigma^+_2\circlearrowleft$. Since $\mu$ is
$\sigma$-invariant, then $\Leb_{K_J}$ is
$f|_{K_J}$-invariant. Consequently $(f, \Leb_{K_J})$ is
mixing,
$h_{\textrm{top}}(f)=\log2=h_{\Leb_{K_J}}(f)=h_{\mu}(\sigma)$
and $\Leb_{K_J}\ll \Leb$.

Hence
$\lim\limits_{n\rightarrow+\infty}\sigma_n(x)=\Leb_{K_J}$
for $\Leb_{K_J}-a.e$, $x\in K_J$ and $|K_J|>0$, so it
follows that $\Leb_{K_J}$ is a SRB-like measure. By Theorem
\ref{mthm:Toda-medida-nu-SRB-like-e-estado-de-equilibrio},
$\Leb_{K_J}$ is an equilibrium state for the potential
$\psi=-\log |f'|$.

Similarly $ \Leb_{K_{J_1}}$ is also a SRB-like measure, but distinct form $\Leb_{K_J}$.
\end{example}

We observe that this construction allows us to obtain
countably many ergodic SRB-like measures by reapplying the
construction to each removed subinterval of the first Cantor
set.

We also note that, if we take a sequence of H\"older
continuous potentials $\phi_n$ converging to
$\psi=-\log|f'|$; choose $\nu_n$ a $\phi_n$-conformal
measure and a $\nu_n$-SRB measure $\mu_n$ and weak$^*$
accumulation points $\nu$, $\mu$ as in Theorem
\ref{mthm:theorem-conv-das-medidas-tipo-SRB}, since $f$ is
topologically exact, then $\nu(A_{\varepsilon}(\mu))=1$ for
all $\epsilon>0$.  Therefore $\nu$ is not Lebesgue measure
on $\mathbb{S}^1$.

\subsection{$C^1$ Non uniformly expanding maps}
\label{sec:c1-non-uniformly}

Next we present an example that is a robust ($C^1$ open)
class of non-uniformly expanding $C^1$ local
diffeomorphisms.

This family of maps was introduced in \cite{ABV00} (see
also subsection 2.1 in \cite[Chapter 1]{Ze03}) and maps in
this class exhibit non-uniform expansion Lebesgue almost
everywhere but are not uniformly expanding.

\begin{example}\label{ex:C1NUE}
  Let $M$ be a compact manifold of dimensiom $d\geq 1$ and
  $f_0:M \circlearrowleft$ is a $C^1-$expanding map. Let
  $V\subset M$ be some small compact domain, so that the
  restriction of $f_0$ to $V$ is injective. Let $f$ be any
  map in a sufficiently small $C^1$-neighborhood $\SN$ of
  $f_0$ so that:
\begin{enumerate}
\item $f$ is volume expanding everywhere: there exists
  $\sigma_1 > 1$ such that
  $$|\det Df(x)|\geq\sigma_1 \ \ \textrm{for every} \ x\in M$$
\item $f$ is expanding outside $V$: there exists
  $\sigma_0 < 1$ such that
$$\|Df(x)^{-1}\|< \sigma_0 \ \ \textrm{for every} \  x\in M\backslash V;$$
  \item $f$ is not too contracting on $V$ : there is some small $\delta > 0$ such that
$$\|Df(x)^{-1}\|< 1 +\delta \ \ \textrm{for every}  \ x\in V.$$
\end{enumerate}

Then every map $f$ in such a $C^1$-neighborhood $\SN$ of
$f_0$ is non-uniformly expanding (see a proof in subsection
2.1 in \cite{Ze03}). By Theorem \ref{mthm:formula de pesin
  para difeo-local com HT}, every expanding weak-SRB-like
probability measure satisfies Pesin's Entropy Formula, and
every ergodic SRB-like measure is expanding.

Such classes of maps can be obtained through deformation of
a uniformly expanding map by isotopy inside some small
region.
\end{example}

\subsection{Weak-expanding and non-uniformly expanding maps}
\label{sec:weak-expanding-non}

Consider $\alpha > 0$ and the map
$T_{\alpha}:[0,1]\rightarrow [0,1]$ defined as follows
\begin{align*}
T_{\alpha}(x)= \left\{ \begin{array}{ll}
x+2^{\alpha}x^{1+\alpha} & \textrm{if $x\in [0,1/2)$}\\
x-2^{\alpha}(1-x)^{1+\alpha} & \textrm{if $x\in [1/2,1]$}
\end{array} \right..
\end{align*}
This map defines a $C^{1+\alpha}$ map of the unit circle
$\sS^1:= [0, 1]/\sim$ into itself, called
\emph{intermittent maps}. These applications are expanding,
except at a neutral fixed point, the unique fixed point is
$0$ and $DT_{\alpha}(0)=1$.  The local behavior near this
neutral point % is responsible for various phenomena. The
% above family of maps
provides many interesting results in ergodic theory. If
$\alpha\geq 1$, i.e. if the order of tangency at zero is
high enough, then the Dirac mass at zero $\delta_0$ is the
unique physical probability measure and so the Lyapunov
exponent of Lebesgue almost all points vanishes (see
\cite{thaler1983}). This example shows that there exists
systems which are weak-expanding but not non-uniformly
expanding. Hence the assumption of weak expansion together
with non-uniform expansion in Theorem \ref{mthm:WE} is not
superfluous.

The following is an example of a map which is weak-expanding
and non-uniformly expanding in the setting of Theorem
\ref{mthm:WE}.

\begin{example}\label{ex:C1wExpNUE}
  If, in the setting of the construction of the previous
  Example~\ref{ex:C1NUE}, the region $V$ of a point $p$ of a
  periodic orbit with period $k$, and the deformation
  weakens one of the eigenvalues of $(Df_0^k)_p$ in such a way
  that $1$ becomes and eigenvalue of $Df_p^k$, then $f$ is
  an example of a weak expanding and non uniformly expanding
  $C^1$ transformation.

  More precisely, consider the function
  $g_0(t)=\frac{t}{\log(1/t)}, 0<t<1/2$, $g_0(0)=0$. It is
  easy to see that
  \begin{itemize}
  \item $g_0'(t)>0, 0<t<1/2$ and so $g_0$ is stricly increasing;
  \item $g_0'$ is not of $\alpha$-generalized bounded
    variation for any $\alpha\in(0,1)$ (see \cite{Ke85} for
    the definition of generalized bounded variation) and so
    $g_0'$ is not $C^\alpha$ for any $0<\alpha<1$.
  \end{itemize}
  Setting $g(t)=g_0(t)/g_0(1/2)$ we obtain
  $g:[0,1/2]\to[0,1]$ a $C^1$ strictly increasing function
  which is not $C^{1+\alpha}$ for any $\alpha\in(0,1)$.  Now
  we consider the analogous map to $T_\alpha$
  \begin{align*}
  T(x)= \left\{ \begin{array}{ll}
                  x+x g(x) & \textrm{if $x\in [0,1/2)$}\\
                  x-(1-x)g(1-x) & \textrm{if $x\in [1/2,1]$}
\end{array} \right..
  \end{align*}
  This is now a $C^1$ map of the circle into itself which is
  not $C^{1+\alpha}$ for any $0<\alpha<1$. Moreover, letting
  $f_0=T\times E$ where $E(x)=2x\mod1$, we see that $f_0$
  satisfies items (1-3) in Example~\ref{ex:C1NUE} for $V$ a
  small neighborhood of the fixed point $(0,0)$.

  Hence $f_0$ is a $C^1$ non-uniformly expanding
  map. Moreover, since $T'(0)=1$, we have that $f_0$ is also
  a $C^1$ weak expanding map with $\SD=\{(0,0)\}$.
\end{example}

Now we extend this construction to obtain a weak expanding
and non uniformly expanding $C^1$ map so that $\SD$ is non
denumerable.

\begin{example}\label{ex:C1wExpNUE-nondenum}
  Let $K \subset I = [0, 1]$ be the middle third Cantor set
  and let $\beta_n(x)=d(x,K\cap[0,3^{-n}])$, where $d$ is
  the Euclidean distance on $I$. Note that $\beta_n$ is
  Lipschitz and $g_0'\circ\beta_n$ is bounded and continuous
  but not of $\alpha$-generalized bounded variation for any
  $0<\alpha<1$, where $g_0$ was defined in
  Example~\ref{ex:C1wExpNUE}. Indeed, since
  $(3^{-k},2\cdot3^{-k})$ is a gap of $K\cap[0,3^{-n}]$ for
  all $k>n$, then
  \begin{align*}
    \frac{g_0'(\beta(3^{1-k}/2)) -
    g_0'(\beta(3^{-k}))}{3^{1-k}/2 -3^{-k}}
    &=
      2\cdot3^k g_0'\left(\frac1{2\cdot 3^{k-1}}\right)
    \\
    &=
      \frac{2\cdot3^k}{(1-k)\log3-\log2}\left(\frac1{(1-k)\log3-\log2}-1\right)
  \end{align*}
  is not bounded when $k\nearrow\infty$.  If
  $h: I\rightarrow \RR$ is given by
  $h(x) = x +\big(\int_0^xg_0'\circ\beta_n
  \big)\big(\int_0^1g_0'\circ\beta_n\big)^{-1}$, where the
  integrals are with respect to Lebesgue measure on the real
  line, then $h(0) = 0$, $h(1)=2$ and $h$ induces a $C^1$
  map $h_0:\sS^1\to\sS^1$ whose derivative is continuous but
  not $C^\alpha$ for any $0<\alpha<1$, such that $h_0'(x)=1$
  for each $x\in K\cap[0,3^{-n}]$ and $h_0'(x)>1$ otherwise.

  We set now
  $h_t(x)=x +\big(\int_0^x(t+g_0'\circ\beta_n)\big)
  \big(\int_0^1(t+g_0'\circ\beta_n)\big)^{-1}$ for
  $t\in[0,1]$, which induces a $C^1$ map of $\sS^1$ into
  itself with continuous derivative and $h_t^\prime>1$ for
  $t>0$. We then define the skew-product map
  $f_0(x,y)=(E(x),h_{\sin \pi x}(y))$ for
  $(x,y)\in\sS^1\times\sS^1$ which is a weak expanding map
  with $\SD=\{0\}\times( K\cap[0,3^{-n}])$.

  For sufficiently big $n$ it is easy to verify that $f_0$
  satisfies items (1-3) in Example~\ref{ex:C1NUE} for $V$ a
  small neighborhood of the fixed point $(0,0)$. Hence $f_0$
  is also a $C^1$ non-uniformly expanding map which is not a
  $C^{1+\alpha}$ map for any $0<\alpha<1$.
\end{example}

%%%%%%%%%%%%%%%%%%%%%%%%%%%%%%%%%%%%%%%%%%%%%%%%%%%%

\section{Preliminary definitions and results}
\label{sec:Preliminary definitions and results}

In this section, we revisit the definition and some results
on the theory $\nu$-SRB-like measures and also some
properties of measure-theoretic entropy.

\subsection{Invariant and $\nu$-SRB-like measures}
\label{sec:invariant measure and SRB-like measure}

Let $X$ be a compact metric space and $T:X\rightarrow X$ be
continuous. Denote $\mathcal{M}$ the set of all Borel
probability measures on $X$ and $\mathcal{M}_{T}$ the set of
all $T-$invariant Borel probability measures on $X$. In
$\SM$ fix a metric compatible with the weak$^*$ topology on
$\mathcal{M}$, e.g.
\begin{eqnarray}\label{eq13}
\dist(\mu,\nu):=\sum\limits_{i=0}^{+\infty}\frac{1}{2^i}\left|\int\phi_id\mu-\int\phi_id\nu\right|,
\end{eqnarray}
where $\{\phi_i\}_{i\geq0}$ is a countable family of
continuous functions that is dense in the space
$C^0(X,[0,1])$.  The following will be used in the proofs of
the large deviations lemma and Theorem
\ref{teorema:as-medidas-SRB-like-tem-pressao-nao-negativa}
that are essential the proofs of Theorems
\ref{mthm:Toda-medida-nu-SRB-like-e-estado-de-equilibrio},
\ref{mthm:WE} and \ref{mthm:formula de pesin para
  difeo-local com HT}.

\begin{lemma}
\label{lm1}
For all $\varepsilon > 0$ there is $\delta>0$ such that if
$d(T^i(x),T^i(y))<\delta$ for all $i=0,\ldots, n-1$ then
$\dist(\sigma_n(x),\sigma_n(y))<\varepsilon$
\end{lemma}

\begin{proof}
  This is a simple consequence of compactness and
  metrizability of $\mathcal{M}$ with the weak$^*$ topology.
\end{proof}

Next results ensure the existence of $\nu$-SRB-like measure
and consequently, by Remark
\ref{rmk:toda-medida-SRB-like-eh-weak-SRB-like}, that
$\nu$-weak-SRB-like measures do always exist.

Given a probability measure $\nu$ on $M$ let
$\SW_T(\nu)\subset\mathcal{M}_T$ be the family of
weak$^*$ accumulation measures $p\omega(x)$ for
$\nu\text{-a.e.} x$. This is a well-defined compact subset
of $\SM_T$ in the weak$^*$ topology as follows.

\begin{proposition}\label{proposition:Wf-compacto}
  Let $T:X\rightarrow X$ be a continuous map of a compact
  metric space $X$. For each reference probability measure
  $\nu$ there exist
  $\mathcal{W}_{T}(\nu)\subset\mathcal{M}_T$ the unique
  minimal non-empty and weak$^*$ compact set, such that
  $p\omega(x)\subset \mathcal{W}_{T}(\nu)$ for
  $\nu$-a.e. $x\in X$. When $\nu=\Leb$ we denote
  $\mathcal{W}_{T}(\Leb)=\mathcal{W}_{T}$.
\end{proposition}

\begin{proof}
  See proof of \cite[Theorem 1.5]{CatsEnrich2011}, replacing
  $\Leb$ by $\nu$.
\end{proof}

The subset $\SW_T(\nu)$ contains all the $\nu-$SRB-like
measures, as follows.

\begin{proposition}
  Given $\nu$ a reference probability measure, a probability
  measure $\mu\in\mathcal{M}$ is $\nu$-SRB-like if and only
  if $\mu\in \mathcal{W}_{T}(\nu)$.
\end{proposition}

\begin{proof}
See proof of \cite[Proposition 2.2]{CatsEnrich2012}.
\end{proof}

%\begin{proof}
%For any $\varepsilon > 0$ and any $\mu\in\mathcal{M}$ let
%$$B_{\varepsilon}(\mu)=\{\nu\in\mathcal{M}; \dist(\mu,\nu)<\varepsilon\}.$$
%If $\mu\in\mathcal{W}_T(\nu)$ then $\nu(A_{\varepsilon}(\mu))>0$ for all $\varepsilon > 0$, because if not, the compact set $K:=\mathcal{W}_T(\nu)\setminus B_{\varepsilon}(\mu)$ would be strictly contained in $\mathcal{W}_T(\nu)$ and such that $p\omega(x)\subset K$ for $\nu$ almost all $x\in X$. This last contradicts the minimality of $\mathcal{W}_T(\nu)$.
%
%On the other hand, if a probability measure $\mu$ satisfies the inequality $\nu(A_{\varepsilon}(\mu))>0$ for all $\varepsilon > 0$, and since  $p\omega(x)\subset \mathcal{W}_T(\nu)$ for $\nu$-a.e. $x\in X$, we obtain $B_{\varepsilon}(\mu)\cap\mathcal{W}_T(\nu)\neq \emptyset$ for all $\varepsilon > 0$. Namely, $\mu$ is in the weak$^*$-closure of $\mathcal{W}_T(\nu)$. Since $\mathcal{W}_T(\nu)$ is weak$^*$-compact (see Proposition \ref{proposition:Wf-compacto}) we conclude that $\mu\in\mathcal{W}_T(\nu)$ as wanted.
%\end{proof}

For more details on SRB-like measures, see
\cite{Cats14,CatsEnrich2011,CatsEnrich2012,catsigeras2016weak}.

%In this paper we consider the family $E^1(M)$ of all the C$^1$-expanding maps $f:M\rightarrow M$. We recall that a $C^1$-map $f:M\rightarrow M$ is expanding if exist $\alpha>1$ and some Riemannian metric on $M$ such that $||Df(x)v||>\alpha||v||$ for all $x\in M$ and all $v\in T_xM$.

\subsection{Measure-theoretic entropy}

For any Borel measurable finite partition $\SP$ of $M$, and
for any (not necessarily invariant) probability $\nu$, the
\emph{entropy of $\nu$ with respect to $\SP$} is defined as
$H(\SP,\nu)=-\sum_{P\in\SP}\nu(P)\log(\nu(P)).$ Given
another finite partition $\tilde{\SP}$ of $M$ the
\emph{conditional entropy of $\SP$ with respect to
  $\tilde{\SP}$} is given by
  \begin{align*}
    H_\nu(\SP / \tilde{\SP})
    =
    \sum_{P\in\SP}\sum_{Q\in\tilde{\SP}}
    -\nu(P\cap Q)\log\frac{\nu(P\cap Q)}{\nu(Q)}.
  \end{align*}

% The \emph{conditional entropy} of partition $\SP$ given the
% partition $\SQ$ is the number
% $$H_{\nu}(\SP/\SQ)=-\sum\limits_{P\in\SP}\sum\limits_{Q\in\SQ}\nu(P\cap Q)\log\frac{\nu(P\cap Q)}{\nu(Q)}.$$

\begin{lemma}\label{lemma9.1.6}
  Given $S\geq1$ and $\varepsilon>0$ there exists $\delta>0$
  such that, for any finite partitions
  $\SP=\{P_1,\ldots, P_S\}$ and
  $\tilde{\SP}=\{\tilde{P}_1,\ldots,\tilde{P}_S\}$ such that
  $\nu(P_i\triangle\tilde{P}_i)<\delta$ for all
  $i=1,\ldots,S$, then
  $H_{\nu}(\tilde{\SP}/\SP)<\varepsilon$.
\end{lemma}
\begin{proof}
  See \cite[Lemma 9.1.6]{OV14}.
\end{proof}

Let we denote $\SP^q=\bigvee\limits_{j=0}^{q-1}f^{-j}(\SP)$,
where
$\SP\vee\SQ=\{P\cap Q\neq\emptyset: P\in\SP, \ Q\in\SQ\}$
for any pair of finite partitions $\SP$ and $\SQ$. If
$\nu\in\SM_f$, then
$$h(\SP,\nu)=\lim_{q\to+\infty}\frac1q H(\SP^q,\nu)
=\inf_{q\ge1}\frac1q H(\SP^q,\nu)$$
is the metric entropy of the partition $\SP$.

Finally, the measure-theoretic entropy $h_{\nu}(f)$ of an
$f$-invariant measure $\nu$ is defined by
$h_{\nu}(f)=\sup_{\SP} h(\SP,\nu)$, where the supremum is
taken over all the Borel measurable finite partitions $\SP$
of $M$.

We define the diameter $\diam (\SP)$ of a finite partition
$\SP$ as the maximum diameter of its atoms.
The following results will be used in the proof of the large
deviation lemma that is essential for the proof of Theorem
\ref{mthm:Toda-medida-nu-SRB-like-e-estado-de-equilibrio}.

\begin{lemma}\label{lemma10.4.4}
  Let $a_1,\ldots,a_{\ell}$ real numbers and let
  $p_1,\ldots,p_{\ell}$ non-negative numbers such that
  $\sum_{k=1}^{\ell}p_k=1$. Denote
  $L=\sum_{k=1}^{\ell}e^{a_k}$. Then
  $\sum_{k=1}^{\ell}p_k(a_k-\log p_k)\leq\log L.$ Moreover,
  the equality holds if, and only if,
  $p_k=\frac{e^{a_k}}{L}$ for all $k$.
\end{lemma}

\begin{proof}
See  \cite[Lemma 10.4.4]{OV14}.
\end{proof}

\begin{lemma}\label{lm2a}
  Let $f:M \rightarrow M$ be a measurable function. For any
  sequence of not necessarily invariant probabilites
  $\nu_n$, let
  $\mu_n:=\frac{1}{n}\sum_{j=0}^{n-1}(f^j)_*\nu_n$ and $\mu$
  be a weak$^*$ accumulation point of $(\mu_n)$. Let $\SP$
  be a finite partition of $M$ with
  $\mu(\partial\SP)=0=\mu_{n}(\partial\SP)$ for all
  $n\geq 1$. Then, for any $\varepsilon>0$, there a
  subsequence of integers $n_i\nearrow\infty$ such that
$$\frac{1}{n_i}H(\SP^{n_i},\nu_{n_i})\leq \frac{\varepsilon}{4}+h_{\mu}(f) \ \ \forall i\geq 1.$$
\end{lemma}

\begin{proof}
  Fix integers $q\geq1$, and $n\geq q$. Write $n=aq+j$ where
  $a, j$ are integer numbers such that $0\leq j\leq
  q-1$. Fix a (not necessarily invariant) probability
  $\nu$. From the properties of the entropy function $H$ of
  $\nu$ with respect to the partition $\SP$, we obtain
\begin{align*}
  H(\SP^{n},\nu)
  &=
    H(\SP^{aq + j},\nu)\leq H(\SP^{aq + q},\nu)
    \leq
    H(\bigvee\limits_{i=0}^{q-1}f^{-i}(\SP),\nu)+H(\bigvee\limits_{i=1}^{a}f^{-iq}(\SP^q),\nu)\\
  &\leq
    \sum\limits_{i=0}^{q-1}H(\SP,(f^{i})_*\nu)+\sum\limits_{i=1}^{a}H(\SP^q,(f^{iq})_*\nu)
    \leq
    q\log S
    +\sum\limits_{i=1}^{a}H(\SP^q,(f^{iq})_*\nu),
\end{align*}
for all $q\geq 1$ and $N\geq q$.  To obtain the last inequality
above recall that $H(\SP,\nu)\leq \log S$ for all
$\nu\in\SM$ where $S$ is the number of elements of the
partition $\SP$. The above inequality holds also for
$f^{-l}(\SP)$ instead of $\SP$, for any $l\geq0$, because it
holds for any partition with exactly $S$ atoms. Thus
\begin{eqnarray*}
  H(f^{-l}(\SP^{n}),\nu)
  &\leq&
         q\log S+\sum_{i=1}^{a}H(f^{-l}(\SP^q),(f^{iq})_*\nu)
         =q\log S+\sum_{i=1}^{a}H(\SP^q,(f^{l+iq})_*\nu).
\end{eqnarray*}
Adding the above inequalities for $0\leq l\leq q-1$, we
obtain on the one hand
\begin{align}\label{eq9.1a}
  \sum_{l=0}^{q-1}H(f^{-l}(\SP^{n}),\nu)
  &\leq
    q^2\log S+
    \sum_{l=0}^{q-1}\sum_{i=1}^{a}H(\SP^q,(f^{l+iq})_*\nu)
    =q^2\log S+\!\!\sum_{l=0}^{aq+q-1}H(\SP^q,(f^{l})_*\nu).
\end{align}
On the other hand, for all $0\leq l\leq q-1$,
\begin{align*}
  H(\SP^{n},\nu)
  &\leq
    H(\SP^{n+l},\nu)\leq
    H(f^{-l}(\SP^{n}),\nu)+\sum\limits_{i=0}^{l-1}H(f^{-i}(\SP),\nu)
    \le
    H(f^{-l}(\SP^q),\nu)+q\log S.
\end{align*}
Therefore, adding the above inequalities for
$0\leq l\leq q-1$ and joining with the inequality
(\ref{eq9.1a}), we obtain
$ qH(\SP^{n},\nu) \leq 2q^2\log
S+\sum_{l=0}^{aq+q-1}H(\SP^q,(f^l)_*\nu)$.

Recalling that $n= aq + j$ with $0\leq j \leq q-1$, we have
$aq+q\leq n+q$ and then
\begin{eqnarray*}
qH(\SP^{n},\nu)&\leq& 2q^2\log S+\sum\limits_{l=0}^{n-1}H(\SP^q,(f^l)_*\nu)+\sum\limits_{l=n}^{aq+q-1}H(\SP^q,(f^l)_*\nu)\\
&\leq& 3q^2\log S+\sum\limits_{l=0}^{n-1}H(\SP^q,(f^l)_*\nu),
\end{eqnarray*}
where we used that the number of non-empty pieces of $\SP^q$
is at most $S^q$. Now we fix a sequence $n_i\nearrow\infty$
such that
$\mu_{n_i}\to\mu$ in the weak$^*$ topology; put
$\nu=\nu_{n_i}$ and divide by $n_i$. Since $H$ is convex we
obtain
\begin{align*}
  \frac{q}{n_i}H(\SP^{n_i},\nu_{n_i})
  &\leq
    \frac{3q^2\log S}{n_i}+
    \frac{1}{n_i}\sum_{l=0}^{n_i-1}H(\SP^q,(f^l)_*\nu_{n_i})
    +\frac{aq^2\log S}{n_i}
  \\
  &\leq
    \frac{3q^2\log S}{n_i}+
    \frac{1}{n_i}\sum_{l=0}^{n_i-1}H(\SP^q,(f^l)_*\nu_{n_i})
    +\frac{aq^2\log S}{n_i}
    \leq
    \frac{3q^2\log
    S}{n_i}+H(\SP^q,\mu_{n_i}).%+\frac{aq^2\log S}{n_i}
\end{align*}
Therefore
$\frac{1}{n_i}H(\SP^{n_i},\nu_{n_i}) \leq
\frac{\varepsilon}{12}+\frac{1}{q}H(\SP^q,\mu_{n_i})$ for
all $i \geq i_0(q)=\max\{q, 36q\epsilon^{-1}\log S\}$.
Since $\mu_{n_i}\to\mu$ and
$\mu_n(\partial \SP)=\mu(\partial \SP)=0$ for all $n\in\NN$,
then
$\lim_{i\rightarrow+\infty}H(\SP^q,\mu_{n_i})=
H(\SP^q,\mu)$. Thus there exists $i_1>i_0(q)$ such that
$\frac{1}{q}H(\SP^q,\mu_{n_i})\leq
\frac{1}{q}H(\SP^q,\mu)+\frac{\varepsilon}{12}$ for all
$i\geq i_1$.  Hence, we obtain
$\frac{1}{n_i}H(\SP^{n_i},\nu_{n_i})\leq
\frac{\varepsilon}{6}+\frac{1}{q}H(\SP^q,\mu)$ for all
$i\geq i_1$.

Moreover, since by definition
$\lim_{q\to+\infty}\frac1q H(\SP^q,\mu)=h(\SP,\mu)$, then
there exists $q_0\in\NN$ such that
$\frac{1}{q}H(\SP^q,\mu)\leq
h(\SP,\mu)+\frac{\varepsilon}{12}$ for all $q\geq q_0$. Thus
taking $i_2:=\max\{q_0, i_1\}$ we get
\begin{eqnarray}\label{desigualdade-das-entropias}
\frac{1}{n_i}H(\SP^{n_i},\nu_{n_i})&\leq& h(\SP,\mu)+\frac{\varepsilon}{4}\leq h_{\mu}(f)+\frac{\varepsilon}{4} \ \ \textrm{for all} \ i\geq i_2.
\end{eqnarray}
%
%Moreover, by definition $\lim\limits_{q\rightarrow+\infty}\frac{1}{q}H(\SP^q,\mu)=h(\SP,\mu)$ then exists $q_0\in\NN$ such that $\frac{1}{q}H(\SP^q,\mu)\leq h(\SP,\mu)+\frac{\varepsilon}{12}\leq h_{\mu}(f)+\frac{\varepsilon}{12}$ for all $q\geq q_0$. Thus taking $n_i\geq n_{1}:=\max\{q_0, n_{i_1}\}$ we have
%
%\begin{eqnarray*}
%\frac{1}{n_i}H(\SP^{n_i},\nu_{n_i})&\leq& \frac{\varepsilon}{4}+h_{\mu}(f) \ \ \textrm{for all} \ n_i\geq n_1\\
%\end{eqnarray*}
This completes the proof using the subsequence
$(n_i)_{i\geq i_2}$ as the sequence claimed in the statement
of the lemma.
\end{proof}

% The following technical lemma will be used in the proof of
% the large deviation lemmas that are essential for showed
% the theorems
% \ref{mthm:Toda-medida-nu-SRB-like-e-estado-de-equilibrio},
% \ref{mthm:WE} and \ref{mthm:existencia de estado de
% equilibrio}.

%
%\begin{lemma}[Borel-Cantelli]
%Let $A_1,A_2,\ldots$ aleatory events in a probability space $(\Omega, \mathcal{A}, \mathbb{P})$, i.e. $A_n\in\SA$ $\forall n$.
%\begin{enumerate}
%  \item If $\sum\limits_{i=1}^{\infty}\mathbb{P}(A_n)<\infty$ then $\mathbb{P}\left(\bigcap\limits_{k=1}^{\infty}\bigcup\limits_{n=k}^{\infty}A_n\right)=0$,\\
%       \noindent
%       (note that $\bigcap\limits_{k=1}^{\infty}\bigcup\limits_{n=k}^{\infty}A_n=\{A_n \ \textrm{ocurrs for infinitely many} \ n\}$).
%
%  \item If $\sum\limits_{i=1}^{\infty}\mathbb{P}(A_n)=\infty$ and $A_n$ are independent events then  $\mathbb{P}\left(\bigcap\limits_{k=1}^{\infty}\bigcup\limits_{n=k}^{\infty}A_n\right)=1$.
%\end{enumerate}
%\end{lemma}

\section{Continuous variation of conformal measures and SRB-like measures}
\label{sec:Continuous variation of conformal measures and SRB-like measures}
Here we prove Theorem
\ref{mthm:theorem-conv-das-medidas-tipo-SRB} showing that
$\nu$-SRB-like measures can be seen as measures that
naturally arise as accumulation points of $\nu_n$-SRB
measures.

\begin{definition}\label{def-de-Jacobiano}
  A measurable function $J_{\nu}T: X\rightarrow [0,+\infty)$
  is the \emph{Jacobian of a map} $T: X\rightarrow X$\emph{
    with respect to a measure $\nu$} if for every Borel set
  $A\subset X$ on which $T$ is injective
$$\nu(T(A)) =\int_A J_{\nu}Td\nu.$$
\end{definition}

Next result guarantees the existence of measures with
prescribed Jacobian.

\begin{theorem}\label{existencia-de-medidas-conforme}
  Let $T:X\rightarrow X$ be a local homeomorphism of a
  compact metric space $X$ and let $\phi:X\rightarrow\RR$ be
  continuous. Then there exists a $\phi$-conformal
  probability measure $\nu=\nu_{\phi}$ and a constant
  $\lambda>0$, such that
  $\mathcal{L}^*_\phi \nu=\lambda \nu$. Moreover, the
  function $J_\nu T=\lambda e^{-\phi}$ is the Jacobian for
  $T$ with respect to the measure $\nu$.
\end{theorem}

\begin{proof}
  See \cite[Theorem 4.2.5]{PrzyUrb10}.
\end{proof}

We say that a continuous mapping $T:X\rightarrow X$ is open,
if open sets have open images. This is equivalent to saying
that if $f(x) = y$ and
$y_n\xrightarrow[n\rightarrow+\infty] {}y$ then there exist
$x_n \xrightarrow[n\rightarrow+\infty]{} x$ such that
$f(x_n) = y_n$ for $n$ large enough.
\begin{definition}
  A continuous mapping $T:X\rightarrow X$ is 
\begin{enumerate}
\item [(a)] \emph{topologically exact} if for all non-empty
  open set $U\subset X$ there exists $N=N(U)$ such that
  $T^{N}U = X$.
\item [(b)] \emph{topologically transitive} if for all
  non-empty open sets $U, V \subset X$ there exists
  $n\geq 0$ such that $T^n(U)\cap V \neq\emptyset$.
\end{enumerate}
\end{definition}

Topological transitiveness ensures conformal measures give
positive mass to any open subset.

\begin{proposition}\label{prop:medida-positiva-em-abertos}
  Let $T:X\rightarrow X$ be an open distance expanding
  topologically transitive map and $\phi:X\rightarrow\RR$ be
  continuous. Then every conformal measure $\nu=\nu_{\phi}$
  is positive on non-empty open sets. Moreover for every
  $r > 0$ there exists $\alpha = \alpha(r) > 0$ such that
  for every $x\in X$, $\nu(B(x,r))\geq\alpha$.
\end{proposition}
\begin{proof}
  See \cite[Proposition 4.2.7]{PrzyUrb10}
\end{proof}

%Consider now $T$, an open distance expanding topologically transitive map.

For H\"older potentials it is known that there exists a
unique $\nu$-SRB probability measure.

\begin{theorem}\label{teorema-existencia-de-medida-nu-SRB-que-sao-est-de-equilibrio-para-distance-expanding}
  Let $T:X\rightarrow X$ an open distance expanding
  topologically transitive map, $\nu=\nu_{\phi}$ a conformal
  measure associated the a H\"older continuous function
  $\phi:X\rightarrow \RR$. Then there exists a unique
  $\mu_{\phi}$ ergodic invariant $\nu$-SRB measure such that
  $\mu_{\phi}$ is an equilibrium state for $T$ and $\phi$.
\end{theorem}

\begin{proof}
  See \cite[Chapter 4]{PrzyUrb10}.
\end{proof}

The following result shows that positively invariant sets
with positive reference measure have mass uniformly bounded
away from zero.

\begin{theorem}\label{teorema-que garante-a-existencia-do-disco-no-conj-inv}
  In the same setting of Theorem
  \ref{teorema-existencia-de-medida-nu-SRB-que-sao-est-de-equilibrio-para-distance-expanding},
  if $G$ is a $T$-invariant set such that $\nu(G)>0$, then
  there is a disk $\Delta$ of radius $\delta/4$ so that
  $\nu(\Delta\backslash G) = 0$.
\end{theorem}

\begin{proof}
  See the proof of \cite[ Lemma 5.3]{VarVia09}.
\end{proof}

% \begin{remark}\label{rmk:potencial-continuo-pode-ser-aproximado-uniformemente-por-potenciais-Lipschitz}
%   It is easy to see that each continuous potential in a
%   compact metric space can be approximated in uniform
%   convergence by Lipschitz continuous potentials (in
%   particular, all Lipschitz continuous potentials are
%   $\alpha$-H\"older continuous for $0<\alpha\leq1$).

%   In fact, let
%   $\SA=\{f:X\rightarrow \RR; \ f \ \textrm{is Lipschitz
%     continuous}\}\subset C(X,\RR)$ and observe that $\SA$ is
%   an subalgebra of the algebra $C(X,\RR)$, since $\SA$ forms
%   a vector space over $\RR$ and given $f,g\in\SA$, $x\in X$
%   then $(f\cdot g)(x)=f(x)\cdot g(x)\in\SA$. Moreover,
%   $f\equiv1$ belongs to $\SA$ and $\SA$ separates points,
%   since given $x,y\in X$, $x\neq y$, we may take
%   $f:X\rightarrow\RR$ given by $f(z)=d(z,\{x\})$, $f$ is
%   Lipschitz continuous (because, $\{x\}$ is a closed set)
%   therefore $f\in\SA$ and $f(x)\neq f(y)$. Thus, the Theorem
%   of Stone-Weierstrass ensures that $\SA$ is dense in
%   $C(X,\RR)$.
% \end{remark}

Now we are ready to prove Theorem
\ref{mthm:theorem-conv-das-medidas-tipo-SRB}.

\begin{proof}[Proof of Theorem \ref{mthm:theorem-conv-das-medidas-tipo-SRB}]
  Let
  $\lambda_{n_j}\xrightarrow[j\rightarrow+\infty]{}\lambda$,
  where
  $\mathcal{L^*}_{\phi_{n_j}}(\nu_{n_j})=\lambda_{n_j}\nu_{n_j}$
  (it is easy to see that $\lambda>0$ since
  $\lambda_{n}=\mathcal{L}_{\phi_n}^*(\nu_n)(1)$ for all
  $n$).  Then for each $\varphi\in C^0(X,\RR)$ we have
\begin{eqnarray*}
  \int\varphi\,d(\mathcal{L}_{\phi_{n_j}}^*\nu_{n_j})=\int\mathcal{L}_{\phi_{n_j}}(\varphi)\,d\nu_{n_j}\xrightarrow[j\rightarrow+\infty]{}\int\mathcal{L}_{\phi}(\varphi)\,d\nu=\int\varphi\,d(\mathcal{L}_{\phi}^*\nu)
\end{eqnarray*}
thus,
$\mathcal{L}_{\phi_{n_j}}^*\nu_{n_j}\xrightarrow[j\rightarrow+\infty]{w^*}\mathcal{L}_{\phi}^*\nu$. Moreover,
$\lambda_{n_j}\nu_{n_j}\xrightarrow[j\rightarrow+\infty]{w^*}\lambda\nu$
and by uniqueness of the limit, it follows
that
$$\lambda\nu=\lim\limits_{j\rightarrow+\infty}\lambda_{n_j}\nu_{n_j}=\lim\limits_{j\rightarrow+\infty}\mathcal{L}_{\phi_{n_j}}^*\nu_{n_j}=\mathcal{L}_{\phi}^*\nu.$$
%Observação: sabemos que podemos obter uma sequência $\lambda_n=\mathcal{L}{\phi_n}^*(\nu_n)(1)$
%dai segue que $\lambda=\int\mathcal_{\phi}(1)d\nu$ logo se $\nu\neq0 segue que $\lambda\neq 0$.
Thus, $\nu$ is a $\phi$-conformal measure.

Let $\mu_{n_j}$ be the
$\nu_{n_j}$-SRB measure and let
$\mu=\lim_{j\rightarrow+\infty}
\mu_{n_j}$. For each fixed
$\varepsilon>0$, let
$N=N(\varepsilon)$ be such that
$\dist(\mu_{n_j},\mu_{n_{m}})<\frac{\varepsilon}{4}$ and
$\dist(\mu_{n_j},\mu)<\frac{\varepsilon}{4}$ for all
$j,m\geq N$.  Thus, $A_{\varepsilon/4}(\mu_{n_j})\subset
A_{\varepsilon/2}(\mu_{n_{m}})$ and
$A_{\varepsilon/2}(\mu_{n_j})\subset
A_{\varepsilon}(\mu)$ for all $j,m\geq N$.

In fact, for each $x \in A_{\varepsilon/4}(\mu_{n_j})$ we
have $\dist(p\omega(x),\mu_{n_j})<\varepsilon/4$. Then,
\begin{eqnarray*}
\dist(p\omega(x),\mu_{n_m})\leq \dist(p\omega(x),\mu_{n_j})+\dist(\mu_{n_j},\mu_{n_{m}})<\frac{\varepsilon}{2}
\end{eqnarray*}
for all $j,m\geq N$. Therefore, $A_{\varepsilon/4}(\mu_{n_j})\subset A_{\varepsilon/2}(\mu_{n_{m}})$ for all $j,m\geq N$. Analogously,  $A_{\varepsilon/2}(\mu_{n_j})\subset A_{\varepsilon}(\mu)$.

Now, from Theorem \ref{teorema-que
  garante-a-existencia-do-disco-no-conj-inv}, for each
$j\geq1$ there exists a disk $\Delta_{n_j}$ of radius
$\delta_1/4$ around $x_{n_j}$ such that
$\nu_{n_j}(\Delta_{n_j}\backslash
A_{\varepsilon/2}(\mu_{n_j}))=0$. Let
$x=\lim_{j\to+\infty}x_{n_j}$ taking a
subsequence if necessary. By compactness, the sequence
$\big(\Delta_{n_j}\big)_{j}$ accumulates on a disc
$\tilde{\Delta}$ of radius $\delta_1/4$ around $x$. Let
$\Delta$ be the disk of radius $0<s\leq\delta_1/8$ around
$x$. Thus, there exists $N_0\geq N$ such that
$\Delta\subset \Delta_{n_j}$ for all $j\geq N_0$.

Let $0<s\leq \delta_1/8$ be such that
$\nu\Big(\partial\big(\Delta\backslash
A_{\varepsilon}(\mu)\big)\Big)=0$ and note that for all
$j\geq N_0$,
$$0=\nu_{n_j}\Big(\Delta_{n_j}\backslash A_{\varepsilon/2}(\mu_{n_j})\Big)\geq\nu_{n_j}\Big(\Delta\backslash A_{\varepsilon/2}(\mu_{n_j})\Big)\geq \nu_{n_j}\Big(\Delta\backslash A_{\varepsilon}(\mu)\Big),$$
since $\Delta\subset\Delta_{n_j}$ and
$ A_{\varepsilon/2}(\mu_{n_j})\subset A_{\varepsilon}(\mu)$
for all $j\geq N_0$. Thus, by weak$^*$ convergence
\begin{align}\label{equacao2}
  \nu\Big(\Delta\setminus A_{\epsilon}(\mu)\Big)
  =\lim_{j\to+\infty}\nu_{n_j}\Big(\Delta\setminus A_{\epsilon}(\mu)\Big)=0.
\end{align}
Since $\nu$ is a conformal measure, it follows from
\eqref{equacao2} and Proposition
\ref{prop:medida-positiva-em-abertos} that
$\nu(A_{\varepsilon}(\mu))\geq \nu(\Delta)>0$. Because
$\varepsilon>0$ is arbitrary, we conclude that $\mu$ is
$\nu$-SRB-like.

Let $\SP$ be a finite Borel partition of $X$ with diameter
not exceeding an expansive constant of $T$ such that
$\mu(\partial\SP) = 0$. Then $\SP$ generates the Borel
$\sigma$-algebra for every $T$-invariant Borel probability
measure in $X$ (see \cite[Lemma 2.5.5]{PrzyUrb10}). By
Kolmogorov-Sinai Theorem, this implies that
$\eta\mapsto h_{\eta}(T)=h_\eta(T,\SP)$ is upper
semi-continuous at $\mu$.

%Thus, for every $\varepsilon >0$ there is $n_0\geq1$ such that $h_{\mu_{n_j}}(T)\leq h_{\mu}(T)+\varepsilon$ for all $j\geq n_1$.

Moreover, since
$\int\phi_{n_j}\,d\mu_{n_j}\to\int\phi\,d\mu$,
then $\mu_{n_j}$ is an equilibrium state for
$(T,\phi_{n_j})$ (by Theorem
\ref{teorema-existencia-de-medida-nu-SRB-que-sao-est-de-equilibrio-para-distance-expanding})
and by continuity of $\varphi\mapsto P_{top}(T,\varphi)$
(see \cite[Theorem 9.7]{walters2000introduction}) it follows
that
$$h_{\mu}(T)+\int\phi\,d\mu\geq \lim\limits_{j\rightarrow+\infty}\left(h_{\mu_{n_j}}(T)+\int\phi_{n_j}d\mu_{n_j}\right)=\lim\limits_{j\rightarrow+\infty}P_{top}(T,\phi_{n_j})=P_{top}(T,\phi).$$

This shows that $\mu$ is an equilibrium state for $T$ with
respect to $\phi$.

Now we assume that $T$ is topologically exact. We know that
given $\varepsilon>0$ there exists a disk $\Delta$ such that
$\nu(\Delta\setminus A_{\epsilon}(\mu))=0$. Since
$\Delta\subset X$ is a non-empty open set, there exists
$N>0$ such that $X=T^N(\Delta)$. Moreover,
$A_{\epsilon}(\mu)=T(A_{\epsilon}(\mu))$, thus
$$0=\nu(T^N(\Delta\setminus A_{\epsilon}(\mu))
\geq\nu(T^N(\Delta)\setminus A_{\epsilon}(\mu))
=\nu(X\setminus A_{\epsilon}(\mu)).$$ Since $\epsilon>0$ was
arbitrary, this completes the proof of
Theorem~\ref{mthm:theorem-conv-das-medidas-tipo-SRB}.
\end{proof}

\section{ Distance expanding maps}
\label{sec:distance-expanding-maps}

Here we state the main results needed to obtain the proof of
Theorem
\ref{mthm:Toda-medida-nu-SRB-like-e-estado-de-equilibrio}. We
start by presenting some results about distance expanding
maps that will be used throughout this section. We close the
section with the proof of Theorem
\ref{mthm:Toda-medida-nu-SRB-like-e-estado-de-equilibrio}.

\subsection{Basic properties of distance expanding open maps}
\label{sec:basic-propert-distan}

In what follows, $X$ is a compact metric space.

\begin{lemma}\label{lemma3.1.2}
  If $T: X\rightarrow X$ is a continuous open map, then for
  every $\varepsilon> 0$ there exists $\delta> 0$ such that
  $T(B(x, \varepsilon)) \supset B(T (x), \delta)$ for every
  $x\in X$.
\end{lemma}

\begin{proof}
See \cite[Lemma 3.1.2]{PrzyUrb10}.
\end{proof}

\begin{remark}\label{def 3.1.3}
  If $T:X\rightarrow X$ is a distance expanding map, then
  by \eqref{eq3.0.1} and \eqref{condicao-de-expansao}, for
  all $x\in X$, the restriction $T|_{B(x,\varepsilon)}$ is
  injective and therefore it has a local inverse map on
  $T(B(x, \varepsilon))$.

  If additionally $T:X\rightarrow X$ is an open map, then,
  in view of Lemma~\ref{lemma3.1.2}, the domain of the
  inverse map contains the ball $B(T(x),\delta)$. So it
  makes sense to define the restriction of the inverse map,
  $T^{-1}_x:B(T(x),\delta)\rightarrow B(x,\varepsilon)$ at
  each $x\in X$.
\end{remark}

\begin{lemma}\label{lemma3.1.4}
  Let $T:X\to X$ be an open distance expanding
  map. If $x\in X$ and $y, z\in B(T(x), \delta)$,
  then\footnote{Recall that $\lambda>1$ is the expansion
    rate; see~\eqref{eq3.0.1} and~\eqref{condicao-de-expansao}.}
  $d(T^{-1}_x(y),T^{-1}_x(z))\leq \lambda^{-1}d(y,z)$.  In
  particular
  $T^{-1}_x(B(T(x), \delta))\subset
  B(x,\lambda^{-1}\delta)\subset B(x,\delta)$ and
\begin{eqnarray}\label{eq3.1.3}
B(T(x), \delta)\subset T(B(x, \lambda^{-1}\delta))
\end{eqnarray}
for each small enough $\delta > 0$.
\end{lemma}
\begin{proof}
  See \cite[Lemma 3.1.4]{PrzyUrb10}.
\end{proof}

\begin{definition}\label{def3.1.5}
  Let $T:X\rightarrow X$ be an open distance expanding
  map. For every $x\in X$, $n\geq 1$ and
  $j = 0, 1, \ldots, n-1$, we write $x_j = T^{j}(x)$. In
  view of Lemma \ref{lemma3.1.4} the composition
$$T^{-1}_{x_0}\circ T^{-1}_{x_1}\circ \cdots\circ T^{-1}_{x_{n-1}}:B(T^n(x),\delta)\rightarrow X$$
is well-defined and will be denoted by $T^{-n}_x$.
\end{definition}

\begin{lemma}\label{lemma-das-propried-basicas}
  Let $T:X\rightarrow X$ be an open distance expanding
  map. For every $x\in X$ we have:
\begin{enumerate}
  \item $T^{-n}(A)=\bigcup\limits_{y\in T^{-n}(x)}T^{-n}_y(A)$ for all $A\subset B(x,\delta)$;
  \item $d(T^{-n}_x(y),T^{-n}_x(z))\leq \lambda^{-n}d(y,z)$ for all $y,z\in B(T^n(x),\delta)$;
  \item $T^{-n}_x(B(T^n(x),r))\subset B(x,\min\{\varepsilon,\lambda^{-n}r\})$ for all $r\leq \delta$.
\end{enumerate}
\end{lemma}
\begin{proof}
  The statements follow from Lemma \ref{lemma3.1.4}.
\end{proof}

For more details on the proofs in this subsection, see
\cite[Subsection 3.1]{PrzyUrb10}.

\subsection{Conformal measures}

Here we relate $\phi$-conformal measures with the
topological pressure of $\phi$. Next results ensure that
$\phi$-conformal measures for expanding dynamics with
continuous potentials are Gibbs measures; see
\cite{keller98,PrzyUrb10} for more details.

\begin{proposition}\label{proposicao1}
  Let $T:X\rightarrow X$ an open distance expanding
  topologically transitive map, $\nu=\nu_{\phi}$ a conformal
  measure associated the a continuous function
  $\phi:X\rightarrow \RR$ and $J_\nu T=\lambda e^{-\phi}$
  the Jacobian for $T$ with respect to the measure
  $\nu$. Given $\delta>0$, for all $x\in X$ and all
  $n \geq1$ there exists $\alpha(\varepsilon)>0$ such that
\begin{eqnarray*}
\alpha(\varepsilon)e^{-n\delta}\leq\displaystyle\frac{\nu((B(x,n,\varepsilon)))}{\exp{(S_n\phi(x)-Pn)}}\leq e^{n\delta}
\end{eqnarray*}
for all $\varepsilon>0$ small enough, where $P=\log\lambda$.
\end{proposition}

\begin{proof}
  Given $\delta>0$, there exist $\gamma>0$ such that, for
  all $x, y\in X$ with $d(x, y) <\gamma$, we have
  $|\phi(x)-\phi (y)|<\delta$. For any given fixed
  $0<\varepsilon<\gamma$, $x\in X$ and $n\geq 1$, we have
\begin{eqnarray}\label{igualdade*}
\nu(B(T^nx,\varepsilon))=\nu(T^n(B(x,n,\varepsilon)))=\int_{B(x,n,\varepsilon)}J_\nu T^n\, d\nu.
\end{eqnarray}
Hence, since $\nu$ attributes mass uniformly bounded away
from zero to balls of fixed radius, we obtain using the
uniform continuity of $\phi$
\begin{eqnarray*}
\alpha(\varepsilon)&\leq&\nu(B(T^nx,\varepsilon))=\int_{B(x,n,\varepsilon)}J_\nu T^n\, d\nu= \int_{B(x,n,\varepsilon)}\lambda^{n}e^{-S_n\phi}\, d\nu\\
&\leq&\lambda^{n}e^{-S_n\phi(x)+n\delta} \cdot \nu(B(x,n,\varepsilon)),
\end{eqnarray*}
and also, since $\nu$ is a probability measure
\begin{eqnarray*}
1\geq\nu(B(T^nx,\varepsilon))= \int_{B(x,n,\varepsilon)}J_\nu T^n\, d\nu\geq \lambda^{n}e^{-S_n\phi(x)-n\delta} \cdot \nu(B(x,n,\varepsilon)).
\end{eqnarray*}
Hence
$\alpha(\varepsilon)e^{-n\delta}
\le
\nu(B(x,n,\varepsilon))\exp(Pn-S_n\phi(x))\leq e^{n\delta}$
where $P=\log \lambda$.
\end{proof}

\begin{lemma}\label{lema1}
  Let $T:X\rightarrow X$ an open distance expanding
  topologically transitive map, let $\phi:X\rightarrow \RR$
  be continuous and $\nu$ a probability measure such that
  $\mathcal{L}^*_{\phi}(\nu)=\lambda\nu$. Then
  $P_{\textrm{top}}(T,\phi)\leq P=\log \lambda$.
\end{lemma}

 \begin{proof}
   Given $\delta>0$ we take $\epsilon>0$ small enough as in
   Proposition~\ref{proposicao1} and $n\geq 1$. Let
   $E_n\subset X$ be a maximal $(n,\varepsilon)$-separated
   set. Therefore $\{B(x,n,\varepsilon):\, x\in E_n\}$
   covers $X$, $\{B(x,n,\varepsilon/2):\, x\in E_n\}$ is
   a collection of pairwise disjoint open sets. Thus
\begin{eqnarray*}
  \sum_{x\in E_n}e^{S_n\phi(x)}
  \leq
  \sum_{x\in E_n}\nu(B(x,n,\varepsilon/2))\cdot
  \frac{e^{n(P+\delta)}}{\alpha(\varepsilon/2)}
  \leq
  \frac{e^{n(P+\delta)}}{\alpha(\varepsilon/2)}
\end{eqnarray*}
by Proposition~\ref{proposicao1}.  Therefore
$P_{\textrm{top}}(T,\phi)
=
\lim_{\epsilon\to0}\limsup_{n\to+\infty}\frac1n\log\sum_{x\in
  E_n}e^{S_n\phi(x)}
\leq P+\delta,$
%a prova da primeira igualdade pode ser encontrada no teorema 2.3.2 no livro de Urbansk da referencia.
and since $\delta>0$ can be taken small enough, we conclude
that $P_{\textrm{top}}(T,\phi)-P\leq0$.
\end{proof}

\subsection{Constructing an arbitrarily small initial
  partition with negligible boundary}
\label{sec:initial-partit}

Let $\nu$ be a Borel probability measure on the compact
metric space $X$.  We say that $T$ is $\nu$-\emph{regular
  map} or $\nu$ is regular for $T$ if $T_*\nu \ll \nu$, that
is, if $E\subset X$ is such that $\nu(E) = 0$, then
$\nu(T^{-1}(E))= 0$. It follows that $T$ admits a Jacobian
with respect to a $T$-regular measure. We write $\partial A$
for the topological boundary and $\interior(A)$ for the
interior of the subset $A$ of the metric space $X$. We also
write $\# A$ for the number of elements of the subset $A$.

\begin{lemma}\label{lemma: Initial partition}
  Let $X$ be a compact metric space
  % of dimension $d\geq1$.
  and $\nu$ be a regular reference measure for $T$, positive
  on non-empty open sets and let $\delta>0$ be given.
  \begin{enumerate}
  \item There exists a finite partition $\SP$ of $X$ with
    $\diam(\SP)<\delta$ such that every atom $P\in\SP$ has
    non-empty interior and $\nu(\partial\SP)=0$
  \end{enumerate}
  Let us choose one interior point $w(P)$ in each $P\in \SP$
  and form the set $C_0=\{w(P), P\in\SP\}$.
  \begin{enumerate}
  \item[(2)] If $(\nu_k)_{k\ge1}$ is a sequence of probability
    measures, then there exists a partition $\tilde\SP$ of $X$
    satisfying
    \begin{enumerate}
    \item $\diam(\tilde{\SP})<\delta$;
  \item $w\in \interior(\tilde{\SP}(w)), \ \forall w\in
    C_0$;
  \item $\#\tilde{\SP}=\#\SP$;
    % $\forall \tilde P\in\tilde{\SP}\exists P\in\SP:
    % \tilde{P}\cap P\neq\emptyset$;
  \item
    $\nu(\partial\tilde{\SP})=0=\nu_k(\partial\tilde{\SP})$
    for all $k\geq 1$; and
  \item $\nu(\SP(w)\triangle\tilde{\SP}(w))<\delta$ for
    all $w\in C_0$.
    \end{enumerate}
  \end{enumerate}
  Moreover, if additionally $T$ is a continuous open map,
  then $\nu(\partial g(P))=\nu_k(\partial g(P))=0$ for each
  $P\in\SP$ and inverse branch $g$ of $T^n$, $k,n\ge1$.
\end{lemma}

Above we write $\SP(w)$ for the atom of $\SP$
containing $w$ and similarly for $\tilde{\SP}(w)$, and also
$A\triangle B$ for the symmetrical difference
$A\setminus B\cup B\setminus A$ of $A$ and $B$. The
partition $\tilde{\SP}$ can be seen as a small perturbation
of the partition $\SP$.

\begin{proof}
  Let us fix $\delta>0$,  take $0<\delta_1\leq\delta$ and
  $\mathcal{B} =\{B(\tilde{x}_l, \delta_1/8), l =1, \ldots,
  q\}$ be a finite open cover of $X$ by $\delta_1/8$- balls
  such that $\nu(\partial B(\tilde{x}_l, \delta_1/8))=0$ for
  all $l =1, \ldots, q$. Such value of $\delta_1>0$ exists
  since the set of values of $\delta_1$ satisfying
  $\nu(\partial B(\tilde{x}_l, \delta_1/8))>0$ for some
  $l\in\{1,\ldots, q\}$ is denumerable, because $\nu$ is a
  finite measure and $X$ is
  % a compact metric space and
  % so it is 
  separable.

  From this we define a finite partition $\SP$ of $X$ as
  follows. We start by setting
\begin{align}\label{eq:geraPart}
  P_1 := B\left(\tilde{x}_1, \delta_1/8\right),
  \, P_2:=B\left(\tilde{x}_2, \delta_1/8\right)\backslash
  P_1,
  \ldots,\,
  P_k:=B\left(\tilde{x}_k, \delta_1/8\right)\backslash(P_1\cup P_2\cup\ldots\cup P_{k-1})
\end{align}
for all $k=2,\ldots, S$ where $S\leq q$.  We note that, if
$P_k\neq\emptyset$, then $P_k$ has non-empty interior (since
$X$ is separable); diameter smaller than
$2\frac{\delta_1}{8}<\frac{\delta_1}{4}$, and the boundary
$\partial P_k$ is a (finite) union of pieces of boundaries
of balls with zero $\nu$-measure. We define $\SP$ to
be the collection of elements $P_k$ constructed above which
are non-empty. This completes the proof of the first item.

Now for item (2), let $\nu_k$ be an enumerable family of
probability measures on $X$, and set
$d_0 = \min\{d(w,\partial \mathcal{P}), w\in C_0\} >
0$. Then we take $0<\eta<\min\{d_0/2,\delta_1/8\}$ (note
that $d_0$ does not depend on $k$) such that
\begin{eqnarray}\label{2222}
\nu\left(\partial
  B(\tilde x_j,\eta+\delta_1/8)\right)
  =0
  =\nu_k\left(\partial B(\tilde x_j,\eta+\delta_1/8)\right),
  \quad j=1,\dots, q; \quad k\ge1,
\end{eqnarray}
where the centers are the ones from the previous
construction.

Again we note that such value of $\eta$ exists since the set
of values of $\eta>0$ such that some of the expressions in
\eqref{2222} is positive for some $j,k$ is at most
enumerable, since the measures involved are probability
measures.  Thus we may take $\eta> 0$ satisfying
\eqref{2222} arbitrarily close to zero.

We consider now the finite open cover of $X$ given by
$
%\begin{eqnarray}\label{nova-cobertura}
  \tilde{\SB}=
  \left\{B(\tilde{x}_j, \eta+\delta_1/8):  j= 1,\ldots, q\right\}
   %  \end{eqnarray}
  $
 and construct the partition $\tilde{\SP}$ induced by
 $\tilde{\SB}$ following the same procedure as before and
 the some order of construction of the atoms of $\SP$, as
 in~\eqref{eq:geraPart}.

 We note that
 $d(w,\partial B(\tilde{x}_j,
 \eta+\delta_1/8))>d_0-\eta>d_0/2$ for all $w\in C_0$ and
 $j=1, \dots, q$ by construction. Therefore, each $w\in C_0$
 is contained in some atom $P_{w}\in\tilde{\SP}$. Moreover,
 there cannot be distinct $w,w'\in C_0$ such that
 $w'\in P_{w}$ by the choice of $\eta$. Thus, the number
 $\#\SP$ of atoms of $\SP$ is less than or equal to the
 number $\#\tilde{\SP}$ of atoms of $\tilde{\SP}$. On the
 other hand, by construction $\#\tilde{\SP}\le S=\#\SP$,
 because this is the number of elements $\#\tilde{\SB}=S$ of
 the open cover $\tilde{\SB}$. In this way we conclude that
 the partition $\tilde{\SP}$ has the same number of atoms as
 $\SP$, that is, we have proved item (c).

 Moreover, it is now easy to see by the choice of $\eta$
 that $\tilde{\SP}$ satisfies also (a), (b) and (d). In
 addition, every element of $\tilde{\SB}$ is in a
 $\eta$-neighborhood of the corresponding element of
 $\SB$. Since we can take $\eta>0$ smaller without loosing
 any of the previous conclusions, we can also obtain subitem
 (e).  This completes the proof of item (2).

Finally, if $T$ is open and continuous, then inverse
branches of $T$ are well defined by
Lemma~\ref{lemma3.1.2}. Since $T$ is $\nu$-regular and
$\partial g(P)=g(\partial P)$, then boundary of $g(P)$ still
has zero $\nu$- and $\nu_k$-measure for every atom $P\in\SP$
and every inverse branch $g$ of $T^n$, for each
$k,n\geq1$.

The proof of Lemma~\ref{lemma: Initial partition} is complete.
\end{proof}

\subsection{Expansive maps and existence of equilibrium states}
\label{sec:expans-maps-existenc}

A continuous transformation $T:X\rightarrow X$ of a compact
metric space $X$ equipped with a metric $\rho$ is said to be
(positively) expansive if and only if
\begin{align*}
  \exists\delta>0\quad\text{such that}\quad
  [\rho(T^n(x), T^n(y))\leq\delta,\, \forall n\geq 0]
  \implies
  x = y
\end{align*}
and the number $\delta$ above is an \emph{expansive
  constant} for $T$.

\begin{theorem}\label{semicontinuidade-para-positively-expansive-maps}
  If $T:X\rightarrow X$ is positively expansive, then the
  function $\SM_T\ni\mu\mapsto h_{\mu}(T)$ is upper
  semi-continuous and each continuous potential
  $\phi:X\rightarrow\RR$ has an equilibrium state.
\end{theorem}
\begin{proof}
See \cite[Theorem 2.5.6]{PrzyUrb10}.
\end{proof}

\begin{theorem}\label{distance-expanding-implica-positively-expansive-maps}
  Every distance expanding map is positively expansive.
\end{theorem}
\begin{proof}
See \cite[Theorem 3.1.1]{PrzyUrb10}.
\end{proof}

\begin{corollary}
  If $T:X\rightarrow X$ is distance expanding, then each
  continuous potential $\phi:X\rightarrow\RR$ has an
  equilibrium state.  In particular,
  $\SK_r(\phi)=\{\mu\in\SM_T: h_{\mu}(T)+\int\phi d\mu\geq
  P_{\textrm{top}}(T,\phi)-r\}\neq \emptyset$ for all
  $r\geq0$ .
\end{corollary}

\begin{proof}
  The proof is immediate from Theorem
  \ref{distance-expanding-implica-positively-expansive-maps}
  and Theorem
  \ref{semicontinuidade-para-positively-expansive-maps}.
\end{proof}

\subsection{Large deviations}
\label{sec:large-deviat-lemma}

The statement of
Theorem~\ref{mthm:Toda-medida-nu-SRB-like-e-estado-de-equilibrio}
is a consequence of the following more abstract result,
inspired in \cite[Lemma 4.3]{CatsEnrich2012} and in
\cite[Proposition 6.1.11]{keller98}.

\begin{proposition}
  \label{Lm:Large deviation lemma for distance expanding}
  Let $T:X\rightarrow X$ be an open distance expanding
  topologically transitive map and let
  $\phi :X\rightarrow\RR$ be continuous. Fix
  $\nu=\nu_{\phi}$ a conformal measure, $r > 0$ and consider
  the weak$^*$ distance defined in (\ref{eq13}). Then, for
  all $0 < \varepsilon < r$, there exists $n_0\geq 1$ and
  $\kappa>0$ such that
\begin{align}\label{eq151}
  \nu\Big(\big\{x\in X; \dist (\sigma_n(x),
  \mathcal{K}_r(\phi))\geq\varepsilon\big\}\Big)
  <
  \kappa\exp{\left[n(\varepsilon-r)\right]},\ \ \forall n\geq n_0.
\end{align}
\end{proposition}

\begin{proof}
  We know by Theorem \ref{existencia-de-medidas-conforme}
  and Proposition \ref{prop:medida-positiva-em-abertos} that
  all the (necessarily existing) conformal measures $\nu$
  are positive on non-empty open sets and
  $J_{\nu}T=\lambda e^{-\phi}$ is the Jacobian for $T$ with
  respect to the measure $\nu$.

  Let $\nu$ be a conformal measure, fix $r>0$ and let
  $0<\varepsilon<r$. For $\varepsilon /6$, fix a constant
  $\gamma>0$ of uniform continuity of $\phi$, i.e.,
  $|\phi(x)-\phi(y)|<\frac{\varepsilon }{6}$ whenever
  $d(x,y)<\gamma$. Let us fix $0<\xi<\gamma$ and a partition
  $\SP$ of $X$ as in Lemma~\ref{lemma: Initial partition}
  such that $\diam(\SP)<\frac{\xi}{4}$ and choose one
  interior point in each atom of $\SP$ to construct
  $C_0=\{w_1,\ldots,w_S\}$ where $S=\#\SP$.

  % Set \marginpar{needed?}
  % $d_0 = \min\{d(w,\partial \mathcal{P}), w\in C_0\} > 0$
  % where
  % $\partial \mathcal{P} = \bigcup_{l=1}^{S}\partial P_l$ is
  % the boundary of $\mathcal{P}$.

  Define
  $\mathcal{A}:=\{\mu\in\mathcal{M}_T; d(\mu,
  \mathcal{K}_r(\phi))\geq \varepsilon\}$ and note that
  $\mathcal{A}$ is weak* compact, so it has a finite
  covering $B_1,\ldots, B_{\kappa}$ for minimal cardinality
  $\kappa\geq1$, with open balls $B_i\subset \mathcal{M}$ of
  radius $\frac{\varepsilon}{3}$. For any fixed $n\geq 1$
  write $C_{n,i}=\{x\in X: \sigma_n(x)\in B_i\}$,
  $C_n =\bigcup_{i=1}^{\kappa}C_{n,i}$,
  $\tilde{C}_{n,i}=\{x\in X: \sigma_n(x)\in \tilde{B}_i\}$
  and $\tilde{C}_n =\bigcup_{i=1}^{\kappa}\tilde{C}_{n,i}$,
  where $\tilde{B}_i$ are open balls concentric with $B_i$
  of radius $\frac{2\varepsilon}{3}$ for
  $i=1,\ldots, \kappa$.

  We note that $C_{n,i}\subset\tilde{C}_{n,i}$ and,
  moreover,
  $\{x\in X:
  d(\sigma_n(x),\mathcal{K}_r(\phi))
  \geq\varepsilon\}
  \subset
  C_{n}\subset \tilde{C}_{n}$.

\begin{lemma}\label{af111}
  For each $1\leq i\leq \kappa$ there exists $n_i>0$ such
  that
  $\nu(C_{n,i})\leq
  \exp{[n(\varepsilon-r)]}$ for all
  $n\geq n_i$.
\end{lemma}

First, let us see that it is enough to prove
Lemma~\ref{af111} to finish the proof of
Proposition~\ref{Lm:Large deviation lemma for distance
  expanding}, and then we prove Lemma~\ref{af111} in the
following Subsection~\ref{sec:exponent-upper-bound}. In
fact, if Lemma~\ref{af111} holds, then put
$n_0=\max\limits_{1\leq i\leq \kappa}n_i$
% \max\left\{\frac{2\log k}{\varepsilon},\max\limits_{1\leq i\leq k}n_i\right\}$.
and we get
\begin{align*}
\nu(C_{n})\leq\sum\limits_{i=1}^{\kappa}\nu(C_{n,i})\leq
\kappa \exp{\left[n\left(\varepsilon-r\right)\right]}
%=\exp{\left[n\left(-r+\frac{\varepsilon}{2}+\frac{\log k}{n}\right)\right]}\leq \exp{\left[n\left(\varepsilon-r\right)\right]}
\end{align*}
for all $n\geq
n_0$, as needed to conclude the proof of the statement of
Proposition~\ref{Lm:Large deviation lemma for distance
  expanding}.
\end{proof}

\subsubsection{Exponential upper bound}
\label{sec:exponent-upper-bound}
Now we present the proof of Lemma~\ref{af111}.

\begin{proof}[Proof of Lemma~\ref{af111}]For $x\in C_{n,i}$ let
$P\in\SP$ be the atom such that $T^n(x)\in P$ and set
$Q=T_x^{-n}(P)$. Then the family $\SQ_n$ of all such sets
$Q$ is finite since both $\SP$ and the number of inverse
branches are finite.  Moreover, by the expression of
$J_{\nu}T$ in terms of $\phi$
%given by (\ref{igualdade*}) in the proof of Theorem \ref{proposicao1}
\begin{eqnarray*}
  \nu(Q\cap C_{n,i})
  =
  \int_{T^n(Q\cap C_{n,i})}J_{\nu}T^{-n}d\nu
  =
  \int_{T^n(Q\cap C_{n,i})}\exp
  \left[\sum\limits_{j=0}^{n-1}\phi\circ
  T^j-n\log\lambda\right]
  \circ T_x^{-n}\,d\nu.
\end{eqnarray*}
We observe that if $\nu(C_{n,i})=0$, then Lemma~\ref{af111}
becomes trivially proved. Consider the finite family of
atoms
$\{Q_1,\ldots, Q_N\}=\{Q\in\SQ_n: \nu(Q\cap C_{n,i})>0\}$
which has $N=N(n,i)$ elements for some $N\geq 1$.

We note that
$\nu(C_{n,i})=\sum_{k=1}^{N}\nu(Q_k\cap
C_{n,i})$. For each $k=1,\ldots, N$,
let us take %$y_k\in int(X_k\cap C_{n,i})$ and let
$x_k\in Q_k$ such that $T^n(x_k)=w_j$ for some
$j=1,\ldots,S$. We recall that $w_j$ are interior points of
each atom of the partition $\SP$, so there is only one
$j=j_{k,n}$ for each $x_k\in Q_k$ such that $T^n(x_k)=w_j$.

Since $\diam(\SP^n)<\diam(\SP)<\frac{\xi}{4}<\xi$ for all
$n>0$ (recall Lemma \ref{lemma-das-propried-basicas}),
then $|\phi(T^j(x_k))-\phi(T^j(y))|<\frac{\varepsilon }{6}$
for all $y\in Q_k$ and $j=0,\ldots, n-1$. Choosing
$y_k\in Q_k\cap C_{n,i}$ for each $k$, then we get
\begin{align*}
  \nu(C_{n,i})
  &=
    \sum_{k=1}^{N}\nu(Q_k\cap C_{n,i})
    \leq
    \sum_{k=1}^{N}\int\limits_{T^n(Q_k\cap C_{n,i})}
    \exp
    \left[\sum_{j=0}^{n-1}\big(\phi(T^j(y_k))+\frac{\varepsilon}{6}\big)
    -n\log\lambda\right]\,d\nu
  \\
  &\leq
    \sum_{k=1}^{N}
    \exp
    \left[\sum_{j=0}^{n-1}\big(\phi(T^j(x_k))+
    \frac{\varepsilon}{6}
    +\frac{\varepsilon}{6}\big)-n\log\lambda\right]
    \cdot\nu(T^n(Q_k\cap C_{n,i}))
  \\
  &\leq
    % \sum_{k=1}^{N}
    % \exp
    % \left[\sum_{j=0}^{n-1}\big(\phi(f^j(x_k))+\frac{\varepsilon}{3}\big)
    % -n\log\lambda\right]
    % =
    \exp
    \left[n\big(\frac{\varepsilon}{3}-\log\lambda\big)\right]
    \sum_{k=1}^{N}
    e^{ S_n\phi(x_k)}.
%    \sum\limits_{j=0}^{n-1}\phi(T^j(x_k))\right].
\end{align*}
Defining $L_N:= \sum_{k=1}^{N}e^{S_n\phi(x_k)}$ and
$\lambda_k:=L_N^{-1}e^{S_n\phi(x_k)}$, then
$\sum_k\lambda_k=1$ and also
$ \log
L_N=\sum_{k=1}^{N}\lambda_k\big(S_n\phi(x_k)-\log\lambda_k\big)$
by Lemma~\ref{lemma10.4.4}. Then
$
  \nu(C_{n,i})
  \le
  \exp\big[n\big(\frac{\varepsilon}{3}-\log\lambda
  +\frac1n\log L_N\big)\big].$
  
  % \begin{lemma}
%   \label{le:thermodyn}
%   Let $\nu$ be a $\phi$-conformal measure with respect to
%   the open continuous map $T:X\to X$ and $X_{n_k}$ a
%   sequence of subsets such that, for some $\epsilon>0$, we
%   have
%   \begin{align*}
%   \nu(X_{n_k})\le
%   \zeta_k\exp\big[n_k(\epsilon-\log\lambda+n_k^{-1}\log
%   L_k)\big]
%   \end{align*}
% for each $k\ge1$, where $n_k\nearrow\infty$,
%   $\zeta_k/n_k\to0$ and
%   $L_k=\sum_{i=1}^{N_k}\lambda_i\big(S_{n_k}\phi(x_i)-\log\lambda_i\big)$
%   and also $\lambda_i:=L_k^{-1}e^{S_{n_k}\phi(x_i)}$ for
%   points $x_i\in X$, $i=1,\dots, N_k$.
%   Then there are $\mu\in\SM_T\setminus\SK_r(\phi)$ and
%   $k_0\ge1$ so that
%   $\nu(X_{n_k})\le \exp[ n_k(\epsilon
%   -\log\lambda+\int\phi\,d\mu+h_\mu(T))]$ for all
%   $k\ge k_0$.
% \end{lemma}

Define the probability measures
\begin{align*}
  \nu_n:=\sum_{k=1}^{N}\lambda_k\delta_{x_k}
  \qand
  \mu_n:=\frac{1}{n}\sum_{j=0}^{n-1}(T^j)^*(\nu_n)
  =\sum_{k=1}^{N}\lambda_k\sigma_n(x_k)
\end{align*}
so that we may rewrite
$\log L_N=n\int\phi d\mu_n+H(\SP^n,\nu_n)$
since the atoms of $\SP^n$ contain at most one $x_k$. We fix a
weak$^*$ accumulation point $\mu$ of $(\mu_n)_n$ and take a
subsequence $n_j\nearrow\infty$ such that $\mu_{n_j}\to\mu$
in the weak$^*$ topology and also
\begin{eqnarray}\label{ls11}
  \limsup_{n\rightarrow+\infty}\frac{1}{n}\log
  \nu(C_{n,i})
  =\lim_{j\rightarrow+\infty}\frac{1}{n_j}\log \nu(C_{n_j,i}).
\end{eqnarray}

%Vamos agora fazer uma pequena perturbação da partição original, de modo que os pontos $y_k,x_k\in X_k$ ainda pertençam ao mesmo atomo até o $n-$esimo refinamento da partição pertubada.

We now use item (2) of Lemma~\ref{lemma: Initial partition}
with $\nu=\mu$ and $\nu_k=\mu_k, k\ge1$ to perform a small
perturbation of the original partition $\SP$ into
$\tilde{\SP}$, so that the points $x_k \in Q_k$ are still
given by the image of the same $n$ th inverse branch of $T$
of an atom of the perturbed partition $\tilde{\SP}$, and the
boundaries of $\tilde{\SP}$ also have negligible $\mu$- and
$\mu_n$-measure.

Note that
$H(\tilde{\SP}^{n_j},\nu_{n_j})=H(\SP^{n_j},\nu_{n_j})$ by
definition of the $\nu_{n_j}$ and by construction of
$\tilde{\SP}$ as a perturbation of $\SP$. Using items (2a)
and (2d) of Lemma~\ref{lemma: Initial partition} we get, by
Lemma \ref{lm2a}, that there exists $j_0>0$ such that
\begin{eqnarray}\label{443}
  \frac{1}{n_j}H(\SP^{n_j},\nu_{n_j})
  =
  \frac{1}{n_j}H(\tilde{\SP}^{n_j},\nu_{n_j})
  \leq
  h_{\mu}(T)+\frac{\varepsilon}{3}, \ \ \forall j\geq j_0.
\end{eqnarray}
For the partition $\tilde{\SP}$, we have that
$x_k\in \tilde{Q}_{k}\cap Q_{k}$ and
$Q_{k}\cap C_{n,i}\neq\emptyset$, where
$\tilde{Q}_{k}=\tilde{\SP}^n(x_k)$ for $k=1,\ldots, N(n,i)$.

Let us choose $y\in Q_k\cap C_{n,i}$. Then
$\sigma_{n}(y)\in B_i$. Since
$d(T^j(x_k),T^j(y))\leq \diam(\SP)<\xi$ for all
$j=0,\ldots, n-1$, then
$d(\sigma_n(x_k),\sigma_n(y))<\varepsilon/3$ by Lemma
\ref{lm1}. Moreover, as $B_i\subset\tilde{B}_i$
concentrically, we obtain
$\sigma_n(x_k)\in \tilde{B}_i$.

In addition, since the ball $\tilde{B}_i$ is convex and $\mu_n$ is a
convex combination of the measures $\sigma_n(x_k)$ (recall
that $\sum\lambda_k=1$), we deduce that
$\mu_n\in\tilde{B}_i$.

Therefore, the weak$^*$ limit $\mu$ of any convergent
subsequence of $\{\mu_n\}_{n}$ belongs to the weak$^*$
closure $\overline{\tilde{B}}_i$. Since the ball
$\tilde{B}_i$ has radius $\frac{2\varepsilon}{3}$ we
get that $\mu\in\SM_T\setminus\SK_r(\phi)$. Then
$\int\phi d\mu+h_{\mu}(T)<P_{\textrm{top}}(T,\phi)-r$, and
hence
\begin{eqnarray*}
\nu(C_{n,i})&\leq& \exp{\left[n\big(\frac{\varepsilon}{3}-\log\lambda\big)\right]}\cdot L=\exp{\left[n\big(\frac{\varepsilon}{3}-\log\lambda\big)+\log L\right]}\\
&=&\exp{\left[n\big(\frac{\varepsilon}{3} -\log\lambda+\int\phi d\mu_n+\frac{1}{n}H(\SP^n,\nu_n)\big)\right]}.
\end{eqnarray*}
Because
$\int\phi d\mu_{n_j}\xrightarrow[j\rightarrow\infty]{}
\int\phi d\mu$ (by weak$^*$ convergence), there exists
$j_1>0$ such that
$\int\phi d\mu_{n_j}\leq\int\phi d\mu+\varepsilon/3$ for all
$j>j_1$.  By (\ref{443}) there exists $j_0>0$ such that
$\frac{1}{n_j}H(\SP^{n_j},\nu_{n_j})\leq
h_{\mu}(T)+\varepsilon/3$, $\forall j\geq j_0$. Taking
$j_2=\max\{j_1,j_0\}$ we have for all $j>j_2$
\begin{align*}
  \nu(C_{n_j,i})
  &\leq
  \exp{\left[
      n_j\big(\frac{\varepsilon}{3}-\log\lambda+\int\phi
      d\mu +\frac{\varepsilon}{3}+
      h_{\mu}(T)+\frac{\varepsilon}{3}\big)
    \right]}
  \\
  &=
  \exp{\left[n_j\big(\varepsilon-\log\lambda+\int\phi d\mu +
      h_{\mu}(T)\big)\right]}
  \\
  &\leq
  \exp{\left[
      n_j\big(\varepsilon-r
      -\log\lambda+P_{\textrm{top}}(T,\phi)\big)
    \right]}
  \leq \exp{\left[n_j\big(\varepsilon-r\big)\right]},
\end{align*}
where the last inequality follows from Lemma \ref{lema1}.

Finally, by the choice of $n_j$ satisfying~\eqref{ls11}, we
conclude that there exist $n_0>0$ such that
$\nu(C_{n,i})\le \exp[n(\varepsilon-r)]$ for all
$n\geq n_0$. This completes the proof of Lemma~\ref{af111}.
\end{proof}

\subsection{Proof of Theorem \ref{mthm:Toda-medida-nu-SRB-like-e-estado-de-equilibrio}}

Given $r>0$, consider the (non-empty) set
$\mathcal{K}_r(\phi)\subset \SM_T$. By the upper
semicontinuity of the metric entropy (see Theorem
\ref{semicontinuidade-para-positively-expansive-maps}), we
have that $\SK_r(\phi)$ is closed, hence, weak$^*$
compact. Since $\{\mathcal{K}_r(\phi)\}_{r}$ is decreasing
with $r$, we have
$\mathcal{K}_0(\phi)=\bigcap_{r>0}\mathcal{K}_r(\phi)$.

By the Variational Principle
$h_{\mu}(T)+\int\phi \,d\mu\leq P_{\textrm{top}}(T,\phi)$
for all $\mu\in\SM_T$. So, to prove Theorem
\ref{mthm:Toda-medida-nu-SRB-like-e-estado-de-equilibrio},
we must prove that the set $\SW_T(\nu)$ of $\nu$-SRB-like
measures satisfy $\SW_T(\nu) \subset \mathcal{K}_r(\phi)$
for all $r > 0$, because
$\mathcal{K}_0(\phi)=\big\{\mu\in\mathcal{M}_T: \
h_{\mu}(T)+\int\phi
d\mu=P_{\textrm{top}}(T,\phi)\big\}$. Since
$\mathcal{K}_r(\phi)$ is weak$^*$ compact, we have

$$\mathcal{K}_r(\phi)=\bigcap_{\epsilon>0}\mathcal{K}^{\epsilon}_r(\phi),
\ \ \textrm{where} \ \mathcal{K}^{\epsilon}_r(\phi)
=\big\{\mu\in\mathcal{M}_T: \ \dist(\mu,\mathcal{K}_r(\phi))\leq \varepsilon\big\}$$
with the weak$^*$ distance defined in
(\ref{eq13}). Therefore, it is enough to prove that
$\SW_T(\nu)\subset \mathcal{K}^{\varepsilon}_r(\phi)$ for
all $0 < \varepsilon < r/2$ and for all $r>0$. By
Proposition \ref{proposition:Wf-compacto} and since
$\mathcal{K}^{\varepsilon}_r(\phi)$ is weak$^*$ compact, it
is enough to prove the following
\begin{lemma}
  The basin of attraction of
  $\mathcal{K}^{\varepsilon}_r(\phi)$
$$W^s(\mathcal{K}^{\varepsilon}_r(\phi)):=\big\{x\in X; p\omega(x)\subset \mathcal{K}^{\varepsilon}_r(\phi)\big\}.$$
has full $\nu$-measure:
$\nu\big(X\setminus
W^s(\mathcal{K}^{\varepsilon}_r(\phi))\big)=0$.
\end{lemma}

\begin{proof}
  By Lemma~\ref{Lm:Large deviation lemma for distance
    expanding}, the subset
  $X_n=X_n(\epsilon,r)=\{x\in X: \sigma_n(x)\in
  \SM\setminus\mathcal{K}^{\varepsilon}_r(\phi)\}$ satisfies
  $\nu(X_n)\le \kappa e^{n(\varepsilon-r)}$ for some
  $\kappa>0$ any $n>n_0$, where $n_0=n_0(\epsilon)\ge1$,
  since $0<\varepsilon<r$.  This implies that
  $\sum_{n=1}^{+\infty} \nu(X_n)<\infty$. By
  the Borel-Cantelli Lemma, it follows that
$\nu\left(\cap_{n_r=1}^{\infty} \cup_{n=n_r}^{\infty} X_n\right)=0$.

In other words, for $\nu$-a.e. $x\in X$ there exists
$n_0\geq1$ such that
$\sigma_n(x)\in \mathcal{K}^{\varepsilon}_r(\phi)$ for all
$n\geq n_0$. Hence,
$p\omega(x)\subset \mathcal{K}^{\varepsilon}_r(\phi)$ for
$\nu$-almost all the points $x\in X$, as required.
\end{proof}

The proof of Theorem \ref{mthm:Toda-medida-nu-SRB-like-e-estado-de-equilibrio} is complete.

\begin{remark}\label{obs:medidas SRB-like sao quase estado de equilibrio}
  If the set $\SK_r(\phi)$ is not closed, we may substitute
  $\overline{\SK_r(\phi)}$ for $\SK_r(\phi)$ in the proof of
  Theorem
  \ref{mthm:Toda-medida-nu-SRB-like-e-estado-de-equilibrio},
  and by the same argument we conclude that
  $\SW_T(\nu)\subset\cap_{r>0}
  \overline{\SK_r(\phi)}$. Thus, in a more general context,
  where Propsition~\ref{Lm:Large deviation lemma for
    distance expanding} is valid and
  $\SK_{1/n}(\phi)\neq\emptyset$ for all $n\geq1$, we can
  say that $\nu$-SRB-like measures are "almost
  $\phi$-equilibrium states". Indeed, given
  $\mu \in \SW_f(\nu)$, then
  $\mu=\lim_{n\rightarrow+\infty}\mu_n$,
  $\mu_n\in\SK_{1/n}(\phi)$ for all $n\geq1$. Therefore, we
  can find a sequence of $T$-invariant probability measures
  so that
  $h_{\mu_n}(T)+\int\phi d\mu_n\geq P(T,\phi)-\frac{1}{n}$
  for all $n\geq 1$ and $\mu_n \to \mu$ in the weak*
  topology.
\end{remark}

To obtain a $\phi$ equilibrium state in the limit we need
only assume that $\phi$ is uniformly approximated by
continuous potentials, as follows.

\begin{corollary}
Let $T:X\rightarrow X$ be an open distance expanding topologically transitive map of a compact metric space $X$, $(\phi_n)_{n\geq1}$ a sequence of continuous potentials, $(\nu_n)_{n\geq1}$ a sequence of conformal measures associated to the $(T,\phi_n)$ and $\mu_n$ a sequence of $\nu_n$-SRB-like measures. Assume that
\begin{enumerate}
  \item $\phi_{n_j}\xrightarrow[j\rightarrow+\infty]{}\phi$ in the topology of uniform convergence;
  %\item $\nu_{n_j}\xrightarrow[j\rightarrow+\infty]{w^*}\nu$
  \item $\mu_{n_j}\xrightarrow[j\rightarrow+\infty]{w^*}\mu$ in the weak$^*$ topology.
 % \item For all $\varepsilon>0$, $\nu_n(A_{\varepsilon}(\mu_n))=1$
\end{enumerate}
Then $\mu$ is an equilibrium state for the potential $\phi$.
\end{corollary}

\begin{proof}
  Let $\mu_{n_j}$ be a $\nu_{n_j}$-SRB-like measure and let
  $\mu=\lim_{j\rightarrow+\infty} \mu_{n_j}$. Since any
  finite Borel partition $\SP$ of $X$ with diameter not
  exceeding an expansive constant and satisfying
  $\mu(\partial\SP) = 0$ generates the Borel
  $\sigma$-algebra for every Borel $T$-invariant probability
  measure in $X$ (see \cite[Lemma 2.5.5]{PrzyUrb10}), then
  Kolmogorov-Sinai Theorem implies that
  $\eta\mapsto h_{\eta}(T)=h_\eta(T,\SP)$ is upper
  semi-continuous.

  Moreover, since
  $\int\phi_{n_j}\,d\mu_{n_j}\to\int\phi\,d\mu$,
  $\mu_{n_j}$ is an equilibrium state for $(T,\phi_{n_j})$
  (by Theorem
  \ref{mthm:Toda-medida-nu-SRB-like-e-estado-de-equilibrio})
  and by continuity of $\varphi\mapsto P_{top}(T,\varphi)$
  (see \cite[Theorem 9.7]{walters2000introduction}) it
  follows that
$$h_{\mu}(T)+\int\phi\,d\mu\geq \lim\limits_{j\rightarrow+\infty}\left(h_{\mu_{n_j}}(T)+\int\phi_{n_j}d\mu_{n_j}\right)=\lim\limits_{j\rightarrow+\infty}P_{top}(T,\phi_{n_j})=P_{top}(T,\phi).$$

This shows that $\mu$ is an equilibrium state for $T$ with
respect to $\phi$.
\end{proof}

\section{Entropy Formula}
\label{sec:Pesin's Entropy Formula}

Here we state the main results needed to obtain the proof of Theorem \ref{mthm:formula de pesin para difeo-local com HT}. Then we prove Theorem \ref{mthm:formula de pesin para difeo-local com HT} in the last subsection.

\subsection{Hyperbolic Times}

The main technical tool used in the study of non-uniformly expanding maps is the notion of hyperbolic times, introduced in \cite{alves2000srb}. We now outline some the properties of hyperbolic times.

\begin{definition}
Given $\sigma\in(0,1)$, we say that $h$ is a $\sigma$-hyperbolic time for a point $x\in M$ if for all $1\leq k \leq h$,
\begin{eqnarray}\label{th}
\prod\limits_{j=h-k}^{h-1}\|Df(f^j(x))^{-1}\|\leq \sigma^k
\end{eqnarray}
\end{definition}

\begin{remark}
  Throughout this section we cite results originally proved
  under the assumption that $f$ is of class $C^{2}$, or
  $f\in C^{1+ \alpha}(M,M)$ for some $0<\alpha<1$. But the
  cited results were proved without using the bounded
  distortion property, and therefore the proofs are easily
  adapted to our setting. Indeed, most of proofs are the
  same.
\end{remark}

\begin{proposition}\label{prop3.1}
  Given $0<\sigma< 1$, there exists $\delta_1 >0$ such that,
  whenever $h$ is a $\sigma$-hyperbolic time for a point
  $x$, the dynamical ball $B(x,h,\delta_1)$ is mapped
  diffeomorphically by $f^h$ onto the ball
  $B(f^h(x),\delta_1)$, with
\begin{align}\label{eq:backcontraction}
d(f^{h-k} (y), f^{h-k} (z))\leq \sigma^{k/2}\cdot d(f^h(y),
  f^h(z))
\end{align}
for every $1\leq k \leq h$ and $y,z\in B(x,h,\delta_1)$.
%such that if $h$ is a $\sigma$-hyperbolic time for $x$, then there are hyperbolic preballs $V_k(x)$ which are neighborhoods of $f^{h-k}(x)$, $k = 1,\ldots, h$, such that
%\begin{enumerate}
%  \item $f^k|_{V_k(x)}$ maps $V_k(x)$ diffeomorphically to the ball of radius $\delta_1$ around $f^h(x)$;\\
%  \item for every $1\leq k \leq h$ and $y,z\in V_k(x)$
%
%$$ d(f^{h-k} (y), f^{h-k} (z))\leq \sigma^{k/2}\cdot d(f^h(y), f^h(z)).$$
%\end{enumerate}
\end{proposition}
\begin{proof}
See Lemma 5.2 in \cite{ABV00}
%prova pode ser encontrada com detalhes tbm no livro de Jose Ferreira Alves pag 24 (proposição 2.3). Ou na referência acima citada
\end{proof}

%\begin{remark}\label{rmk:medidatotaldostemposhiperbolicos}
%Let $\SG$ the set of points $x\in M$ that have no hyperbolic time. Then
%$\Leb(\SG) = 0$ by assumption. Thus, if $x$ has only finitely many hyperbolic times,
%then some iterate of $x$ belongs to $\SG$. Hence, the subset of points with finitely many
%hyperbolic times is contained in $\bigcup\limits_{j\geq0}f^{-j}(\SG)$. Moreover, $\Leb(f^{-j}(\SG)) = 0$ because $f$ is a local diffeomorphism. Therefore, $\Leb$-a.e. $x\in M$ has infinitely many hyperbolic times.
%\end{remark}

\begin{remark}\label{rmk0}
  For an open distance expanding and topologically
  transitive map $T$ of a compact metric space $X$, every
  time $h\ge1$ satisfies~\eqref{eq:backcontraction} for
  every $x\in X$.
\end{remark}

\begin{definition}
We say that the frequency of $\sigma$-hyperbolic times for $x\in M$ is positive, if there is some $\theta> 0$ such that all sufficiently for large $n\in\NN$ there are $l\geq \theta n$ and integers $1\leq h_1 < h_2<\ldots< h_l\leq n$ which are $\sigma$-hyperbolic times for $x$.
\end{definition}

The following Theorem ensures existence of infinitely many
hyperbolic times Lebesgue almost every point for
non-uniformly expanding maps. A complete proof can be found
in \cite[Section 5]{ABV00}.

\begin{theorem}\label{prop 2.12}
  Let $f:M \rightarrow M$ be a $C^1$ non-uniformly expanding
  local diffeomorphism. Then there are $\sigma\in (0,1)$ and
  there exists $\theta =\theta(\sigma)>0$ such that
  $\Leb$-a.e. $x\in M$ has infinitely many
  $\sigma$-hyperbolic times. Moreover, if we write
  $0 < h_1 < h_2 < h_3 <\ldots$ for the hyperbolic times of
  $x$, then their asymptotic frequency satisfies
  $\liminf_{N\rightarrow\infty} \#\{k\geq1: h_k\leq
  N\}/N\geq\theta$, for $\Leb-a.e. x\in M$.
\end{theorem}

%para a prova veja Lema 2.11 (pag 29 do livro de J.F. Alves) e a proposição 2.13

The Lemma below shows that we can translate the density of
hyperbolic times into the Lebesgue measure of the set of
points which have a specific (large) hyperbolic time.

\begin{lemma}\label{lemma3.3}
  Let $B\subset M$, $\theta > 0$ and $g:M\rightarrow M$ be a
  local diffeomorphisms such that $g$ has density $>2\theta$
  of hyperbolic times for every $x\in B$. Then, given any
  probability measure $\nu$ on $B$ and any $n \geq 1$, there
  exists $h > n$ such that
  $\nu(\{x\in B : h \ \textrm{is a hyperbolic time of $g$
    for $x$}\})>\theta/2.$
\end{lemma}

\begin{proof}
See \cite[Lemma 3.3]{AP06}
\end{proof}

%The next Lemma is easily adapted some statements of our interest.
The next result is the flexible covering lemma with hyperbolic preballs which will enable us to approximate the Lebesgue measure of a given set through the measure of families of hyperbolic preballs.

\begin{lemma}\label{lemma3.5}
  Let a measurable set $A\subset M$, $n \geq1$ and
  $\varepsilon > 0$ be given with $\Leb(A) > 0$. Let
  $\theta> 0$ be a lower bound for the density of hyperbolic
  times for Lebesgue almost every point. Then there are
  integers $n<h_1<\ldots < h_k$ for
  $k = k(\varepsilon)\geq 1$ and families $\SE_i$ of subsets
  of $M$, $i= 1,\ldots k$ such that
\begin{enumerate}
\item $\SE_1\cup\ldots\cup\SE_k$ is a finite pairwise
  disjoint family of subsets of $M$;
\item $h_i$ is a $\frac{\sigma}{2}$-hyperbolic time for
  every point in $Q$, for every element $Q\in\SE_i$,
  $i=1,\ldots, k$;
\item every $Q\in\SE_i$ is the preimage of some element
  $P\in\SP$ under an inverse branch of $f^{h_i}$,
  $i=1,\ldots, k$;
\item there is an open set $U_1\supset A$ containing the
  elements of $\SE_1\cup\ldots\SE_k$ with
  $\Leb(U_1\backslash A) < \varepsilon$;
  \item $\Leb(A\triangle\cup_i\SE_i)<\varepsilon$.
\end{enumerate}
\end{lemma}

\begin{proof}
See \cite[Lemma 3.5]{AP06}.
\end{proof}

\begin{remark}\label{rmkX}
  The statement of Lemma~\ref{lemma3.5} remains valid, with
  the same proof, replacing $f$ by an open distance
  expanding and topologically transitive map $T$ of a
  compact metric space $X$; and $\Leb$ by a $\phi$-conformal
  measure $\nu$ for a continuous potential
  $\phi:M\rightarrow \RR$; recall Remark \ref{rmk0}.
\end{remark}

We use this covering lemma to prove the following. Fix a
reference probability measure $\nu$ for the space $M$. We
recall that a $f$-invariant probability measure $\mu$ is
$\nu$-weak-SRB-like if
\begin{align*}
\limsup\limits_{n\rightarrow+\infty}\frac{1}{n}\log\nu(A_{\varepsilon,n}(\mu))=0, \ \forall \varepsilon>0,
\end{align*}
where $A_{\varepsilon,n}(\mu)$ was defined at
\eqref{eq3}. We denote by $\SW_f^*(\nu)$ the set of
$\nu$-weak-SRB-like probability measures. When $\nu=\Leb$ we
denote $\SW_f^*(\Leb)$ by $\SW_f^*$. 

%We denote $\SW_f^*$ is the unique minimal non-empty and weak$^*$
%compact set such that $p\omega(x)\subset \SW_f$ for
%$\Leb$-a.e. $x\in X$; see Section~\ref{sec:invariant measure
%  and SRB-like measure}.

\begin{proposition}\label{teorema:as-medidas-SRB-like-tem-pressao-nao-negativa}
  Let $f:M \rightarrow M$ be a non-uniformly expanding
  map. For each $\mu\in\SW_f^*$ we have
  $h_{\mu}(f)+\int\psi d\mu\geq0$\footnote{Recall that
    $\psi=-\log Jf=-\log|\det Df|$.}.
\end{proposition}

\begin{proof}
  Given $\mu\in\SW_f^*$, let $\delta_1>0$ be as in
  Proposition~\ref{prop3.1} and fix $\tau>0$. Since
  $f:M \to M$ is a $C^1$ local diffeomorphism, then $f$ is a
  regular map. Let us take a partition $\SP$ of $M$ as given
  in Lemma~\ref{lemma: Initial partition} with $\nu=\mu$, so
  that $\diam(\SP)<\xi<\frac{\delta_1}{4}$ so that
  $|\psi(x)-\psi(y)|<\tau/2$ if $d(x,y)<\xi$ and also
  $\mu(\partial\SP)=0$.  Since $\mu$ is $f$-invariant, then
  the function $\lambda\mapsto h(\SP,\lambda)$ is upper
  semicontinuous at $\mu$, that is, for each small enough
  $\tau>0$ we can find $\delta_2>0$ such that $0<\delta_2<\tau$ and
\begin{eqnarray}\label{desigualdade-da-semi-cont-superior-de
    h(f,P)}
  \textrm{if} \ \ \dist(\mu, \tilde{\mu})\leq\delta_2,
  \ \ \textrm{then} \ \  h(\SP,\tilde{\mu})\leq h(\SP,\mu)+\tau.
\end{eqnarray}
%Fix $\tau>0$ and $0<\delta_2<\tau$ as above.
Since $\mu$ is a weak-SRB-like probability measure, for any
given $0<\varepsilon<\delta_2/3$ there exists a subsequence
of integers $n_l\rightarrow+\infty$ such that
$\Leb(A_{\varepsilon,n_l}(\mu))>0$ for all $l>0$.  Let us
take $\epsilon=\delta_2/3>0$ and set $\gamma_0=\delta>0$
given by Lemma~\ref{lm1} and also $0<\gamma_1<\xi$ satisfying
$|\psi(x)-\psi(y)|<\delta_2/3$ if $d(x,y)<\gamma_1$. We
denote $\gamma=\min\{\gamma_0,\gamma_1\}\le\xi$.

Let us choose one interior point having density $\geq\theta$
of $\sigma$-hyperbolic times of $f$ in each atom
$P\in \mathcal{P}$ and form the set $W_0=\{w_1,\ldots,w_S\}$
where $S=\#\SP$.

We now use Lemma~\ref{lemma3.5} to obtain a covering of
$A_{\varepsilon,n_l}(\mu)$ by hyperbolic preballs. We take
positive integers $l,m$ and
$\beta_l= \frac{1}{n_l} \Leb(A_{\varepsilon,n_l}(\mu))>0$
such that $\sigma^{\frac{m}{2}}\delta_1/4<\gamma$. Then
there are integers $n_l<n_{l}+m\leq h_1<h_2<\ldots<h_{k} $
with $k=k\left(l\right)\geq1$ (here $\beta_l$ takes the
place of $\varepsilon$ in Lemma \ref{lemma3.5}) and families
$\SE_j$ of subsets of $M$, $j= 1,\ldots k$ so that
\begin{eqnarray}\label{desig12}
  \Leb(A_{\varepsilon,n_l}(\mu))&=& \sum\limits_{j=1}^{k}\Leb(A_{\varepsilon,n_l}(\mu)\cap \SE_{j})+\sum\limits_{j=1}^{k}\Leb(A_{\varepsilon,n_l}(\mu)\backslash \SE_{j})\nonumber\\
                                &\leq& \sum\limits_{j=1}^{h_{k}}\sum_{Q\in\SE_{j}}\Leb(A_{\varepsilon,n_l}(\mu)\cap Q)+\beta_l, \ \ \textrm{hence}\nonumber\\
  \Leb(A_{\varepsilon,n_l}(\mu))&\leq& \frac{n_l}{n_l-1}\sum\limits_{j=1}^{h_{k}}\sum_{Q\in\SE_{j}}\Leb(A_{\varepsilon,n_l}(\mu)\cap Q),
\end{eqnarray}
where $\SE_{j}=\SE_{h_j}$ and $A_{\varepsilon,n_l}(\mu)\cap \SE_{j}=\bigcup_{Q\in \SE_{j}}A_{\varepsilon,n_l}(\mu)\cap Q$.
\begin{figure}[htpb]
  \centering
  \includegraphics[scale=0.35]{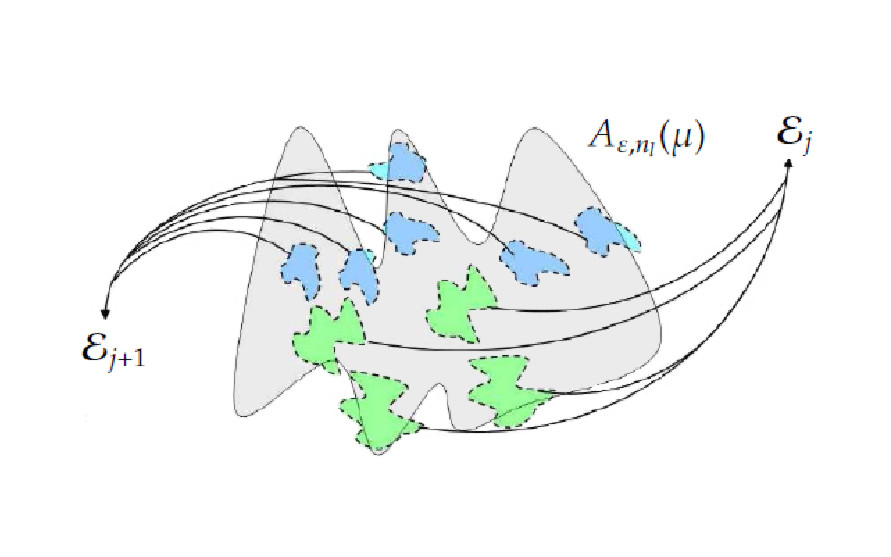}
  \caption{A sketch of $\SE_j$, the family of all sets
    $f^{-h_j}(P)$ which intersect $A_{\varepsilon,n_l}(\mu)$
    in points for which $h_j$ is a hyperbolic time, where
    $P\in\SP$. Analogously for $\SE_{j+1}$.}
  \label{conj cni}
\end{figure}
Since $\diam(\SP)<\frac{\delta_1}{4}$, by Lemma
\ref{lemma3.5}, $f^{h_j}|_{Q}:Q\rightarrow f^{h_j}(Q)$ is a
diffeomorphism for all $Q\in\SE_{j}$ and $1\leq j\leq k$,
where $h_j$ is a $\frac{\sigma}{2}$-hyperbolic time for
every point in $Q$.  We recall that $Q\in\SE_{j}$ is the
preimage of some element $P\in\SP$ under an inverse branch $\phi$
of $f^{h_j}$, $j=1,\ldots, k$.  Then
\begin{eqnarray*}
  \Leb(Q\cap A_{\varepsilon,n_l}(\mu))
  =\int_{f^{h_j}(Q\cap
  A_{\varepsilon,n_l}(\mu))}\!\!\!\!\!\!
  |\det Df^{-h_j}|\, d\Leb
  =
  \int_{f^{h_j}(Q\cap
  A_{\varepsilon,n_l}(\mu))}\!\!\!\!\!\!e^{S_{h_j}\psi}\circ\phi
  \,d\Leb.
\end{eqnarray*}
Note that
$S_{h_j}\psi(y)=S_{n_l}\psi(y)+S_{h_j-n_l}\psi(f^{n_l}(y))$
and, since $h_j\geq n_{l}+m$ is a hyperbolic time for all
$y\in \SE_j$, we have
$S_{h_j-n_l}\psi(f^{n_l}(y))=-\sum_{i=0}^{h_j-n_l-1}\log|\det
Df(f^{i+n_l}(y))|\leq0$ and so
$S_{h_j}\psi(y)\leq S_{n_l}\psi(y)$ for all
$y\in\SE_j$. Therefore,
\begin{eqnarray*}
  \Leb(Q\cap A_{\varepsilon,n_l}(\mu))
  \leq
  \int_{f^{h_j}(Q\cap
  A_{\varepsilon,n_l}(\mu))}\!\!\!\!\!\!\!e^{S_{n_l}\psi}\circ\phi
  \,d\Leb, \quad \forall Q\in\SE_j, \forall j=1,\ldots,k.
\end{eqnarray*}
Let us take $y_Q\in Q\cap A_{\varepsilon,n_l}(\mu)$ for each
$Q\in\SE_{j}$ such that
$\Leb(Q\cap A_{\varepsilon,n_l}(\mu))>0$, and let $x_Q\in Q$
be such that $f^{h_j}(x_Q)\in W_0$ (recall that elements of
$W_0$ are interior points of each atom of the partition
$\SP$). We write $W_{l}$ for the set of all points $x_Q$ for
all $Q\in\SE_j$ such that
$\Leb(Q\cap A_{\varepsilon,n_l}(\mu))>0$ for all
$j=1,\ldots,k$.

For each atom $P\in\SP^{n_l}$ with $P\cap W_l\neq\emptyset$ we
choose only one point of $P\cap W_l$ to form the subset
$\hat W_l$. We note that for any pair
$x_Q,x_Q'\in P\cap W_l$ we have
$d(f^i(x_Q),f^i(x_Q'))<\xi, i=0,\dots,n_l-1$ so
$Q\cup Q'\subset B(x_Q,n_l,\xi+\xi)$ and so from
\eqref{desig12} we obtain
\begin{align*}
  \Leb(A_{\varepsilon,n_l}(\mu))
  \leq
  \frac{n_l}{n_l-1}
  \sum_{x\in\hat W_l}\Leb(A_{\varepsilon,n_l}(\mu)\cap B(x,n_l,2\xi))
\end{align*}
and also by the choice of the initial partition and of $\xi$
\begin{align*}
  \Leb(B(x,n_l,2\xi))
  \le
  \int_{B(f^{n_l}(x),2\xi)}|\det Df^{-n_l}|\,d\Leb
  \le
  e^{S_{n_l}(\psi(x)+\tau/2)}\Leb(B(f^{n_l}(x),2\xi)).
\end{align*}
% Altogether we obtain
% \begin{eqnarray*}
%   \Leb(A_{\varepsilon,n_l}(\mu)\cap Q)
%   \leq
%   \int_{f^{h_j}(Q\cap
%   A_{\varepsilon,n_l}(\mu))}\!\!\!\!\!\!\!\!\!e^{S_{n_l}\psi}d\Leb
%   \leq \exp\big[S_{n_l}\psi(y_Q)+n_l\frac{\delta_2}{3}\big]
% \leq \exp\big[2n_l\frac{\delta_2}{3}+S_{n_l}\psi(x_Q)\big].
% \end{eqnarray*}
% Thus, by (\ref{desig12})
So we can write
$\Leb(A_{\varepsilon,n_l}(\mu)) \leq
\frac{n_l}{n_l-1}e^{\tau n_l/2}\sum_{x\in
\hat W_{n_l}}e^{S_{n_l}\psi(x)}$.  Setting
$L(n_l):=\sum_{x\in \hat W_{n_l}}e^{S_{n_l}\psi(x)}$ and
$\lambda_x:=\frac{1}{L(n_l)}e^{S_{n_l}\psi(x)}$ we can
rewrite
\begin{eqnarray}\label{equa12}
  \Leb(A_{\varepsilon,n_l}(\mu))
  &\leq&
         \frac{n_l}{n_l-1}
         \exp\left[
         n_l\left(\frac{\tau}2+\frac{1}{n_l}\log
         L(n_l)
         \right)\right].
\end{eqnarray}
Note that since $\sum_{x\in W_{l}}\lambda_x=1$, then
$\log L(n_l)=\sum_{x\in
  W_{l}}\lambda_x\big(S_{n_l}\psi(x)-\log\lambda_x\big)$ by
Lemma~\ref{lemma10.4.4}.

Defining the probability measures
$\nu_{n_l}:=\sum_{x\in W_{n_l}}\lambda_x\delta_{x}$ and
$\mu_{n_l}:=\frac{1}{n_l}\sum_{i=0}^{n_l-1}(f^i)_*(\nu_{n_l})=\sum_{x\in
  W_{n_l}}\lambda_x\sigma_{n_l}(x)$, we rewrite
again\footnote{Since the atoms of $\SP$ contain at most one
  element from $\hat W_l$.}
\begin{eqnarray}\label{eq-do-ln2}
\log L(n_l) = n_l\int\psi d\mu_{n_l}+H(\SP^{n_l},\nu_{n_l}).
\end{eqnarray}
Now we take a subsequence $n_{l_i}\nearrow\infty$ such that
$\mu_{n_{l_i}}\to\tilde{\mu}$ in the weak$^*$ topology and
also
\begin{eqnarray}\label{eq17N}
\lim_{i\to+\infty}\frac{1}{n_{l_i}}\log
  \Leb(A_{\epsilon,n_{l_i}}(\mu))
  =
  \limsup_{n\to+\infty}\frac{1}{n}\log \Leb(A_{\epsilon,n}(\mu)).
\end{eqnarray}
We keep the notation $n_l$ for simplicity in what follows.

From Proposition \ref{prop3.1}, we know that
\begin{align*}
  \max\{\diam(f^l(Q)); \ Q\in\SE_j, \ l=0,\ldots n_l\}
  <
  \sigma^{\frac{1}{2}(h_j-n_l)}\delta_1/4
  <
  \sigma^{\frac{m}{2}}\delta_1/4<\gamma\le\xi,
\end{align*}
for all $j=1,\ldots, k$. By uniform continuity of $\psi$ we
get $|\psi(f^l(x_Q))-\psi(f^l(y))|<\delta_2/3$ for all
$y\in Q$ and $l=0,\ldots, n_l-1$.
We observe that, for each $x\in \hat W_{n_l}$ there exists
$y_Q\in Q\cap A_{\varepsilon,n_l}(\mu)$ such that
$x,y_Q\in Q$. Hence $d(f^i(x),f^i(y_Q))<\gamma$ for all
$i=0,\ldots, n_l-1$.  We then have
$\dist(\sigma_{n_l}(x),\sigma_{n_l}(y_Q))<\delta_2/3$ by
Lemma~\ref{lm1} and  by the triangular inequality
$$\dist(\sigma_{n_l}(x),\mu)\leq\dist(\sigma_{n_l}(x),\sigma_{n_l}(y_Q))+\dist(\sigma_{n_l}(y_Q),\mu)<\frac{\delta_2}{3}+\varepsilon<\delta_2,$$
because $y_Q\in A_{\varepsilon,n_l}(\mu)$ and
$\varepsilon<\delta_2/3$. Thus, for any
$\varphi\in C^0(M,\RR)$,
\begin{eqnarray*}
  \left|\int\varphi d\mu_{n_l}- \int\varphi d\mu \right|
  % &=&
  %     \left|\sum_{x\in
  %     W_{n_l}}\lambda_x\int\varphi\,d\sigma_{n_l}(x) -
  %     \sum\limits_{x\in W_{n_l}}\lambda_x\int\varphi d\mu
  %     \right|
  % \\
  &\leq&
         \sum\limits_{x\in W_{n_l}}\lambda_x
         \left|\int\varphi\,d\sigma_{n_l}(x)
         -\int\varphi d\mu \right|<\delta_2.
\end{eqnarray*}
Hence we have $\dist(\mu,\tilde{\mu})\leq\delta_2$ and consequently
\begin{eqnarray}\label{equa123}
  \int\psi d \mu_{n_l}\leq \int\psi d\mu+\delta_2
  \qand
  \int\psi d \tilde{\mu}\leq \int\psi d\mu+\delta_2.
\end{eqnarray}
We now use again item (2) of Lemma~\ref{lemma: Initial
  partition} with $C_0=W_0$, $\nu=\tilde\mu$,
$\nu_k=\mu_{k+1}, k\ge1$ and $\mu_1=\mu$. We obtain a small
perturbation $\tilde{\SP}$ of $\SP$ so that
$x \in \hat W_{l}$ still belongs to the same atom of the
$h_j-$th refinement of $\tilde{\SP}$ which intersects
$A_{\varepsilon,n_l}(\mu)$, for $Q\in \SE_j; j=1,\ldots, k$
and $l\geq1$.  More precisely, $\tilde{\SP}$ satisfies,
besides item (2) of Lemma~\ref{lemma: Initial partition},
the useful property: $x_Q\in \tilde{Q}\cap Q$ for
$Q\in\SE_{j}$ where $\tilde{Q}\in \tilde{\SE}_{j}$ and
$ f^{h_j}(x_Q)=w\in \tilde{\SP}(w)\cap \SP(w)$.

Now we observe that
$H(\tilde{\SP}^{n_l},\nu_{n_l})=H(\SP^{n_l},\nu_{n_l})$ by
definition of $\nu_{n_l}$ and by construction of the
$\tilde{\SP}$ as a perturbation of $\SP$. Following the
proof of Lemma \ref{lm2a} (see inequality
(\ref{desigualdade-das-entropias})) there exists $l_0\geq 0$
such that
\begin{eqnarray*}
  \frac{1}{n_l}H(\SP^{n_l},\nu_{n_l})
  =
  \frac{1}{n_l}H(\tilde{\SP}^{n_l},\nu_{n_l})
  \leq
  h(\tilde{\SP},\tilde{\mu})+\frac{\delta_2}{4}, \quad \forall l\geq l_0.
\end{eqnarray*}
We have
$H_{\tilde{\mu}}(\tilde{\SP}/\SP)<\frac{\delta_2}{4}$ by
Lemma \ref{lemma9.1.6} and item (2e) from Lemma~\ref{lemma:
  Initial partition}. Then, because
$h(\tilde{\SP},\tilde{\mu})\leq
h(\SP,\tilde{\mu})+H_{\tilde{\mu}}(\tilde{\SP}/\SP)$, we get
\begin{eqnarray*}
  \frac{1}{n_l}H(\SP^{n_l},\nu_{n_l})\leq h(\SP,\tilde{\mu})+\frac{\delta_2}{4}\leq h(\SP,\mu)+\frac{\delta_2}{4}+\tau, \ \ \forall l\geq l_0,
\end{eqnarray*}
where the last inequality follows by the choice of
$\delta_2$ in (\ref{desigualdade-da-semi-cont-superior-de
  h(f,P)}). Thus,
\begin{eqnarray}\label{equa441}
\frac{1}{n_l}H(\SP^{n_l},\nu_{n_l})\leq h(\SP,\mu)+\frac{\delta_2}{4}+\tau\leq h_{\mu}(f)+\frac{5}{4}\tau, \ \ \forall l\geq l_0.
\end{eqnarray}
and we recall that $\delta_2<\tau$.  Combining assertions
(\ref{equa12}), (\ref{eq-do-ln2}), (\ref{equa441}) and
(\ref{equa123}) we arrive at
\begin{align*}
  \Leb(A_{\varepsilon,n_l}(\mu))
  &\leq
    \frac{n_l}{n_l-1}\exp{\left[n_l\left(\frac{\tau}{2}+\frac{1}{n_l}\log
    L(n_l)\right)\right]}
  \\
  &\leq
    \frac{n_l}{n_l-1}\exp{\left[n_l\left(\frac{\tau}{2}+\int\psi\,
    d\mu_{n_l}+\frac{1}{n_l}H(\SP^{n_l},\nu_{n_l})\right)\right]}
  \\
  &\leq
    \frac{n_l}{n_l-1}\exp{\left[n_l\left(3\tau+h_{\mu}(f)+\int\psi d\mu\right)\right]}.
%&=&\exp{\left[n_l\left(h_{\mu}(f)+\int\psi d\mu+4\tau\right)\right]}.
\end{align*}
Hence we obtain
\begin{eqnarray*}
\frac{1}{n_l}\log \Leb(A_{\varepsilon,n_l}(\mu))
&\leq& \frac{1}{n_l}\log\left(\frac{n_l}{n_l-1}\right)+h_{\mu}(f)+\int\psi d\mu+3\tau, \ \ \forall l\geq l_0.
\end{eqnarray*}
Since $\mu\in\SW_f^*$, we conclude
$$0=\limsup\limits_{n\rightarrow+\infty}\frac{1}{n}\log \Leb(A_{\varepsilon,n}(\mu))=\lim\limits_{l\rightarrow+\infty}\frac{1}{n_l}\log\Leb(A_{\varepsilon,n_l}(\mu))\leq h_{\mu}(f)+\int\psi d\mu+3\tau.$$

As $\tau > 0$ is arbitrary, the proof of the proposition is
complete.
\end{proof}

\begin{remark}\label{rmkY}
  The statement of
  Proposition~\ref{teorema:as-medidas-SRB-like-tem-pressao-nao-negativa}
  is still valid if: (i) we replace $f$ by an open distance
  expanding and topologically transitive map $T$ of a
  compact metric space $X$; and (ii) $\Leb$ by a
  $\phi$-conformal measure $\nu$ with
  $\mathcal{L}_{\phi}^*(\nu)=\lambda\nu$, for some
  $\lambda>0$ and a continuous potential
  $\phi:X\rightarrow\RR$. Indeed, the proof uses that $\Leb$
  is $\psi$-conformal together with a covering lemma that
  clearly holds for distance expanding maps (recall
  Remark~\ref{rmkX}). Thus, we obtain
  $h_{\mu}(f)+\int\psi d\mu-\log\lambda\geq0, \ \textrm{for
    all} \ \mu\in\SW_T^*(\nu).$ Together with
  Lemma~\ref{lema1}, this shows that
  $P_{\textrm{top}}(T,\phi)=\log\lambda$ and so all
  $\nu$-weak-SRB-like probability measures are
  $\phi$-equilibrium states.
\end{remark}

\subsection{Proof of Theorem \ref{mthm:formula de pesin para
    difeo-local com HT}}
\label{sec:proof-theorem-refmth}

To prove Theorem \ref{mthm:formula de pesin para difeo-local
  com HT}, consider an expanding weak-SRB-like measure
$\mu$. Then there exists $\sigma\in(0,1)$ such that
$\limsup_{n\to\infty}\frac{1}{n}\sum_{i=0}^{n-1}\log\|Df(f^j(x))^{-1}\|\leq\log\sigma<0$
for $\mu$-a.e $x\in M$.

Thus, the Lyapunov exponents are non-negative. Hence the sum
$\Sigma^+(x)$ of the positive Lyapunov exponents of a
$\mu-$generic point $x$, counting multiplicities, is such
that
$\Sigma^+(x)=\lim_{n\rightarrow+\infty}\frac{1}{n}\log|\det
Df^n(x)|$ (by the Multiplicative Ergodic Theorem) and
$\int \Sigma^+ d\mu=\int\log|\det Df| d\mu=-\int\psi d\mu$
by the standard Ergodic Theorem.

For $C^1$-systems, Ruelle's Inequality \cite{Ru78} states
that for any $f$-invariant probability measure $\mu$ on the
Borel $\sigma$-algebra of $M$, the corresponding
measure-theoretic entropy $h_{\mu}(f)$ satisfies
$h_{\mu}(f)\leq\int\Sigma^+ d\mu$ and consequently
$h_{\mu}(f)+\int\psi d\mu\leq0$.

%theorem(Furstenberg-Kesten) pg. 86 do livro do Oliveira-Viana.
By definition, Pesin's Entropy Formula holds if the latter
difference is equal to zero. Since $\mu\in\SW_f^*$, by
Proposition
\ref{teorema:as-medidas-SRB-like-tem-pressao-nao-negativa},
we have that $h_{\mu}(f)+\int\psi d\mu=0$ which proves the
first statement of Theorem \ref{mthm:formula de pesin para
  difeo-local com HT}.

The next result completes the proof of
Theorem~\ref{mthm:formula de pesin para difeo-local com HT}.

\begin{proposition}\label{corolario:SRB-like-e-expansora}
  Let $f:M \rightarrow M$ be non-uniformly expanding. Then
  all the ergodic $SRB$-like probability measures are
  expanding probability measures.
  % and, in particular, the Lyapunov exponents for any
  % ergodic $SRB$-like probability measure are non-negative.
\end{proposition}

\begin{proof}
  The assumptions on $f$ ensure that there exists
  $\sigma\in(0,1)$ such that $\Leb(H(\sigma))=1$.  The proof
  uses the following.

  \begin{lemma}\label{lema-simples1}
    If $f:M\rightarrow M$ is a $C^1$ local diffeomorphism
    such that $\Leb(H(\sigma))=1$ for some $\sigma\in(0,1)$,
    then each $\mu\in\SW_f$ satisfies
    $\int\log\|(Df)^{-1}\|d\mu<\log\sqrt{\sigma}$.
\end{lemma}

\begin{proof}
  Fix $0<\varepsilon<-\frac{1}{2}\log\sigma$ small
  enough. Since that $\varphi(x):=\log\|Df(x)^{-1}\|$ is a
  continuous potential, from the definition of the weak$^*$
  topology in space $\SM_1$ of probability measures, we
  deduce that there exists $0<\varepsilon_1<\varepsilon$
  such that if $\dist(\mu,\nu)<\varepsilon_1$ then
  $|\int\varphi d\mu-\int\varphi d\nu|<\varepsilon$ for all
  $\mu,\nu\in\SM_1$.

  Let us take $\mu\in\SW_f$, then
  $\Leb(A_{\varepsilon_1}(\mu))>0$. Let
  $x\in A_{\varepsilon_1}(\mu)\cap H(\sigma)$ and consider
  $\nu_x\in p\omega(x)$ such that
  $\dist(\mu,\nu_x)<\varepsilon_1$. Then
$$\left|\int \log\|(Df)^{-1}\|d\mu-\int\log\|(Df)^{-1}\|d\nu_x\right|<\varepsilon,$$
%by weak convergence, where $\nu_x=\lim\limits_{k\rightarrow+\infty}\sigma_{n_k}(x)\in p\omega(x)$. Thus,
%pela convergencia fraca (pq $\log\|Df^-1\|$ é contínua)
% onde $\nu_x=\lim\limits_{k\rightarrow+\infty}\sigma_{n_k}(x)\in p\omega(x)$. Assim,
and therefore there exists $n_k\nearrow\infty$ so that
$\sigma_{n_k}(x)\xrightarrow[]{w^*}\nu_x$. Thus
\begin{align*}
  \int \log\|(Df)^{-1}\|d\mu
  &\leq
    \int\log\|(Df)^{-1}\|d\nu_x +\varepsilon
    =
    \lim\limits_{k\rightarrow+\infty}\frac{1}{n_k}\sum\limits_{j=0}^{n_k-1}\log\|Df(f^j(x))^{-1}\|+\varepsilon
  \\
  &\leq
    \limsup\limits_{n\rightarrow+\infty}\frac{1}{n}\sum\limits_{j=0}^{n-1}\log\|Df(f^j(x))^{-1}\|+\varepsilon
    < \log\sigma+\varepsilon<\log\sqrt{\sigma}
\end{align*}
as stated.
\end{proof}

Going back to the proof of the proposition, since $\mu$ is
$f$-invariant and ergodic, then by the previous lemma and by
the standard Ergodic Theorem
$$\lim\limits_{n\rightarrow+\infty}\frac{1}{n}\sum\limits_{j=0}^{n-1}\log\|Df(f^j(y))^{-1}\|=\int \log\|(Df)^{-1}\| d\mu<\log\sqrt{\sigma} \ \ \textrm{for} \ \mu-\textrm{a.e} \ y\in M.$$
Therefore $\mu$ is an expanding measure. This finishes the
proof of the proposition.
\end{proof}
This completes the proof of Theorem~\ref{mthm:formula de
  pesin para difeo-local com HT}.

%%%%%%%%%%%%%%%%%%%%%%%%%%%%%%%%%%%%%%%%%%%%%%%%%%%%%%%

\section{Ergodic weak-SRB-like measure}
\label{sec:Ergodic SRB-like measure}

In this section, we prove Corollary
\ref{mthm:pressao-top-zero-implica-existencia-de-medida-ergodica}
on the existence of ergodic weak-SRB-like measures for
non-uniformly expanding local diffeomorphisms
$f:M\to M$.

\subsection{Ergodic expanding invariant measures}

\begin{theorem}\label{teo*}
  Let $f:M\rightarrow M$ be a $C^1$ local diffeomorphism. If
  $\mu$ is an ergodic expanding $f$-invariant probability
  measure, then
%such that $\mu$-a.e. $x$ has positive frequency $\geq\theta$ of $\sigma$-hyperbolic times of $f$, then:
\begin{eqnarray}\label{desigualdade 1}
\lim\limits_{\varepsilon\rightarrow 0^+}\limsup\limits_{n\rightarrow+\infty}\frac{\log(\Leb(A_{\varepsilon,n}(\mu)))}{n}\geq \int\psi d\mu+h_{\mu}(f).
\end{eqnarray}
\end{theorem}
\begin{proof}

%no enunciado estamos assumindo implicitamente que $h_{top}(f)<\infty$.

  Since $\mu$ is an ergodic probability measure, we have
  $\lim_{n\rightarrow+\infty}\sigma_n(x) = \mu$ for
  $\mu$-a.e. $x\in M$. So for $\mu$-a.e. $x\in M$, there
  exists $N(x)\geq 1$ such that
  $\dist(\sigma_n(x),\mu)<\varepsilon/4$
  $\forall n\geq N(x).$

  Given $\varepsilon>0$ and any natural value of $N\geq 1$,
  define the set
\begin{eqnarray}\label{conj-BN}
B_N := \{x\in M: \  \dist(\sigma_n(x),\mu) < \varepsilon/4 \ \ \forall n\geq N\}.
\end{eqnarray}

Consider $\delta_1>0$ such that for each $\sigma$-hyperbolic
time $h\geq1$ time for $x$, $f^h|_{B(x,h,\delta_1)}$ maps
$B(x,h,\delta_1)$ diffeomorphically to the ball of radius
$\delta_1$ around $f^h(x)$.

Fix $\delta>0$ such that Lemma \ref{lm1} holds with
$\varepsilon/8$ in the place of $\varepsilon$ and fix
$\xi>0$ satisfying
$|\psi(x)-\psi(y)|<\frac{\varepsilon }{4}$ if $d(x,y)<\xi$.

Consider $0<\gamma_0<\min\{\xi,
\delta,\delta_1/2\}$ %, in the sequel, we denote $B(x,n,\gamma/2)$ the hyperbolic preballs of $x$.
and let $N(n, \gamma, b)$ be the minimum number of points
needed to $(n, \gamma)$-span a set of $\mu$-measure $b$ (see
(\ref{def-entropia-do-conj-K})). Choose
$0 < \gamma_1 < \gamma_0$ such that
\begin{eqnarray}\label{def-de-entropia}
\liminf_{n\rightarrow+\infty}\frac{1}{n}\log
  N(n,4\gamma,1/2)
  \geq
  h_{\mu}(f)-\frac{\epsilon}{2}, \quad \forall\gamma < \gamma_1.
\end{eqnarray}
Set
$A=\left\{x\in M: \
  \limsup_{n\rightarrow+\infty}-\frac{1}{n}\log
  \Leb(B(x,n,\gamma_2))\leq
  h_{\Leb}(f,\mu)+\frac{\varepsilon}{4}\right\}$ where
$0<\gamma_2\leq\gamma_1$ is such that
\begin{eqnarray}\label{conj-CN}
  \mu(A)>2/3.
\end{eqnarray}
This is possible by definition of $h_{\Leb}(f,\mu)$ (see
(\ref{def-de-entropia-via-supremo-essencial})). We have
implicitly assumed that $h_{\Leb}(f,\mu)<\infty$ here. If
$h_{\Leb}(f,\mu)=\infty$, then there is nothing to prove
since $h_{\mu}(f) < h_{top}(f) <\infty$ because $f$ is a
$C^1$ local diffeomorphism.
%Pela desigualdade de Ruelle $h_{\mu}(f)\leq \int\sum^+d\mu$ por outro lado, como a derivada é limitada
%(pois é continua definida num compacto) segue que os expoentes de Lyapunov são limitados e portanto $h_{\mu}(f)\leq K$ para todo $\mu\in\SM_f$ portanto $h_{top}(f)\leq K$.

We note that $B_N\subset B_{N+1}$ and $\mu(\cup B_N) =
1$. So there exists $N\geq 1$ such that
$\mu(B_{N})\geq \frac{5}{6}$. If $C_N:=A\cap B_N$, then
$\mu(C_{N})\geq \frac{1}{2}$ and for all $ x\in C_N$ and
$n\geq N(x)$ we have:

%\mu(A\cap B_N)+\mu(B_N\cap A^c)=\mu(B_N) then 5/6=\mu(B_N)\leq \mu(A^c)+\mu(A\cap B_N) \leq 1/3+\mu(A\cap B_N) =>\mu(A\cap B_N)\geq5/6-1/3
\begin{enumerate}
  \item [(1)] $B(x,n,\gamma_2)\subset A_{\varepsilon,n}(\mu)$;
  \item [(2)] $\Leb(B(x,n,\gamma_2)\geq e^{-(h_{\Leb}(f,\mu)+\varepsilon/2)n}$.
\end{enumerate}

Note that $(2)$ immediately follows from
(\ref{conj-CN}). Moreover, $(1)$ holds
because, %$B(x,n,\gamma)\subset A_{\varepsilon,n}(\mu)$ for all $n\geq N$ and for any $x\in B_N$.
given $y\in B(x,n,\gamma)$ then $d(f^j(y),f^j(x))<\gamma$
for all $j=0,\ldots,n-1$. By Lemma \ref{lm1} we have
$\dist(\sigma_{n}(y),\sigma_{n}(x))<\frac{\varepsilon}{8}$. Since
$x\in B_N$, by the triangular inequality

$$\dist(\sigma_{n}(y),\mu)\leq\dist(\sigma_{n}(y),\sigma_{n}(x))+\dist(\sigma_{n}(x),\mu)< \varepsilon.$$
Therefore, $y\in A_{\varepsilon,n}(\mu)$.

For each $n$, let $E_n=E_n(2\gamma_2)$ be a maximal
$(n, 2\gamma_2)$-separated subset of points contained in
$C_N$. Then
$\bigcup\limits_{x\in E_n}B(x,4\gamma_2,n)\supset C_N$ by
maximality of $E_n$ and so $\#E_n\geq N(n, 4\gamma_2,
1/2)$. Also, given $x,y\in E_n$, $x \neq y$ then
$B(x, n ,\gamma_2) \cap B(y, n, \gamma_2) =
\emptyset$. Thus,
\begin{eqnarray*}
\liminf\limits_{n\rightarrow+\infty}\frac{1}{n}\log \Leb(A_{\varepsilon,n}(\mu))&\geq& \liminf\limits_{n\rightarrow+\infty}\frac{1}{n}\log\sum\limits_{x\in E_n}\Leb(B(x, n ,\gamma_2))\\
&\geq& \liminf\limits_{n\rightarrow+\infty}\frac{1}{n}\log\left(\#{E_n}\cdot e^{-(h_{\Leb}(f,\mu)+\varepsilon/2)n}\right)\\
&\geq& \liminf\limits_{n\rightarrow+\infty}\frac{1}{n}\log N(n,4\gamma_2,1/2) -h_{\Leb}(f,\mu)-\varepsilon/2.
\end{eqnarray*}
%obs: acima podemos tomar o log, pq vimos que $B(x,n,\gamma_2) \subset A_{\varepsilon,n}(\mu)$ e $B(x,n,\gamma_2)$ é aberto e $\Leb$ é positiva em abertos, portanto \Leb(A_{\varepsilon,n}(\mu))>0$.
Thus, by (\ref{def-de-entropia}) we have,
\begin{eqnarray}\label{desigualdade-a}
\liminf\limits_{n\rightarrow+\infty}\frac{1}{n}\log \Leb(A_{\varepsilon,n}(\mu))\geq h_{\mu}(f)-h_{\Leb}(f,\mu).
\end{eqnarray}
Moreover, since
$\mu$ is expanding, then there exist
$\theta>0$ and $\sigma\in(0,1)$ such that $\mu$-a.e. $x\in
M$ has positive frequency
$\geq\theta$ of $\sigma$-hyperbolic times of
$f$. Let $\tilde{M}\subset M$ be such that for all $x\in
\tilde{M}$,
$\lim\limits_{n\rightarrow+\infty}\sigma_n(x)=\mu$ and
$f^h|_{B(x,h,\gamma_2)}$ maps
$B(x,h,\gamma_2)$ diffeomorphically to the ball of radius
$\gamma_2$ around $f^h(x)$, for $h$ a
$\sigma$-hyperbolic time for $x$.
%Let $\tilde{M}\subset M$ such that $\mu(\tilde{M})=1$ and for all $x\in \tilde{M}$, $\lim\limits_{n\rightarrow+\infty}\sigma_n(x)=\mu$. Given $x\in\tilde{M}$, consider $y\in M$ such that $y$ has positive frequency $\geq\theta$ of $\sigma$-hyperbolic times of $f$ and $x\in B(y,h,\gamma_2/2)$ the hyperbolic preballs of $y$ such that $f^h|_{B(y,h,\gamma_2/2)}$ maps $B(y,h,\gamma_2/2)$ diffeomorphically to the ball of radius $\gamma_2/2$ around $f^h(y)$, for $h$ a $\sigma$-hyperbolic time for $y$.
%isso é possível pq Leb-qtp tem frequencia positiva de tempos hiperbolicos então existe y suficientemente proximo de x tal que B(y,h,\gamma_2/4)\subset B(x,h,\gamma_2)
We observe that
$|\psi(f^j(y))-\psi(f^j(z))|<\frac{\varepsilon
}{4}$ for all $z\in
B(y,h,\gamma_2/2)$, since
$d(f^jx,f^jy)\leq\gamma_2/2$ for all $j=0,\ldots,h-1$.

Hence, by uniform continuity of
$\psi$ and because Lebesgue measure assigns mass uniformly
bounded away from zero to balls of fixed radius, we obtain
\begin{eqnarray*}
  0<\alpha(\gamma_2)&\leq&\Leb(B(f^{h}x,\gamma_2))=\int_{B(x,h,\gamma_2)}|\det Df^{h}|\, d\Leb\\
                    &\leq& \int_{B(x,h,\gamma_2)}e^{-S_{h}\psi}\, d\Leb \leq e^{-S_{h}\psi(x)+h\varepsilon/4} \cdot \Leb(B(x,h,\gamma_2)).
%&\leq&\leq e^{-S_{h}\psi(x)+h\varepsilon/4} \cdot \Leb(B(x,h,\gamma_2)).
\end{eqnarray*}
Thus
\begin{eqnarray*}
0&=& \liminf\limits_{h\rightarrow+\infty}\frac{1}{h}\log\alpha(\gamma_2/2)\leq\liminf\limits_{h\rightarrow+\infty}-\frac{1}{h}S_{h}\psi(x)+\frac{\varepsilon}{4} +  \liminf\limits_{h\rightarrow+\infty}\frac{1}{h}\log\Leb(B(x,h,\gamma_2))\\
&\leq&-\limsup\limits_{h\rightarrow+\infty}\int\psi d\sigma_{h}(x)+\frac{\varepsilon}{4}-  \limsup\limits_{h\rightarrow+\infty}-\frac{1}{h}\log\Leb(B(x,h,\gamma_2))\\
&=&-\int\psi d\mu+\frac{\varepsilon}{4} -h_{\Leb}(f,x).
\end{eqnarray*}
Therefore, $h_{\Leb}(f,x)\leq -\int\psi
d\mu+\frac{\varepsilon}{4}$ for $\mu$-a.e $x\in
M$. Hence $h_{\Leb}(f,\mu)\leq-\int\psi
d\mu+\frac{\varepsilon}{4}$ and by (\ref{desigualdade-a})
$$\lim_{\varepsilon\rightarrow0^+}\limsup_{n\rightarrow+\infty}\frac{1}{n}\log \Leb(A_{\varepsilon,n}(\mu))\geq h_{\mu}(f)+\int\psi d\mu,$$
which gives the desired inequality.
\end{proof}

\begin{proposition}\label{corolario:medida-ergodica-expansora-com-pressao-nao-negativa-eh-SRB-like}
  Let $f:M\rightarrow M$ be a $C^1$ local
  diffeomorphism. Every expanding ergodic $f$-invariant
  probability measure $\mu$ such that
  $\int\psi d\mu+h_{\mu}(f)\geq0$ is a weak-$SRB$-like
  probability measure.
\end{proposition}

\begin{proof}
  Let $\mu\in\SM_f$ be a expanding ergodic probability
  measure such that $\int\psi d\mu+h_{\mu}(f)\geq0$. By
  Theorem~\ref{teo*}, we have that
$$\lim\limits_{\varepsilon\rightarrow0^+}\limsup\limits_{n\rightarrow+\infty}\frac{\log(\Leb(A_{\varepsilon,n}(\mu)))}{n}\geq h_{\mu}(f)+\int\psi d\mu\geq 0.$$

Moreover, if $\varepsilon_1 < \varepsilon_2$ then
$A_{\varepsilon_1, n}(\mu)\subset A_{\varepsilon_2,
  n}(\mu)$. So
$\limsup_{n\rightarrow+\infty}\frac{\log(\Leb(A_{\varepsilon,n}(\mu)))}{n}$
is increasing with $\varepsilon > 0$. Thus
$\limsup_{n\rightarrow+\infty}\frac{\log(\Leb(A_{\varepsilon,n}(\mu)))}{n}\geq0$
for all $\varepsilon>0$. But since $\Leb$ is a probability
measure, we conclude that
$$\limsup\limits_{n\rightarrow+\infty}\frac{\log(\Leb(A_{\varepsilon,n}(\mu)))}{n}=0,
\quad\forall \varepsilon>0.$$
Then $\mu$ is a weak-SRB-like measure.
\end{proof}

\begin{remark}\label{rmkZ}
  The statement of Theorem~\ref{teo*} still holds if we
  replace $f$ by an open distance expanding and
  topologically transitive map $T$ of a compact metric space
  $X$; $\Leb$ by a $\phi$-conformal measure $\nu$ for some
  continuous potential $\phi:X\rightarrow\RR$ with
  $\mathcal{L}_{\phi}^*(\nu)=\lambda\nu$ for some
  $\lambda>0$. So we get
$$\lim\limits_{\varepsilon\rightarrow0^+}\limsup\limits_{n\rightarrow+\infty}\frac{1}{n}\log \nu(A_{\varepsilon,n}(\mu))\geq h_{\mu}(f)+\int\psi d\mu-\log\lambda,$$
since, in the proof of Theorem~\ref{teo*}, we used that
$\Leb$ is $\psi$-conformal together with general results
from Ergodic Theory. Analogously for Proposition
\ref{corolario:medida-ergodica-expansora-com-pressao-nao-negativa-eh-SRB-like}.
\end{remark}

\begin{proposition}\label{corolario:toda-medida-ergodica-expansora-que-satisfaz-a-formula-de-Pesin-eh-SRB-like}
  Let $f:M\rightarrow M$ be a $C^1$ local
  diffeomorphism. Every expanding ergodic $f$-invariant
  probability measure that satisfies Pesin's Entropy Formula
  is a weak-$SRB$-like probability measure.
\end{proposition}

\begin{proof}
  Let $\mu\in\SM_f$ be an expanding ergodic probability
  measure such that $h_{\mu}(f)=\int\sum^+d\mu$, where
  $\sum^+(x)$ is the sum of the positive Lyapunov exponents
  of a $\mu$-generic point $x$ counting multiplicities.

  Since $\mu$ is an expanding probability measure, we deduce
  by the Multiplicative Ergodic that
  $-\int\psi d\mu\leq\int\sum^+d\mu=h_{\mu}(f)$. Therefore,
  $h_{\mu}(f)+\int\psi d\mu\geq0$ and the result follows
  from Proposition
  \ref{corolario:medida-ergodica-expansora-com-pressao-nao-negativa-eh-SRB-like}.
\end{proof}

\subsection{Ergodic expanding weak-SRB-like measures}
\label{sec:ergodic-expand-srb}

Now we are ready to prove Corollary
\ref{mthm:pressao-top-zero-implica-existencia-de-medida-ergodica}.

\begin{proof}[Proof of Corollary~\ref{mthm:pressao-top-zero-implica-existencia-de-medida-ergodica}]
  Since $P_{\textrm{top}}(f,\psi)=0$, we conclude by
  Proposition~\ref{teorema:as-medidas-SRB-like-tem-pressao-nao-negativa}
  that all weak-SRB-like measures are equilibrium states for
  the potential $\psi$, that is,
  $h_{\mu}(f)+\int\psi d\mu=0$ for all $\mu\in\SW_f^*$.

  We know that there exists $\sigma\in(0,1)$ such that
  $\Leb(H(\sigma))=1$. Given $\mu\in\SW_f$, by Lemma
  \ref{lema-simples1} we have that
  $\int \log\|(Df)^{-1}\|d\mu\leq\log\sqrt{\sigma}$.

  By the Ergodic Decomposition Theorem, there exists
  $A\subset M$ such that $\mu(A)>0$ and for all $y\in A$,
  $\int\log\|(Df)^{-1}\|d\mu_y\leq\log\sqrt{\sigma}$, where
  $\mu_y$ is an ergodic component of $\mu$.

  Fix $y_0\in A$. By Birkhoff's Ergodic
  Theorem,
  $$\lim\limits_{n\rightarrow+\infty}\frac{1}{n}\sum\limits_{j=0}^{n-1}\log\|Df(f^j(y))^{-1}\|=\int\log\|(Df)^{-1}\|d\mu_{y_0}\leq\log\sqrt{\sigma},$$
%Therefore $y$ has positive frequency $\geq\theta_1$ of $\sqrt{\sigma}$-hyperbolic times of $f$.\sum
for $\mu_{y_0}$-a.e. $y\in M$. Therefore $\mu_{y_0}$ is an
expanding probability measure.

%Hence $\sum^+(x) =\lim_{n\rightarrow+\infty}\frac{1}{n}\sum\limits_{j=0}^{n-1}\log|\det Df(f^j(x))|$ (by the Multiplicative Ergodic Theorem). Thus, combining the standard Ergodic Theorem and Ruelle's Inequality, we have $h_{\mu_{y_0}}(f)+\int\psi d\mu_{y_0}\leq0$.

Since $P_{\textrm{top}}(f,\psi)=0$ and
$h_{\mu}(f)+\int\psi d\mu=0$ we conclude that
$h_{\mu_{y}}(f)+\int\psi d\mu_{y}=0$ for $\mu$-a.e $x\in
M$. In particular, by
Corollary~\ref{corolario:medida-ergodica-expansora-com-pressao-nao-negativa-eh-SRB-like}
we conclude that there exist $y_0\in A$ such that
$\mu_{y_0}\in\SW_f^*$ showing that there are ergodic
expanding weak-SRB-like probability measures which satisfy
Pesin's Entropy Formula, completing the proof of Corollary
\ref{mthm:pressao-top-zero-implica-existencia-de-medida-ergodica}.
\end{proof}

\section{ Weak-Expanding non-uniformly expanding maps}
\label{sec:weak-expanding}

In this section we reformulate Pesin's Entropy Formula for a
class of weak-expanding and non-uniformly expanding maps
with $C^1$ regularity and prove Corollary \ref{mthm:WE}.

\subsection{Weak-SRB-like, equilibrium and expanding
  measures}
We divide the proof of Corollary \ref{mthm:WE} into the next
two results below.

\begin{proposition}\label{corolario:medidas SRB-like sao
    estados de equilibrio p/ WE}
  Let $f:M \rightarrow M$ be weak-expanding and
  non-uniformly expanding. Then, all (necessarily existing)
  weak-SRB-like probability measures are $\psi$-equilibrium
  states and, in particular, satisfy Pesin's Entropy
  Formula. Moreover, there exists some ergodic weak-SRB-like
  probability measure.
\end{proposition}

\begin{proof}
  Let $f:M\rightarrow M$ be as in statement of Corollary
  \ref{mthm:WE}. Then, for every $x\in M$ and all
  $v\in T_xM\backslash\{0\}$ we have
  $\liminf_{n\rightarrow+\infty}\frac{1}{n}\log\|Df^n(x)\cdot
  v\|\geq0$. Thus, the Lyapunov exponents for any given
  $f-$invariant probability measure $\mu$ are
  non-negative. Hence
  $\Sigma^+(x)=\lim_{n\rightarrow+\infty}\frac{1}{n}\log|\det
  Df^n(x)|$ and
  $\int \Sigma^+ d\mu=\int\log|\det Df|\, d\mu=-\int\psi\,
  d\mu.$

  For $C^1$-systems, Ruelle's Inequality \cite{Ru78} ensures
  $h_{\mu}(f)\leq\int\Sigma^+ d\mu$, so
  $h_{\mu}(f)+\int\psi d\mu\leq0$. By Proposition
  \ref{teorema:as-medidas-SRB-like-tem-pressao-nao-negativa},
  we have that $P_{\textrm{top}}(f,\psi)=0$ and
  $ h_{\mu}(f)+\int\psi d\mu=0$ for all $\mu\in\SW_f^*$.

  Therefore, all weak-SRB-like probability measures are
  $\psi$-equilibrium states and satisfy Pesin's Entropy
  Formula. Moreover, by Corollary
  \ref{mthm:pressao-top-zero-implica-existencia-de-medida-ergodica},
  we conclude that there exist ergodic weak-SRB-like
  measures.
\end{proof}

Here we obtain a sufficient condition to guarantee that all
$\psi$-equilibrium states are generalized convex linear
combinations of weak-SRB-like measures.

\begin{proposition}\label{corolario2}
  Let $f:M \rightarrow M$ be weak-expanding and
  non-uniformly expanding. If $\psi<0$, then there is no atomic
  weak-SRB-like probability measure. Moreover, if
  $\SD=\{x\in M; \|Df(x)^{-1}\|=1\}$ is finite and $\psi<0$,
  then almost all ergodic components of a $\psi$-equilibrium
  state are weak-SRB-like measures and all weak-SRB-like
  probability measures $\mu$ have ergodic components $\mu_x$
  which are expanding weak-SRB-like probability measures for
  $\mu$-a.e. $x\in M\backslash \SD$.
\end{proposition}

\begin{proof}
  Note that, if $\psi<0$, then $\int\psi d\mu<0$ for all
  $\mu\in\SM_f$. On the other hand, if $\mu\in\SM_f$ is an
  atomic invariant probability measure, then
  $h_{\mu}(f)=0$. Therefore $\mu$ does not satisfy Pesin’s
  Entropy Formula and, by Proposition~\ref{corolario:medidas
    SRB-like sao estados de equilibrio p/ WE}, we conclude
  that there is no atomic weak-SRB-like probability measure.

  Since $\int\psi d\mu<0$ and $h_{\mu}(f)\leq-\int\psi d\mu$
  for all $\mu\in\SM_f$, then a $\psi$-equilibrium state
  satisfies $h_{\mu}(f)+\int\psi d\mu=0$ and so
  $h_{\mu_x}(f)+\int\psi d\mu_x=0$, $\mu$-a.e. $x$ by the
  Ergodic Decomposition Theorem. Thus $h_{\mu_x}(f)>0$.

  Hence $\mu_x$ is non-atomic and thus expanding because
  $\int\log\|(Df)^{-1}\|d\mu_x<0$ and $\mu_x$ is ergodic,
  for $\mu$-a.e $x$. Such $\mu_x$ also satisfy the Entropy
  Formula. Therefore Proposition
  \ref{corolario:toda-medida-ergodica-expansora-que-satisfaz-a-formula-de-Pesin-eh-SRB-like}
  ensures that $\mu_x$ is weak-SRB-like for $\mu$-a.e. $x$.

  Suppose now that $\SD$ is finite and consider
  $\mu\in\SW_f^*$. Then $\int\log\|(Df)^{-1}\|d\mu<0$, for
  otherwise, we would have $\supp(\mu)\subset \SD$ so that
  $\mu$ is purely atomic, which is a contradiction with the
  previous conclusions.  Because $\mu$ is a
  $\psi$-equilibrium state, then
  $h_{\mu}(f)+\int\psi d\mu=0$. Moreover,
  $0>\int\log\|(Df)^{-1}\|d\mu=\int_{M\backslash\SD}\log\|(Df)^{-1}\|d\mu$
  thus, by the Ergodic Decomposition Theorem, we conclude
  that for $\mu$-a.e $x\in M\backslash\SD$ we have
  $h_{\mu_{x}}(f)+\int\psi d\mu_{x}=0$ (remember that
  $P_{top}(f,\psi)=0$) and
  $\int\log\|(Df)^{-1}\|d\mu_x<0$.

  Therefore $\mu$-a.e $x\in M\backslash\SD$ has expanding
  ergodic components which are $\psi$-equilibrium states. By
  Proposition
  \ref{corolario:medida-ergodica-expansora-com-pressao-nao-negativa-eh-SRB-like},
  we deduce that $\mu_{x}$ is a weak-SRB-like probability
  measure for $\mu$-a.e $x\in M\backslash\SD$. The proof is
  complete.
\end{proof}

Putting Propositions~\ref{corolario:medidas SRB-like sao
  estados de equilibrio p/ WE} and~\ref{corolario2} together
we complete the proof of Corollary~\ref{mthm:WE}.

\subsection{Expanding Case}
\label{sec:expanding-case}

Corollary \ref{mthm:Expansora} improves the main result of
\cite{CatsEnrich2012} and allows rewriting all the results
from \cite{CatsEnrich2012}, which were only proved for
$C^1$-expanding maps in circle. In this section we prove
Corollaries \ref{mthm:Expansora} and
\ref{mthm:estabilidade-estatistica}.

\begin{proof}[Proof of Corollary \ref{mthm:Expansora}]
  The assumptions on $f$ ensure that all $f$-invariant
  probability measures $\mu$ are expanding. Moreover, by
  Proposition
  \ref{teorema:as-medidas-SRB-like-tem-pressao-nao-negativa}
  and Ruelle's Inequality, we conclude that
  $P_{top}(f,\psi)=0$ and all weak-SRB-like probability
  measures are $\psi$-equilibrium states satisfying Pesin's
  Entropy Formula.

  Then, on the one hand, if $\mu\in\SM_f$ satisfies Pesin's
  Entropy Formula, then $h_{\mu}(f)+\int\psi d\mu=0$. By the
  Ergodic Decomposition Theorem, we have that
  $h_{\mu_{x}}(f)+\int\psi d\mu_{x}=0$ for $\mu$-a.e
  $x\in M$, because $P_{\textrm{top}}(f,\psi)=0$. By
  Proposition
  \ref{corolario:medida-ergodica-expansora-com-pressao-nao-negativa-eh-SRB-like},
  we have that $\mu_{x}$ is a weak-SRB-like probability
  measure for $\mu$-a.e $x\in M$.

  On the other hand, if $\mu\in\SM_f$ is such that its
  ergodic components $\mu_x$ are weak-SRB-like probability
  measures for $\mu$-a.e. $x\in M$, then
  $h_{\mu_x}(f)+\int\psi d\mu_x=0$ for $\mu$-a.e $x\in
  M$. Thus, by the Ergodic Decomposition Theorem we have
  that
$$-\int\psi d\mu=-\int\left(\int\psi
  d\mu_x\right)d\mu(x)=\int\left(
  h_{\mu_x}(f)\right)d\mu(x)=h_{\mu}(f).$$

Assume now that $\mu$ is the unique weak-SRB-like
probability measure. By \cite[item G of Theorem
1]{catsigeras2016weak} we have that $\mu$ is SRB and
$\Leb(B(\mu))=1$.

This finishes the proof of Corollary \ref{mthm:Expansora},
except for the large deviations result which will be
obtained together with the proof of
Corollary~\ref{mthm:pressao-nao-negativa-implica-caracterizacao-dos-estados-de-equilibrio}
in Subsection \ref{sec:proof-coroll-refmthm}.
\end{proof}

As observed in \cite{CatsEnrich2012}, the $SRB$-like
condition is a sufficient but not necessary condition for a
measure $\mu$ to be an equilibrium state for the potential
$\psi$, because it may exist a non-ergodic invariant measure
$\mu\ll\Leb$ that is neither $SRB$ nor $SRB$-like (see
\cite{ANQuas}). In such a case $\mu$ satisfies Pesin’s
Entropy Formula, as follows.

\begin{proposition}\label{cor5}
  Let $f:M \rightarrow M$ be a $C^1$-expanding map. Let
  $\mu$ be a non-ergodic $f$-invariant probability such that
  $\mu\ll\Leb$. Then $\mu$ satisfies Pesin’s Entropy Formula.
\end{proposition}

\begin{proof}
  See \cite[Corollary 2.6]{CatsEnrich2012}.
\end{proof}

Now we are ready to prove Corollary
\ref{mthm:estabilidade-estatistica}.

\begin{proof}[Proof of Corollary
  \ref{mthm:estabilidade-estatistica}]
  Let $f_n\xrightarrow[n\rightarrow+\infty]{}f$ in the
  $C^1$-topology, where $f_n,f:M\to M$ are $C^1$-expanding
  maps for all $n\geq1$. For each $n\geq1$ consider $\mu_n$
  weak-SRB-like measures associated to $f_n$ and let $\mu$
  be a weak$^*$ limit point:
  $\mu=\lim_{j\rightarrow+\infty}\mu_{n_j}$.

  Fix $\SP$ a generating partition for every $f_{n_j}$ for
  all $j\geq1$ and such that $\mu(\partial\SP) = 0$.  This
  is possible, since $f_{n_j}$ is $C^1$-expanding and
  $f_n\xrightarrow[]{}f$ in $C^1$-topology and $f$ is also
  $C^1$-expanding.

  By the Kolmogorov-Sinai Theorem, this implies that
  $h_{\mu_n}(f_n)=h(\SP,\mu_n)$ and $h_{\mu}(f)=h(\SP,\mu)$,
  that is,
$$h_{\mu_{n_j}}(f_{n_j})=\inf\limits_{k\geq1}\frac{1}{k}H(\SP_{n_j}^k,\mu_{n_j}) \ \ \textrm{and} \ \ h_{\mu}(f)=\inf\limits_{k\geq1}\frac{1}{k}H(\SP^k,\mu)$$
Since $\mu$ gives zero measure to the boundary of $\SP$ then
$H(\SP_{n_j}^k,\mu_{n_j})$ converge to $H(\SP^k,\mu)$ as
$j\rightarrow\infty$.  Furthermore, for every
$\varepsilon>0$ there is $n_0\geq 1$ such that
$$
h_{\mu_{n_j}}(f_{n_j})\leq\frac{1}{n_0}H(\SP_{n_j}^{n_0},\mu_{n_j})\leq \frac{1}{n_0}H(\SP^{n_0},\mu)+\varepsilon\leq h_{\mu}(f)+2\varepsilon.
$$
By Corollary \ref{mthm:Expansora},
$h_{\mu_{n_j}}(f_{n_j})+\int\psi_{n_j} d\mu_{n_j}=0$ for all
$j\geq 0$, since $\mu_{n_j}$ is a weak-SRB-like probability
measure and $\psi_{n_j}=-\log|\det Df_{n_j}|$.

Since $\psi_{n_j}\rightarrow \psi$ in the topology of
uniform converge, we have that
$\int\psi_{n_j} d\mu_{n_j}\rightarrow\int\psi d\mu$. By
Ruelle's inequality, $h_{\nu}(f)+\int\psi d\nu\leq0$ for any
$f$-invariant probability measure $\nu$ on the Borel
$\sigma$-algebra of $M$. Thus,
$$0\geq h_{\mu}(f)+\int\psi d\mu\geq\limsup_{n\rightarrow+\infty}\left(h_{\mu_n}(f_n)+\int\psi_n d\mu_n\right)=0.$$
This shows that $\mu$ satisfies Pesin's Entropy Formula, is
a $\psi$-equilibrium state since
$P_{\textrm{top}}(f,\psi)=0$ and, by Corollary
\ref{mthm:Expansora}, its ergodic components $\mu_x$ are
weak-SRB-like probability measures for $\mu$-a.e $x\in M$.
\end{proof}

\subsection{Proof of Corollary
  \ref{mthm:pressao-nao-negativa-implica-caracterizacao-dos-estados-de-equilibrio}}
\label{sec:proof-coroll-refmthm}

We are now ready to prove Corollary
\ref{mthm:pressao-nao-negativa-implica-caracterizacao-dos-estados-de-equilibrio}.

\begin{proof}[Proof of Corollary
  \ref{mthm:pressao-nao-negativa-implica-caracterizacao-dos-estados-de-equilibrio}]
  From Remarks~\ref{rmkX} and~\ref{rmkY}, we can use
  Proposition
  \ref{teorema:as-medidas-SRB-like-tem-pressao-nao-negativa},
  Theorem~\ref{teo*} and Proposition
  \ref{corolario:medida-ergodica-expansora-com-pressao-nao-negativa-eh-SRB-like}
  in the setting of Theorem
  \ref{mthm:Toda-medida-nu-SRB-like-e-estado-de-equilibrio}. Let
  $\mu\in\SM_T$ be such that
\begin{eqnarray*}
  h_{\mu}(T)+\int\phi\, d\mu
  =
  P_{\textrm{top}}(T,\phi)
  =
  \int\left(h_{\mu_x}(T)+\int\phi\, d\mu_x\right)\,d\mu(x).
\end{eqnarray*}
Since we also have
$h_{\mu_x}(T)+\int\phi\,d\mu_x\leq P_{\textrm{top}}(T,\phi)$
then $h_{\mu_x}(T)+\int\phi\,d\mu_x=P_{\textrm{top}}(T,\phi)$
for $\mu$-a.e $x$. Now from Theorem \ref{teo*} (see Remark
\ref{rmkZ}) we get
$$\lim_{\epsilon\to0^+}\limsup_{n\to+\infty}
\frac{1}{n}\log\nu(A_{\epsilon,n}(\mu_x))
\geq h_{\mu_x}(T)+\int\phi\,d\mu_x
-\log\lambda%=2P_{\textrm{top}}(T,\phi)\geq0
$$
and then we conclude that $\mu_x\in\SW_T^*(\nu)$ following
the same arguments in the proof of Proposition
\ref{corolario:medida-ergodica-expansora-com-pressao-nao-negativa-eh-SRB-like}.

Assume now that $\mu$ is the unique $\phi$-equilibrium
state. By Proposition \ref{proposition:Wf-compacto} and
Theorem
\ref{mthm:Toda-medida-nu-SRB-like-e-estado-de-equilibrio}, we
conclude that there exist a unique $\nu$-SRB-like measure
$\mu$ and, by \cite[Theorem 1.6]{CatsEnrich2011}, it follows
that $\mu$ is $\nu$-SRB and $\nu(B(\mu))=1$.

Let now $\SV$ be a small neighborhood of $\mu$ in
$\SM$. Since $\{\SK_r(\phi)\}_r$ is decreasing with $r$ and
$\{\mu\}=\SK_0(\phi)=\cap_{r>0}\SK_r(\phi)$, we have that
there exists $r_0>0$ such that $\SK_r(\phi)\subset\SV$ for
all $0<r<r_0$. Since $\mathcal{K}_r(\phi)$ is weak$^*$
compact (by upper semicontinuity of the metric entropy) we
have
$\mathcal{K}_r(\phi)=\bigcap_{\varepsilon>0}\mathcal{K}^{\varepsilon}_r(\phi)$,
where
$\mathcal{K}^{\varepsilon}_r(\phi)=\big\{\mu\in\mathcal{M}_T:
\ \dist(\mu,\mathcal{K}_r(\phi))\leq \varepsilon\big\}$ with
the weak$^*$ distance defined in \eqref{eq13}. Let us take
$0<\varepsilon<r_0$ such that
$\mathcal{K}^{\varepsilon}_r(\phi)\subset\SV$. By
Proposition \ref{Lm:Large deviation lemma for distance
  expanding}, there exists $n_0\geq1$ and
$\kappa=\kappa(\varepsilon,r)>0$ such that
\begin{eqnarray*}
  \nu(\{x\in X: 
  \sigma_n(x)\in\SM\backslash\SV\})
  &\leq&
         \nu(\{x\in X:
         \sigma_n(x)\in\SM\backslash\SK_r^{\varepsilon}(\phi)\})
  \\
  &=&
      \nu(\{x\in X:
      \dist(\sigma_n(x),\SK_r(\phi))\geq\varepsilon\})
      < \kappa e^{n(\varepsilon-r)},
\end{eqnarray*}
for all $n\geq n_0$. Thus
$\limsup_{n\rightarrow+\infty}\frac{1}{n}\log\nu(\{x\in X: \
\sigma_n(x)\in\SM\backslash\SV\})< \varepsilon-r$. As
$\varepsilon>0$ can be taken arbitrary small, we conclude
that
$\limsup_{n\rightarrow+\infty}\frac{1}{n}\log\nu(\{x\in X: \
\sigma_n(x)\in\SM\backslash\SV\})<-r$ for all $0<r<r_0$
and $r_0=I(\SV):=\sup\{r>0; \
\SK_r(\phi)\subset\SV\}$. Therefore
$$\limsup\limits_{n\rightarrow+\infty}\frac{1}{n}\log\nu(\{x\in X: \ \sigma_n(x)\in\SM\backslash\SV\})\leq-I(\SV).$$

This shows that
$\nu(\{x\in X: \ \sigma_n(x)\in\SM\backslash\SV\})$
decreases exponentially fast with $n$ at a rate that depends
on the ``size'' of $\SV$.

The assumptions on $\mu$ are the same as in Corollary
\ref{mthm:Expansora} with $\nu = \Leb$ and $\phi =\psi$, so
the upper large deviations statement of this corollary
follows with the same proof.
\end{proof}

%%%%%%%%%%%%%%%%%%%%%%%%%%%%%%%%%%%%%%%%%%%%%%%%%%%%%%%%%%%%%%%%%%

 \def\cprime{$'$}

 \bibliographystyle{abbrv}
% \bibliography{bibliography}

\end{document}